\documentclass{amsart}
\usepackage{amsfonts,amscd,amssymb,amsmath,amsthm,mathdots,upgreek}
\usepackage{stmaryrd}
\SetSymbolFont{stmry}{bold}{U}{stmry}{m}{n}
\usepackage{graphicx,caption,subcaption,mathrsfs,appendix}
\usepackage{color}
\usepackage{enumerate}
\usepackage{bm}
\usepackage[colorlinks=true,linkcolor=blue]{hyperref}

\numberwithin{equation}{section}

\usepackage{xparse}
\definecolor{airforceblue}{rgb}{0.36, 0.54, 0.66}
\definecolor{amethyst}{rgb}{0.6, 0.4, 0.8}
\definecolor{applegreen}{rgb}{0.55, 0.71, 0.0}
\def\corEE{\textcolor{amethyst}}

\def\corO{}
\def\corE{}
\def\corA{}
\def\corEE{}
\def\corAB{}
\def\corABa{}
\def\corOZ{}

\def\corABrev{\textcolor{black}}

\definecolor{purple}{rgb}{0.9,0,0.8}

\newcommand{\abbr}[1]{{\sc\lowercase{#1}}}
\begin{document}

\newtheorem{theorem}{Theorem}
\newtheorem{proposition}[theorem]{Proposition}
\newtheorem{conjecture}[theorem]{Conjecture}
\newtheorem{corollary}[theorem]{Corollary}
\newtheorem{lemma}[theorem]{Lemma}
\theoremstyle{definition}
\newtheorem{dfn}{Definition}
\newtheorem{assumption}[theorem]{Assumption}
\newtheorem{claim}[theorem]{Claim}
\newtheorem{remark}[theorem]{Remark}

\numberwithin{theorem}{section}
\numberwithin{dfn}{section}


\newcommand{\R}{\mathbb{R}}
\newcommand{\T}{\mathcal{T}}
\newcommand{\C}{\mathbb{C}}
\newcommand{\D}{\mathbb{D}}
\newcommand{\G}{\mathcal{G}}
\newcommand{\Z}{\mathbb{Z}}
\newcommand{\Q}{\mathbb{Q}}
\newcommand{\E}{\mathbb E}
\renewcommand{\P}{\mathbb P}
\newcommand{\N}{\mathbb N}
\newcommand{\gd}{\mathfrak{d}}
\newcommand{\gb}{\mathfrak{b}}
\newcommand{\gL}{\mathfrak{L}}
\newcommand{\vep}{\varepsilon}
\newcommand{\cS}{\mathcal{S}}
\newcommand{\cN}{\mathcal{B}}
\newcommand{\supp}{{\mbox{\rm Supp }}}

\newcommand{\barray}{\begin{eqnarray*}}
\newcommand{\earray}{\end{eqnarray*}}

\newcommand{\dvec}[1]{ \llbracket #1 \rrbracket}

\newcommand{\Def}{:=}


\DeclareDocumentCommand \Pr { o }
{%
\IfNoValueTF {#1}
{\operatorname{Pr}  }
{\operatorname{Pr}\left[ {#1} \right] }%
}
\newcommand{\Prob}{\Pr}
\newcommand{\Exp}{\mathbb{E}}
\newcommand{\expect}{\mathbb{E}}
\newcommand{\1}{\one}
\newcommand{\Pto}{\overset{\mathbb{P}}{\to} }
\newcommand{\weakto}{\Rightarrow}
\newcommand{\lawequals}{\overset{\mathcal{L}}{=}}
\newcommand{\prob}{\Pr}
\newcommand{\pr}{\Pr}
\newcommand{\filt}{\mathscr{F}}
\newcommand{\ohadI}{\mathbbm{1}}
\DeclareDocumentCommand \one { o }
{%
\IfNoValueTF {#1}
{\ohadI }
{\ohadI\left\{ {#1} \right\} }%
}
\newcommand{\sgn}{\operatorname{sgn}}
\newcommand{\Bernoulli}{\operatorname{Bernoulli}}
\newcommand{\Binomial}{\operatorname{Binom}}
\newcommand{\Binom}{\Binomial}
\newcommand{\Poisson}{\operatorname{Poisson}}
\newcommand{\Exponential}{\operatorname{Exp}}

\newcommand{\Var}{\operatorname{Var}}
\newcommand{\Cov}{\operatorname{Cov}}


\newcommand{\Id}{\operatorname{Id}}
\newcommand{\diag}{\operatorname{diag}}
\newcommand{\tr}{\operatorname{tr}}
\newcommand{\proj}{\operatorname{proj}}
\newcommand{\Span}{\operatorname{span}}


\newcommand{\nicefrac}{\frac}
\newcommand{\half}{\frac12}
\DeclareDocumentCommand \JB { O{n} O{\lambda} } {J_{{#1}}({#2})}

\DeclareDocumentCommand \LP { O{\ESD} } {U_{ {#1} }}
\newcommand{\LPL}{ \LP[{\mu_N}] }

\newcommand{\ESD}{ L_N^{\model} }

\newcommand{\model}{\mathcal{M}}
\newcommand{\SVD}{\Sigma}
\newcommand{\LL}{\mathcal{L}}
\newcommand{\PI}{\Pi}
\DeclareDocumentCommand \PG { O{n} }
{
\mathfrak{S}_{{ #1 }}
}
\newcommand{\RS}{\mathcal{C}}

\newcommand{\TODO}[1]{ {\bf TODO: #1} }

\newcommand{\row}{X}
\newcommand{\col}{Y}
\newcommand{\srow}{x}
\newcommand{\scol}{y}
\newcommand{\csrow}{w}
\newcommand{\cscol}{z}
\newcommand{\COMP}[1]{ \check{#1} }

\title[Regularization of non-normal matrices] {Regularization of non-normal matrices by Gaussian noise - the banded Toeplitz
and twisted Toeplitz cases}
\author[A.\ Basak]{Anirban Basak$^*$}\thanks{${}^*$Partially supported by
 grant 147/15 from the Israel Science Foundation, a US-Israel BSF grant, and ICTS--Infosys Excellence Grant}
 \address{$^*$International Centre for Theoretical Sciences
 \newline\indent Tata Institute of Fundamental Research
 \newline\indent Bangalore 560089, India
 \newline\indent and
 \newline\indent Department of Mathematics, Weizmann Institute of Science
 \newline\indent POB 26, Rehovot 76100, Israel}
 \author[E.\ Paquette]{Elliot Paquette$^\ddagger$}
 \address{$^\ddagger$Department of Mathematics, The Ohio State University
 \newline\indent Tower 100, 231 W 18th Ave, Columbus, Ohio 43210, USA}
\author[O.\ Zeitouni]{Ofer Zeitouni$^{\mathsection}$}\thanks{${}^{\mathsection}$Partially supported by
 grant 147/15 from the Israel Science Foundation and by a US-Israel BSF grant}
 \address{$^{\mathsection}$Department of Mathematics, Weizmann Institute of Science
 \newline\indent POB 26, Rehovot 76100, Israel
 \newline \indent and
 \newline\indent Courant Institute, New York University
 \newline \indent 251 Mercer St, New York, NY 10012, USA}


\begin{abstract}
  We consider the spectrum of
  additive, polynomially vanishing  random
  perturbations of deterministic matrices, as follows.
  Let $M_N$ be a deterministic $N\times N$ matrix, and
  let $G_N$ be \corEE{a complex Ginibre} matrix. We consider
  the matrix $\model_N=M_N+N^{-\gamma}G_N$, where $\gamma>1/2$. With $L_N$
  the empirical measure of eigenvalues of $\model_N$, we
  provide a general deterministic equivalence theorem that ties $L_N$ to
  the singular values of $z-M_N$, with $z\in \C$. We then compute the limit
  of $L_N$ when $M_N$ is an upper triangular
  Toeplitz matrix of finite symbol: if
  $M_N=\sum_{i=0}^{\corAB{\mathfrak{d}}} a_i J^i$ where $\corAB{\mathfrak{d}}$ is fixed, $a_i\in \C$ are deterministic scalars and $J$ is the nilpotent
  matrix $J(i,j)={\bf 1}_{j=i+1}$, then
  $L_N$ converges, as $N\to\infty$, to the law of $\sum_{i=0}^{\corAB{\mathfrak{d}}} a_i U^i$
  where $U$ is a uniform random variable \corAB{on the unit circle in the complex plane}.
  We also consider the case of slowly varying diagonals (twisted
  Toeplitz matrices), and, when $\corAB{\mathfrak{d}}=1$,
  also
  of i.i.d.~entries on the diagonals in $M_N$.
\end{abstract}

\maketitle

\section{Introduction}
Write $G_{N}$ for an $N \times N$ random matrix whose
entries are i.i.d.~standard \emph{complex} Gaussian variables
(a \textit{Ginibre} matrix), and
let $\{M_N\}_{N=1}^{\infty}$ be a sequence of deterministic
$N \times N$ matrices.
Consider a noisy counterpart
given by
\begin{equation}
  \label{eq-1}
  \model_N:=M_N+N^{-\gamma}G_N,
\end{equation}
where $\gamma\in (1/2,\infty)$ is fixed, noting that
by standard estimates, see \cite[Corollary 1.2]{gordon},
\begin{equation}
  \label{eq-gordon}
  \|N^{-\gamma} G_N\|\to_{N\to\infty} 0\quad  \mbox{\rm a.s.},
\end{equation}
where
$\|\cdot \|$ denotes the operator norm.
Let $\lambda_i, i=1,\ldots,N$
denote the eigenvalues of $\model_N$, and let
\begin{equation}
  \label{eq-defln}
  L_N:=N^{-1}\sum_{i=1}^N \delta_{\lambda_i}
\end{equation}
denote the associated
empirical measure.
In this paper, we study the convergence of $L_N$ for a class of matrices $M_N$. Discussions of background and related approaches are deferred to subsections 
\ref{subsec-1.3} and \ref{subsec-1.4}.

\subsection{Main results}
 For a
probability measure $\mu$ on $\mathbb{C}$ which integrates the
log function at infinity, and $z\in \mathbb{C}$,
denote the \textit{Logarithmic potential} associated with $\mu$ by
\begin{equation}
  \label{eq-logpot}
  \LL_\mu(z):=\int \log|z-y|
\mu(dy).
\end{equation}
The importance of the logarithmic potentials lies in the
fact that the pointwise convergence of $\LL_{\mu_n}(z)$ to a limit $\LL_\mu(z)$
implies the
weak convergence $\mu_n\to\mu$.

Our first main result is a deterministic equivalence theorem for
$\LL_{L_N}(z)$.
We formulate here a simplified version under more stringent
conditions than necessary, and refer to Theorem  \ref{thm:logdet}
for the general statement,
which also has an explicit description
of the functions $g_N$ appearing in the statement of Theorem \ref{thm:logdet-simp}. \corEE{Let $\Id = \Id_N$ stand for
the identity matrix of dimension $N$.
Let $J=J_N$ denote the nilpotent matrix with $J_{ij}={\bf 1}_{j=i+1}$
for $1 \leq i < j \leq N$.}
\begin{theorem}
  \label{thm:logdet-simp}
 Fix $\gamma>1/2$. \corABrev{Fix $z \in \C$ and
 let} $\corAB{s_N(z)}$ denote the number of singular
 values of $M_N-z\Id$ smaller than $N^{-\gamma+1/2+\delta_N}$, where
 $0<\delta_N\to_{N\to\infty} 0$. 
 Suppose
 $\corAB{s_N(z)}\log N/ N\to_{N\to \infty} 0$.
 Then, there exist explicit, deterministic functions $g_N(z)$ so that
 \[
 |\LL_{L_N}(z)- g_N(z)|\to_{N\to\infty} 0, \qquad \text{ \corABrev{in probability}}.\]
 \end{theorem}
 The importance of Theorem \ref{thm:logdet-simp} (and its more elaborate
 version, Theorem \ref{thm:logdet}) lies in the fact that it reduces
 the question of weak convergence of the random empirical measure $L_N$
 to computations involving the deterministic matrices $M_N$. Still,
 these computations are, in general, non-trivial. The other results in
 this paper are instances in which these computations can be carried through
 and the limit of $L_N$ can be described explicitly.

Our second main result deals with upper triangular Toeplitz matrices
of finite symbol, that is banded upper triangular Toeplitz matrices.
\begin{theorem}
  \label{theo-1}
  Let $a_i, i=0,1,\ldots,\gd$ be complex (deterministic) numbers.  Let
  $M_N:=\sum_{i=0}^\gd a_i J^i$, and let $\model_N$ and $L_N$
  be as in \eqref{eq-1} and \eqref{eq-defln}.
  Then $L_N$ converges weakly in probability to the law of $\sum_{i=0}^\gd
  a_i U^i$, where $U$ is uniformly distributed on the unit circle.
\end{theorem}


\begin{figure}[t]
\begin{center}
  \includegraphics{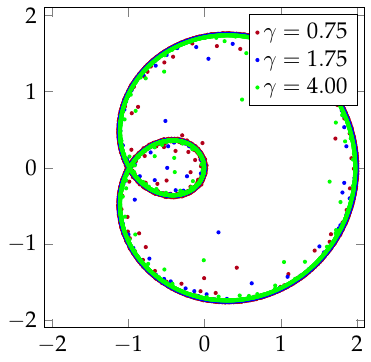}
  \includegraphics{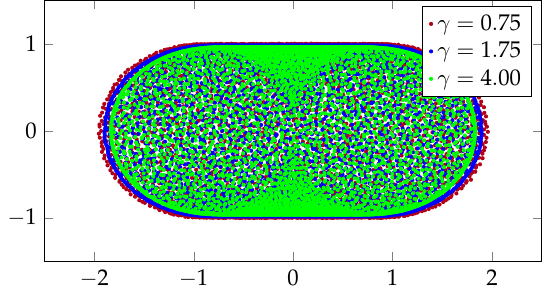}
  \includegraphics{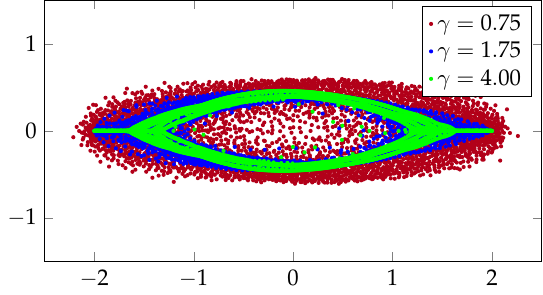}
\caption{
 The eigenvalues of $\model_N$, with $N=4000$ and various $\gamma$. On the top
 left,
 $M_N=J+J^2$. On the top right, $M_N=D_N+J$ with $D_N(i,i)=-1+2i/N$. On the
 bottom, $M_N=D_N+J$ with
 $D_N$ i.i.d.  uniform on $[-2,2]$.
}
\label{fig:1}
\end{center}
\end{figure}
\noindent
A generalization of Theorem \ref{theo-1} to the twisted Toeplitz setup
appears in Section \ref{sec:tt}, see Theorem \ref{thm:finite-off-diag} there. As the next theorem shows,
in the case of two diagonals in $M_N$ more can be said.
\corOZ{For $x_1, x_2 \in \R$ we denote $x_1 \vee x_2:=\max\{x_1,x_2\}$.}
\begin{theorem}\label{theo-2} Let $D_N$ be a diagonal matrix with
  entries $d_i$, set $M_N=D_N+J$ and let $\model_N$ be as in \eqref{eq-1}.\\
  a) Let $d_i$ be i.i.d.~random variables of law $\nu$ supported on a subset of a simply connected compact set with Lebesgue area $0.$
  Then $L_N$ converges weakly in probability to a measure $\mu$ characterized
  by $\LL_\mu(z)=(\E_\nu \log |z-d_1|)\vee 0$.\\
  b) Let $f:[0,1]\to \mathbb{C}$ be H\"{o}lder-continuous
  and set $d_i=f(i/n)$. Then
  $L_N$ converges weakly in probability to a probability measure $\mu$ satisfying
  $$ \mu = \int_0^1 \mbox{\rm unif}_{f(z),1} dz,$$
  where $\mbox{\rm unif}_{a,b}$ denotes the uniform law on a circle \corAB{in the complex plane} of radius $b$ and center $a$.
\end{theorem}
\noindent
See Section \ref{sec-d=1} for details and further examples, and note
that Theorem \ref{theo-2} a) is Corollary \ref{cor:iid}, while
Theorem \ref{theo-2} b) is Corollary \ref{cor:holder}.

An illustration of Theorems \ref{theo-1} and
\ref{theo-2} is provided in Figure \ref{fig:1}.

\begin{remark}
We chose to consider  throughout the paper only the case of perturbation
matrices $G_N$ which are complex Ginibre  matrices.
We believe that the results should carry over in a rather
straightforward way
to the case of real Ginibre
matrices, and with a significant effort to the i.i.d.~setup,
\corOZ{in the same spirit as \cite{W}.} To avoid additional
technicalities, we did not pursue these extensions here.
\end{remark}
\begin{figure}[t]
\begin{center}
  \includegraphics{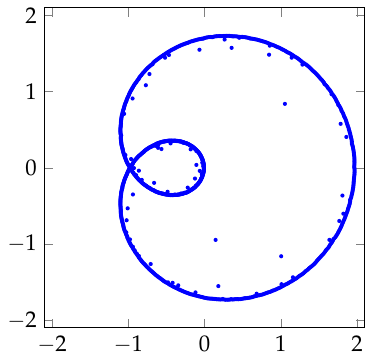}
  \includegraphics{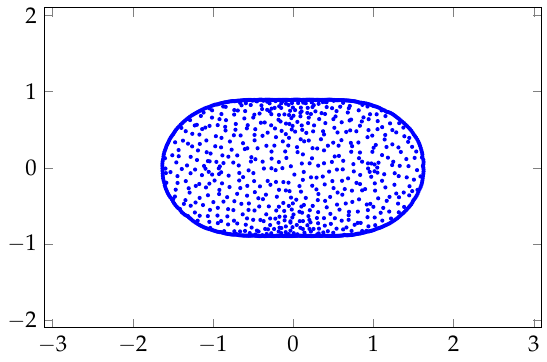}
  \includegraphics{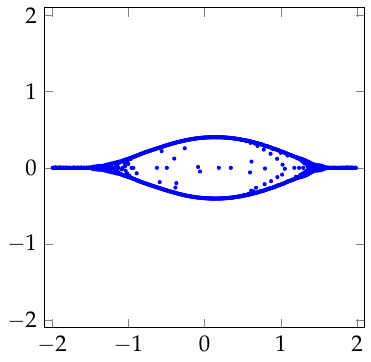}
\caption{
 Numerical evaluation of eigenvalues of $U_N M_N U_N^*$ with
 $U_N$ Haar-distributed unitary and $N=1000$.
  On the top left,
 $M_N=J+J^2$. On the top right, $M_N=D_N+J$ with $D_N(i,i)=-1+2i/N$. On the
 bottom, $M_N=D_N+J$ with
 $D_N$ i.i.d.  uniform on $[-2,2]$, $N=4000$.
The computation was performed in Scipy with
a 64-bit floating point precision.
Note the similarity with Figure \ref{fig:1}.
}
\label{fig:2}
\end{center}
\end{figure}

\subsection{\corEE{A Thouless--type formula}}

Both Theorems \ref{theo-1} and \corABa{\ref{theo-2}(a)} can also be formulated in terms of Lyapunov exponents.  Consider the vector space
\[
V:= \{ x \in \R^N : ((M_N - z\Id)x)_j = 0, 1 \leq j \leq N-\gd \}.
\]
This space is $\gd$--dimensional.  Further, to find $x \in V,$ having chosen $x_1,x_2,\dots, x_{\gd},$ one can solve for the remaining entries of $x$
using the equations $(M_N - z\Id)x)_j = 0$ for $1 \leq j \leq N-\gd$ to propagate the solution.  Concisely, we  can find $\gd \times \gd$ \emph{transfer} matrices $(T_j(z))_{1}^{N-\gd}$
so that for all $1 \leq j \leq N-\gd$
\[
(x_\ell)_{\ell = j+1}^{j+\gd} = T_j(z) \cdot (x_\ell)_{\ell = j}^{j+\gd-1}.
\]
For details, see Definition \ref{dfn:transfer-matrix}; for the exposition here, the explicit form of this matrix will not be necessary.

In the setup of Theorem \ref{theo-1}, these matrices will all be identical.  In the setup of Theorem \ref{theo-2} part (a), the matrices will be i.i.d.\ scalars. In either case, the sequence  $(T_j(z))_{1}^{N-\gd}$ is stationary, and we can consider the Lyapunov spectra as the set of values
\[
\{\mu_1(z), \mu_2(z), \ldots , \mu_\gd(z)\} := \left\{ \lim_{n \to \infty} \frac{1}{n} \log \| T_n(z) T_{n-1}(z) \cdots T_1(z) v\| : v \in \C^\gd\right\}
\]

In the setup of Theorem \ref{theo-2} part (a), there is a single Lyapunov eigenvalue, given by $\mu_1(z) = \Exp \log |d_1 - z|.$  This allows us to write that
\[
  \LL_{L_N}(z) \to \corOZ{\mu_1(z)} \vee 0.
\]
In the setup of Theorem \ref{theo-1}, if we set $P(x) = \sum_{i=0}^{\gd} a_i x^i$ we \corABa{have that} 
\[
\LL_{L_N}(z) \to \int_0^{2\pi} \log |P(e^{i\theta})-z|\frac{d\theta}{2\pi}.
\]
On the other hand, factorizing $P(x) - z = a_{\gd} \prod_{i=1}^\gd (x-\lambda_i(z)),$ we can write
\[
\int_0^{2\pi} \log |P(e^{i\theta})-z|\frac{d\theta}{2\pi}
= \sum_{i=1}^\gd  ((\log |\lambda_i(z)|) \vee 0)+\log |a_\gd|.
\]
Furthermore, in the Toeplitz case, the eigenvalues of $T_j(z) = T_1(z)$ are just the roots of the symbol $P(x),$ and the Lyapunov spectra are nothing but $\log | \lambda_i(z)|$ for $1 \leq i \leq \gd.$  Hence, we have that in both Theorem \ref{theo-2} part (a) and Theorem \ref{theo-1},
\[
\LL_{L_N}(z) \to \sum_{i=1}^\gd ( \mu_i(z) \vee 0)+\log|a_\gd|.
\]

A \corABa{similar result had} appeared previously in the study of the Thouless formula for \corOZ{the strip \cite[Theorem 2.4]{CraigSimon}.}
For the twisted Toeplitz cases (such as in Theorem \ref{theo-2} part (b)), the formula must be replaced by an average over local Lyapunov exponents.

\subsection{Connection to pseudospectra}
\label{subsec-1.3}
\par
The fact that the spectrum of non--normal matrices and operators
is not stable with respect
to perturbations is well known, see e.g. \cite{trefethen} for a
comprehensive account and \cite{DH}
for a recent study.
To illustrate the issue,
we attach in Figure \ref{fig:2} an
actual simulation of $U_N M_N U_N^*$ where $M_N$ is as in Figure
\ref{fig:1} and $U_N$
is a random Haar-distributed
unitary matrix. While the spectrum of $M_N$ is real, the numerical
simulations produce errors that make the spectrum look similar to
the one for the noisy perturbed model $\model_N$, compare with
Figure \ref{fig:1}. See \cite{Trefethen1991} for early 
examples of the same phenomenon.

The $\epsilon$--\emph{pseudospectrum}, defined by
  \[
    \Lambda_\epsilon(M_N): = \left\{ z \in \C : \sigma_N(M_N-z\Id) \leq \epsilon \right\},
  \]
  with $\sigma_N(\cdot)$ the smallest singular value,
  is a type of worst--case quantification of the instability of the spectrum.
  See \cite{Trefethen1991} for the original formulation and
  \cite{trefethen} for an extensive background
and applications in numerical analysis and beyond.

  In the literature on pseudopsectra, an outsize importance is placed on \emph{exponentially good pseudospectra}, i.e.\ $\Lambda_\epsilon$ where $\epsilon < e^{-\delta N}$ for some $\delta.$  Of particular relevance here, simulations of randomly perturbed non--normal matrices suggest that their spectra concentrate on sets that strongly resemble exponentially good pseudospectral level lines,
  see e.g. \cite{TrefethenReichel}. In particular, in the upper-triangular
 Toeplitz case, these curves
 are precisely the image of the unit circle in the complex
 plane by the Toeplitz symbol.

  Furthermore, all of the models of non--normal matrices that we consider have been, not coincidentally, the subjects of study from the point of view of exponentially good pseudospectra.
  The work \cite{Trefethen1991} describes many examples of non--normal matrices and gives plots of pseudospectral level lines adjacent their perturbed eigenvalues.   The top two plates in Figures \ref{fig:1} and \ref{fig:2} are Examples 2 and 4 from \cite{Trefethen1991}.  Subsequent work \cite{TrefethenReichel} proved, using transfer matrices, some estimates for the locations of the $\epsilon$-pseudospectrum and exponentially good pseudospectrum 
of large Toeplitz matrices, and showed in the upper triangular case  that the latter converges to the spectrum of the limit Toeplitz operator, namely to the the image of the unit circle by the symbol; our Theorem 
\ref{theo-1}
 shows that indeed, for upper triangular symbols of finite support and under small Gaussian perturbations, the empirical measure converges to a limit with precisely this support.  The work \cite{TrefethenContediniEmbree}, motivated in part by the Hatano--Nelson model, considers the pseudospectra of random bidiagonal matrices, identifying four regions of distinct pseudospectral growth.   Finally \cite{TrefethenChapman} computes the exponentially good pseudospectrum of some classes of twisted Toeplitz matrices, including the top--right example of Figures \ref{fig:1} and \ref{fig:2} (the ``Wilkinson'' matrix). See also 
\cite{Trefcont} for related results in the continuous setup.

As we shall see in Section~\ref{sec:det-equiv}, adding small
Gaussian noise to $M_N - z\Id$ roughly has the effect of boosting any exponentially small singular values to the order of unity.  Hence in situations in which there are only a few singular values of $M_N -z\Id$ that are exponentially small, the log potential at $M_N -z\Id + N^{-\gamma}G_N$ can be approximated by computing the log potential of $M_N -z\Id$ and subtracting from it the contribution of exponentially small singular values.  Indeed, if the exponential growth rates of these extremal singular values are harmonic as functions of $z$ away from the spectra, then a discontinuity in the Laplacian of the log--potential occurs exactly where the exponential growth of extremal singular values changes signs.  In particular, an exponentially good pseudospectral level line would be contained in the limiting spectral support of $M_N - z\Id + N^{-\gamma}G_N.$ See Figure \ref{fig:3} for an illustration.
\begin{figure}[t]
\begin{center}
  \includegraphics[width=2.75in]{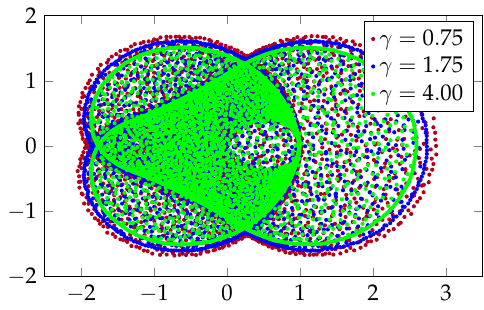}
  \includegraphics[width=3.25in]{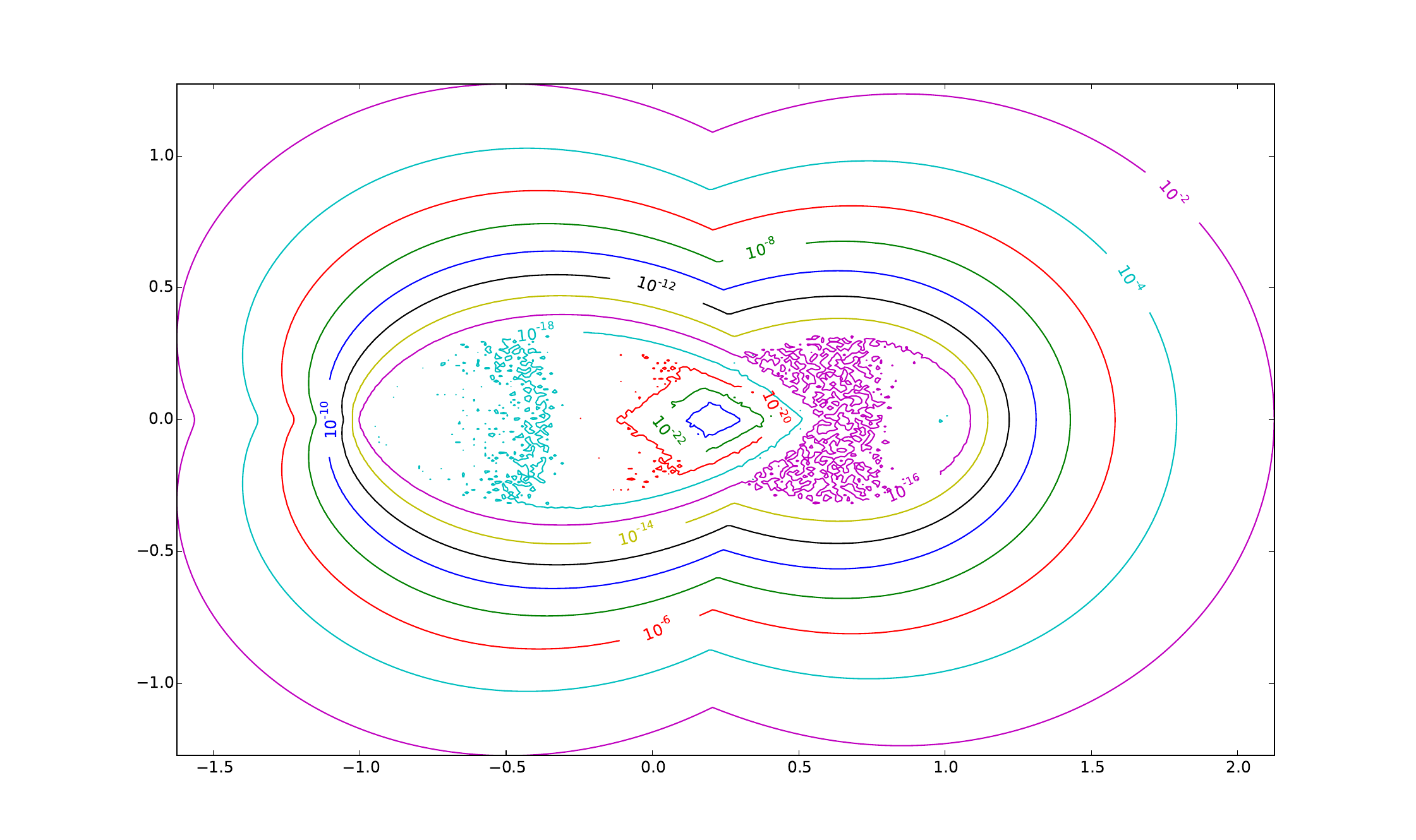}
\caption{
  The matrix $M_N = D_N - D_NJ + J^2,$ with $D_N(i,i)=-1+2i/N.$
  The leftmost image shows eigenvalues of $M_N + N^{-\gamma}G_N$ for $N=4000.$  Theorem \ref{thm:finite-off-diag} gives the distributional convergence of the spectra of this matrix on adding Gaussian noise.
  The rightmost image shows pseudospectral level lines for $N=100.$ The levels displayed are $100^{-k},$ for $k$ ranging from $1$ to $12.$
  Pseudospectral lines generated by Andr\'e Gaul's pseudopy package, based off Eigtool by Thomas G. Wright.
 }
\label{fig:3}
\end{center}
\end{figure}

However, as a consequence of the theorems we show in this paper, we see that pseudospectrum alone is in general
not sufficient for understanding the limiting spectral distribution of randomly perturbed non--normal matrices, except for special cases, e.g. when
only one singular value of $M_N - z\Id$ is exponentially small, or in the
Toeplitz upper-triangular case.  

\subsection{Previous results on typical perturbations and strategy}
\label{subsec-1.4}
Our work is an attempt
to address the same issue from a ``typical'' perturbation point of view, and
in this sense continues the line of research initiated in \cite{GWZ,W} and
\cite{FPZ}, which we now describe.

In \cite{GWZ}, the authors consider the case where $M_N$ converges in
$*$-moments, that is, there exists an operator $a$ in a non-commutative
probability space $(\mathcal{A}, \tr)$ so that for any non commutative
polynomial $P$, $N^{-1}\tr P(M_N, M_N^*)\to \tr P(a,a^*)$. Under a regularity
assumption on $a$ and the existence of polynomially vanishing perturbations
of $M_N$ with empirical measure converging to the spectral
measure $\mu_a$ associated with $a$,
they show that $L_N$ converges to $\mu_a$ in probability.
They further show that their assumptions are satisfied
when $M_N=J$. The paper \cite{W} shows that this result is stable in the sense
that replacing $G_N$ by a matrix with i.i.d. entries satisfying mild assumptions
does not change the result. The proofs in \cite{GWZ} and \cite{W} controls
the log potential of $\model_N$ by methods inspired by free probability, and in particular breaks down if the $*$-moment limit does not coincide
with the limit of $L_N$. Further, in \cite{GWZ} one can
find an example of a matrix $M_N$ (with only one non-zero diagonal)
where the latter is indeed the case - namely, $M_N(i,j)=J(i,j)\cdot
{\bf 1}_{j\neq 0 \mod \log N}$.

In \cite{FPZ}, the authors consider the latter situation and prove a limit
theorem, where the limit does depend on $\gamma$. The method of proof is very
different from \cite{GWZ,W} - it involves a combinatorial analysis of
$\det \model_N(z)$ where $\model_N(z):=\model_N-z\Id$. \corAB{Noting that $\LL_{L_N}(z)=
\frac1N \log |\det \model_N(z)|$, concentration
of measure arguments then identify the terms contributing to the determinant.}

Our approach in this paper is related to \cite{FPZ} in that we also compute
$|\det(\model_N(z))|$. However, our starting point is to relate the latter
to a truncation of $M_N(z)=M_N-z\Id$, where the lowest lying singular values
of $M_N(z)$ are eliminated (we refer to this as a ``deterministic
equivalent model'', using terminology borrowed from \cite{hachem}).
The level of truncation depends on $\gamma$, which parametrizes the
strength of the perturbation.
Once this step has been established, we can study the small singular values
of $M_N(z)$ using transfer-matrix techniques, in case
$M_N$ is Toeplitz with a finite symbol or a
slowly-varying version of such a matrix. This analysis was not present
in \cite{GWZ,FPZ,W} and seems to be new in the context of the stability that
we study.

We note that other approaches to the study of perturbations of non-normal
operators exist. In particular, Sj\"{o}strand and Vogel \cite{SV},
\cite{SV1}
identify the limit of the empirical value of a random perturbation of
a  banded Toeplitz matrix
with two non-zero diagonals, one above and one below the main
diagonal. Their methods, which are quite different
from ours, are limited to $\gamma>5/2$, and yield more quantitative
estimates on the empirical measure and its outliers.

The structure of the paper is as follows. In the next section, we introduce
and prove the deterministic equivalence. Various standard
algebraic
facts needed for the proofs are collected in Appendix \ref{sec:prelim}.
Section
\ref{sec-d=1} presents the analysis of the deterministic equivalent model in
the case where only two diagonals are present; the latter restriction simplifies
the analysis because transfer matrices reduce to a scalar in that case. Section \ref{sec:tt} treats the case of \corAB{$\gd>1$}, and reduces the twisted Toeplitz
case to
piecewise-constant twisted Toeplitz matrices,
described in Theorem \ref{thm:finite-off-diag-pc}.
The proof of the latter appears in
Section \ref{sec:pf-pc-constant}.\\[0.2em]

\subsection{Notation}
\corOZ{We use throughout the following standard notation.}
\corAB{For two real-valued sequences $\{a_n\}$ and $\{b_n\}$ we write
$a_n=o(b_n)$ if $\limsup_{n\to\infty} b_n/a_n=0$,
$a_n=O(b_n)$ if $\limsup_{n\to\infty} |b_n|/|a_n|<\infty$,
 and
 $a_n \gg b_n$ if $\limsup_{n \to \infty} |a_n|/|b_n| =\infty$.}
\corOZ{For any $x \in \R$, we denote $x_-:=-\min\{x, 0\}$ and $x_+:=\max\{x,0\}$.}
\corAB{For any $x >0$, we denote $\log_+(x):=(\log(x))_+$.}
\corOZ{For $x_1, x_2 \in \R$ we denote $x_1 \vee x_2:=\max\{x_1,x_2\}$ and
 $x_1 \wedge x_2=\min\{x_1,x_2\}$.}
 \corOZ{We use $e_j$ to denote the standard unit vector, all of whose entries are $0$ except for $1$ at the $j$-th entry.}
\corOZ{For two random variables $X_i, i=1,2$, we write $X_1 \lawequals X_2$ to denote that $X_1$ and $X_2$ have the same law.} \corABa{For any matrix $M$,
  denote by $\|M\|_2:=[{\rm Tr}(M M^*)]^{1/2}$ its Hilbert-Schmidt norm,
and for any vector $v$ denote by  $\|v\|_2$ its Euclidean norm}.

\subsection*{Acknowledgement}
We thank Nick Trefethen for many useful discussions and examples that
motivated this study, and in particular for urging us to consider the
twisted Toeplitz case. We also thank him for his comments on a previous version
of this article. \corABrev{We are grateful to the anonymous referee for thoughtful comments which led to the improvement of the presentation and clarity of this paper.}

\section{The deterministic equivalent}\label{sec:det-equiv}

Let \corA{$\SVD_N(M_N(z))$} be the diagonal matrix of singular values of \corA{$M_N(z):=M_N-z\Id$}. Suppose the entries of $\SVD_N(M_N(z))$ are arranged to be non-decreasing going down the diagonal. \corA{That is, $\Sigma_{ii}(z)=\sigma_{N-i+1}(M_N(z))$ for $i \in [N]:=\{1,2,\ldots,N\}$, where $\Sigma_{ii}(z)$ is the $i$-th diagonal entry of $\SVD_N(M_N(z))$, and $\sigma_i(M_N(z))$ is the $i$-th largest singular value of $M_N(z)$.} By invariance of the Gaussian matrix,
\[
  \det( z\Id - \model_N ) \lawequals \det( \corA{\SVD_N(M_N(z))} + N^{-\gamma}G_N).
\]
 Suppose that \corA{$\SVD_N(M_N(z))$} is decomposed into two blocks of sizes $N_1+N_2 = N,$ so that
\begin{equation}\label{eq:truncation-dfn}
 \corA{ \SVD_N(M_N(z)) := \begin{pmatrix}
    S_N(M_N(z)) & \\
    & B_N(M_N(z))
  \end{pmatrix},}
  \text{ and }
  N^{-\gamma}G_N := \begin{pmatrix}
    X_1 & X_2 \\
    X_3 & X_4 \\
  \end{pmatrix}.
\end{equation}
For ease of writing, when the matrix $M_N$ is clear from the context, we simply write $\SVD_N(z), S_N(z)$, and $B_N(z)$ instead of $\SVD_N(M_N(z)), S_N(M_N(z))$, and $B_N(M_N(z))$.  Now by the Schur complement formula,
\begin{equation}\label{eq:schur}
\det( \SVD_N(z) + N^{-\gamma}G_N)
\lawequals
\det(B_N(z) + X_4) \cdot \det(\widetilde{S}_N(z)),
\end{equation}
where
\begin{equation}\label{eq:tilde-schur}
\widetilde{S}_N(z):=S_N(z) + X_1 \corAB{-} X_2 ( B_N(z) +X_4 )^{-1} X_3.
\end{equation}
\corABrev{(Since the entries of $X_4$ are i.i.d.~Gaussian, the matrix $(B_N(z)+X_4)$ is  a.s.  invertible  and hence $\widetilde S_N(z)$ is well defined.)  The decomposition in \eqref{eq:schur}} proves useful when we choose the decomposition so that the entries of \corA{$B_N(z)$} are somewhat large with respect to the noise. For such a decomposition as we will see later (see Theorem \ref{thm:logdet}) the log determinant of $B_N(z)$ correctly characterizes the log-potential of the limiting spectral distribution of $\model_N$. So we need to define $B_N(z)$ appropriately.

Fix a sequence of $\{\varepsilon_N\}$ going down to zero.
We define $N^*:=N^*(z,\gamma,\varepsilon_N)$ as the largest integer $i$ so that
\begin{equation}\label{eq:N*-dfn}
  \Sigma_{ii}(z) < \corA{\varepsilon_N^{-1}} N^{-\gamma}(N-i\corAB{+1})^{1/2}.
\end{equation}
If no such $1 \leq i \leq N$ exists
then
we let $N^*=1$. Now set $N_1 = N^*$. This defines $B_N(z)$. With this choice of $B_N(z)$ we have the following result.

\begin{theorem} \label{thm:logdet}
\corABrev{Fix $z \in \C$.} Suppose that $N^* \log N / N \to \alpha < \infty.$  Then for any $\{\varepsilon_N\}$ that tends to $0$ slowly enough that $\log(\varepsilon_N^{-1})/\log N \to 0,$
  \begin{equation}\label{eq:log-det-conv}
    \frac{1}{N} \log| \det( \SVD_N(z) + N^{-\gamma}G_N)|
    -
    \frac{1}{N} \log| \det( B_N(z)|
    \to  -\alpha(\gamma - \tfrac{1}{2}),
  \end{equation}
in probability, as $N\to\infty$. If $\alpha = 0,$ we may take $\varepsilon_N = N^{-\eta}$ for any $\eta > 0.$
\end{theorem}


The proof of Theorem \ref{thm:logdet} requires a two-fold argument. First we show that the truncation level chosen above assures that $X_4$ is negligible with respect to $B_N(z)$. In particular, we obtain the following result.
\begin{lemma}\label{lem:det-B_N}
For any given sequence of $\{\varepsilon_N\}$, such that $\varepsilon_N \in (0,1)$ for all $N$, we have
\begin{equation}\label{eq:det-exp}
\E \det (B_N(z)+X_4)= \det(B_N(z)).
\end{equation}
and
  \begin{equation}
    \label{eq-2.7}
    \Var \det(B_N(z) + X_4) \leq \frac{\varepsilon_N^2}{1-\varepsilon_N^2}\left( \det(B_N(z))\right)^2.
  \end{equation}
  \label{lem:chebyshev}
\end{lemma}
\begin{proof}
  Since the entries of $X_4$ are independent with zero mean, \eqref{eq:det-exp}
  follows from \corAB{Lemma \ref{lem:cauchy-binet}}.
  To compute the second moment, we \corAB{again use Lemma \ref{lem:cauchy-binet}}
  and the fact that entries of $X_4$ are independent with zero mean and variance $N^{-2\gamma}$ to obtain
  \begin{align*}
    \Exp| \det(B_N(z) + X^4)|^2
  & \,  =  \sum_{k=0}^{N-N^*} \sum_{\substack{ S \subset [N]\setminus[N^*] \\ |S| = N-N^*-k}}
    \Exp |\det X^4[\COMP{S}]|^2 \cdot \left(
    \prod_{i \in S} \left|\Sigma_{ii}(z)\right|^2\right)\\
    & \, =
    \sum_{k=0}^{N-N^*}
    (k!)N^{-2\gamma k}
    \sum_{\substack{ S \subset [N]\setminus[N^*] \\ |S| = N-N^*-k}} \prod_{i \in S} \left|\Sigma_{ii}(z)\right|^2,
  \end{align*}
  \corAB{where $\COMP{S}:=([N]\setminus [N^*])\setminus S$}. As the diagonal entries of $\SVD_N(z)$ are arranged in non-decreasing order, recalling the definition of $B_N(z)$ and $N^*$ we find that
\[
\prod_{i \in S}| \Sigma_{ii}(z)|^2 \le \frac{|\det B_N(z)|^2}{\left( \varepsilon_N^{-1} N^{-\gamma}\right)^{2k} (N-N^*)^k},
\]
for any $S\subset [N]\setminus[N^*]$ with $|S| = N-N^*-k$. Therefore,
  \begin{align*}
    \Exp |\det(B_N(z) + X^4)|^2
   & \,  \leq
    \sum_{k=0}^{N-N^*}
    \binom{N-N^*}{k}(k!)N^{-2\gamma k}
    \frac{| \det(B_N(z) )|^2}{\left( \varepsilon^{-1} N^{-\gamma}\right)^{2k} (N-N^*)^k}\\
  & \, \leq \frac{1}{1-\varepsilon_N^2}   | \det(B_N(z) )|^2.
  \end{align*}
The last display together with
\eqref{eq:det-exp}
yields
\eqref{eq-2.7}.
\end{proof}

\begin{remark}
  \corAB{Bounding higher (centered) moments of $\det (B_N(z)+X_4)$ and applying}
  \corOZ{the}
  \corAB{Borel-Cantelli lemma one can strengthen the conclusion of Theorem \ref{thm:logdet}} \corOZ{and show}
  \corAB{that \eqref{eq:log-det-conv} holds almost surely. This in turn shows that the conclusions of the main theorems of the paper, such as Theorems \ref{theo-1}, \ref{theo-2}, and \ref{thm:finite-off-diag-pc} hold almost surely. We do not pursue this direction here.}
\end{remark}

Lemma \ref{lem:det-B_N} shows that if $\varepsilon_N \downarrow 0$ then
the log determinant of $B_N(z)+X_4$ is asymptotically the
same as that of $B_N(z)$.
To establish Theorem \ref{thm:logdet} we also need to show
that the log-determinant of the Schur complement,
$\log \det(\widetilde{S}_N(z))$,
is asymptotically negligible,
see \eqref{eq:schur}.  To this end, we obtain the following
lower bound.
\begin{proposition} \label{lem:lb}
 Set
  $\widetilde{N}:= N^* \vee \corAB{\lceil \sqrt{N}\rceil}$. 
 If $N^* \le N/2$, then there exist absolute constants $c_2, c_1 > 0$ so that
  \begin{equation}\label{eq:schur-det-lbd}
    \mathbb{P}[|\det(\widetilde{S}_N(z))| \leq N^{-\gamma N^*}\sqrt{ ({N}^*!)} \corA{e^{-c_1\widetilde{N}}}] \leq \corA{e^{-c_2\widetilde{N}}}.
  \end{equation}
 \end{proposition}
 Before bringing the proof of Proposition \ref{lem:lb}, we recall the
following lemma, which
is proved in \cite[Lemma 4.4]{FPZ}
(in the real case, but the proof carries over to the complex case).
An alternative
proof can be given based on \cite[Theorem 4]{Sn}.
\begin{lemma}
  \label{lem:sd}
  Suppose that $E$ is a $Q \times Q$ standard Gaussian matrix.  Then for any $Q \times Q$ matrix $M$ independent of $E$ and all $t \geq 0,$
  \[
    \P[ |\det(E + M)| \leq t] \leq \P[ |\det E| \leq t].
  \]
\end{lemma}
We will also need the following lemma, whose
proof is an adaptation of the proof of \cite[Lemma 2.5]{FPZ}.
\begin{lemma}
  \label{lem:Gdet_lb}
  Let $E$ be an $N_0\times N_0$ matrix of
  independent complex standard Gaussians.
  There are absolute constants $c'_1 >0$ and $c'_2>0$ so that for all
  $N_0' \geq N_0$,
  \[
    \P[
      |\det E| \leq \sqrt{N_0!} e^{-c'_1 N_0'}
    ] \leq \frac{1}{c'_2}e^{-c'_2N_0'}.
  \]
\end{lemma}
\begin{proof}
  \corOZ{We begin by recalling from \cite{G} that if $E$ is a complex Ginibre matrix}
  \corAB{of dimension $N \times N$ matrix, then
\[
|\det E|^2 \lawequals \corABrev{2^{-N} \prod_{r=1}^N \chi_{2r}^2},
\]
where $\chi_r^2$ are independent}
\corOZ{chi-square random variables with $r$ degrees of freedom, i.e.\ they
have} \corAB{the distribution of the square of the
length of an $r$-dimensional standard \corABrev{real} Gaussian vector.} Now fix a large integer $K$ and denote
\[
F_K:= \prod_{r=1}^K \corABrev{\chi_{2r}}.
\]
Then, $F_K \le 2^{-N_0'}$ with exponentially (in $N_0'$)
small probability. Then,
proceeding as in the proof of \cite[Lemma 2.5]{FPZ}, we find
that
\[
\corABrev{\E L_K^{-2} \le \frac{K!}{N_0!}}
\]
for a sufficiently large $K$, where
\[
L_K:=\prod_{r=K+1}^{N_0} \corABrev{\chi_{2r}}.
\]
The rest follows from Markov's inequality.
\end{proof}
\begin{proof}[Proof of Proposition \ref{lem:lb}]
 By Lemma~\ref{lem:sd}, for all $t \geq 0$
 \[
   \P[|\det(\widetilde{S}_N(z))| \leq t] \leq \P[|\det(X_1)| \leq t],
 \]
 where $X_1$
 is an $N^* \times N^*$ matrix of \corAB{i.i.d.~complex
 Gaussians of variance $N^{-2\gamma}$}.  Hence, the desired conclusion follows from Lemma~\ref{lem:Gdet_lb}.
 \end{proof}
We turn to finding an upper bound on the determinant of
$\widetilde{S}_N(z)$. To this end, we first derive an
upper bound on the norm of the inverse of $B_N(z) + X_4$.
\begin{lemma} \label{lem:opbound}
Fix $\{\varepsilon_N\}$
such that $\varepsilon_N <1/8$ for all $N$, and
assume that
$N^* \le N/2$. Then,
  \[
\mathbb{P}\left(    \| (B_N(z) + X_4)^{-1} \| \leq 2\varepsilon_N N^{\gamma} (N-N^*)^{-1/2}\right) \ge 1 -\exp(-cN),
  \]
for some absolute constant $c$.
\end{lemma}
\begin{proof}
Gordon's theorem for Gaussian matrices (see \cite[Corollary 1.2]{gordon})
and the triangle inequality give
\[
\E\|X_4\| \le 2\sqrt{2}N^{-\gamma}\sqrt{N-N^*}.
\]
Since $A \mapsto \|A\|$ is a $1$-Lipschitz function,
the standard concentration inequality for Gaussian
random variables applies and yields
\[
\mathbb{P}\left(\|X_4\| \ge 4 N^{-\gamma}\sqrt{N-N^*}\right) \le \exp(-2c(N-N^*)) \le \exp(-cN),
\]
for some absolute constant $c$. On the other hand, by our definition,
see \eqref{eq:N*-dfn},
\[
\sigma_{\min}(B_N(z)) \ge \varepsilon_N^{-1} N^{-\gamma} \sqrt{N-N^*},
\]
where $\sigma_{\min}(B)$ denotes the minimum singular value of $B$. Since $\varepsilon_N < 1/8$, Weyl's inequality
(see \cite[Theorem 3.3.16(c)]{horn37topics}) gives that
\[
\sigma_{\min}(B_N(z)+X^4) \ge (2\varepsilon_N)^{-1} N^{-\gamma} \sqrt{N-N^*},
\]
with probability at least $1-\exp(-cN)$. This completes the proof.
\end{proof}
Building on Lemma \ref{lem:opbound} and using a
standard concentration inequality we have the following result.
\begin{lemma}\label{lem:schur-row-norm}
  Fix $\{\varepsilon_N\}$
such that $\varepsilon_N <1/8$ for all $N$, and
assume that
$N^* \le N/2$.
Then there exist absolute constants
$c'$ and $C'$ such that the $\ell_2$-norm of each of the rows
of $X_2 ( B_N(z) +X_4 )^{-1} X_3$ is at most
\[
    C'\varepsilon_N N^{-\gamma+1/2},
\]
with probability at least
\(
1- \exp(-c'N).
\)
\end{lemma}
\begin{proof}
    By the rotation invariance of $X_2$ and $X_3,$
    \[
      X_2 ( B_N(z) +X_4 )^{-1} X_3
      \lawequals
      X_2 D X_3,
    \]
    where $D$ is a diagonal matrix with entries
    equal to the singular values of $( B_N(z) +X_4 )^{-1}.$  Hence,
    a row of \corABrev{$X_2( B_N(z) +X_4 )^{-1}X_3$} is equal in distribution to
    \[
      N^{-\gamma}(D {\bm x})^t X_3,
    \]
    where ${\bm x}$ is an $(N-N^*)$ dimensional standard complex
    Gaussian vector independent of $X_3$ \corAB{and for any vector ${\bm y}$ the notation ${\bm y}^t$ denotes its transpose}.  From the
    rotation invariance of $X_3$, it follows that
    \[
      (D {\bm x})^t X_3 \lawequals \|D {\bm x}\|_2 e_1^t X_3.
    \]
    The law of $\|D {\bm x}\|_2$ is stochastically dominated by the law
    of $\|D\| \|{\bm x}\|_2$, and so we conclude that the law of \corAB{the $\ell_2$-norm of} a
    row of $X_2( B_N(z) +X_4 )^{-1}X_3$ is stochastically dominated
    \corABrev{by $\frac14 \cdot N^{-2\gamma}\chi_{2N^*} \cdot  \chi_{2(N-N)*}\cdot \|( B_N(z) +X_4 )^{-1}\|,$
    where again $\{\chi_r\}$ are random variables distributed as the
    length of an $r$-dimensional standard real Gaussian vector}.
    Applying Lemma~\ref{lem:opbound} and standard tail bounds for
    $\chi$--variables, the result follows.
\end{proof}
Using Lemma \ref{lem:schur-row-norm} we now find an upper bound on the determinant of the Schur complement.

\begin{lemma}\label{lem:hadamard}
  Fix $\{\varepsilon_N\}$
such that $\varepsilon_N\to 0$ and $\varepsilon_N <1/8$ for all $N$, and
assume that
$N^* \le N/2$.
Then there are absolute constants $\bar{c}$ and $\bar{C}$ such that
  \[
    |\det(\widetilde{S}_N(z))|
    <
    \bar{C}^{N^*} N^{-\gamma N^*} \varepsilon_{N}^{-N^*} \left(\corE{ N }\right)^{N^*/2}
  \]
  with probability at least $1- \corE{\exp(-\bar{c} N)}$.
\end{lemma}
\begin{proof}
Note that the rows of $X_1$ have $\ell_2$ norm at most
$2N^{1/2-\gamma}$
with probability at least $1-\exp(-\bar{c}_1N)$ where
$\bar{c}_1$ is some absolute constant.
By  construction of the diagonal matrix $S_N(z)$,  its entries (and hence
the $\ell_2$ norm of its rows)
are bounded by \corAB{$\varepsilon_N^{-1} N^{-\gamma+1/2}$}.
It follows from these facts and
Lemma \ref{lem:schur-row-norm} that the
$\ell_2$ norm of the rows of $\widetilde{S}_N(z)$ are bounded
by a constant multiple of $\varepsilon_N^{-1} N^{1/2-\gamma}$,
with probability at least $1-\exp(-\bar{c}N)$.
Hadamard's bound on the determinant yields the desired conclusion.
\end{proof}
Equipped with all the ingredients we are now ready to prove Theorem \ref{thm:logdet}.
\begin{proof}[Proof of Theorem \ref{thm:logdet}]
From Lemma \ref{lem:det-B_N}, upon using Markov's inequality we find
\begin{equation}
\mathbb{P} \left(\left| \frac{\det(B_N(z)+X_4)}{\det(B_N(z))}-1 \right| \ge \frac{1}{2} \right) \le \frac{4\varepsilon_N^2}{1-\varepsilon_N^2}. \notag
\end{equation}
Therefore we see that there is a set $\Omega_{N,1}$ such that on $\Omega_{N,1}$ we have
\begin{equation}\label{eq:non-schur-bd}
\left|\frac{1}{N} \log \det(B_N(z)+X_4) - \frac{1}{N}\log \det B_N(z)\right| \le  \corAB{\frac{2}{N}} \left| \frac{\det(B_N(z)+X_4)}{\det(B_N(z))}-1\right| \le \corAB{\frac{{1}}{N}},
\end{equation}
and
\begin{equation}\label{eq:non-schur-prob}
\mathbb{P}(\Omega_{N,1}^c) \le \frac{4\varepsilon_N^2}{1-\varepsilon_N^2}.
\end{equation}
Since $\gamma >1/2$, from Proposition \ref{lem:lb} and Lemma \ref{lem:hadamard} we also obtain that
\begin{align}
\left|
\frac{1}{N} \log |\det(\widetilde{S}_N(z))|
-
\frac{N^*\log N}{N}(-\gamma + \tfrac{1}{2})
\right| & \, \le\frac{N^* }{N} \left[{\log\bar{C}}+ {\log \left( \frac{1}{\varepsilon_N}\right)}+\frac{1}{2}{\log\frac{N}{N^*}}\right]+\corAB{(c_1+1)} \frac{\widetilde{N}}{N} \notag\\
& \, \le \frac{N^* \log N}{N} \left[\frac{\log\bar{C}+\log \left( \frac{1}{\varepsilon_N}\right)+\frac{1}{2}\log (N/N^*)}{\log N} + \gamma\right]+ \corAB{(c_1+1)} \frac{\widetilde{N}}{N} \label{eq:schur-bound}
\end{align}
on an event $\Omega_{N,2}$ such that
\begin{equation}\label{eq:schur-bd}
\mathbb{P}(\Omega_{N,2}^c) \le \exp(-\bar{c}N)+ \exp(-c_2 \widetilde{N}).
\end{equation}
Hence when $N^* \log N / N \to \alpha < \infty,$ then taking $\corAB{|\log(\varepsilon_N)|} = o(\log N),$ the desired conclusion holds.
With $\varepsilon_N=N^{-\eta}$ and $N^*=o(N/\log N)$ we deduce from \eqref{eq:schur-bound} that
\begin{equation}\label{eq:schur-prob}
\left|\frac{1}{N} \log \det(\widetilde{S}_N(z))\right| =o(1),
\end{equation}
on the event $\Omega_{N,2}$. Now combining \eqref{eq:non-schur-bd}-\eqref{eq:non-schur-prob} and \eqref{eq:schur-bd}-\eqref{eq:schur-prob} we see that the convergence in \eqref{eq:log-det-conv} holds in probability.
\end{proof}

\section{Bidiagonal matrices: rigidity theorems for small singular vectors of $D_N+J$}
\label{sec-d=1}

In this section, we develop estimates for the small
singular values of the bidiagonal matrix $M_N = D_N+J$
where \corABrev{$D_N=-\diag(d_1,d_2, \dots, d_N)$} is a diagonal matrix. \corOZ{These}
\corAB{are then used to prove Theorem \ref{theo-2}}. \corAB{For all $i \leq j$, let $\mathcal{D}^{i,j}$ be defined by}
\[
  \mathcal{D}^{i,j} := \prod_{i \leq \ell < j} d_\ell.
\]
Note that $\mathcal{D}^{i,i}=1$ by this definition.
For all $i \leq j$ define a vector by, for each $k,$
\begin{equation}\label{eq:vdefine}
  {\bm v}^{i,j}_k
  :=
  \begin{cases}
    0 &   k < i, \\
    \mathcal{D}^{i,k} & i \leq k \leq j, \\
    0 &   j < k.
  \end{cases}
\end{equation}
The point of these vectors is that they solve $(M_N {\bm v}^{i,j})_k = 0$ for $k$ between the boundaries.   Precisely:
\begin{equation}
  (M_N {\bm v}^{i,j})_k
  = \begin{cases}
    1 & k = i-1 \\
    -\mathcal{D}^{i,j+1} & k = j \\
    0 & \text{otherwise}.
  \end{cases}
  \label{eq:MNv}
\end{equation}
Hence if $\mathcal{D}^{i,j+1}$ is much smaller than some $\mathcal{D}^{i,k} \gg 1$ for $i < k \leq j,$  then this will be nearly a small singular vector, \corAB{i.e.~a singular vector corresponding to a small singular value}.

Note that for $i \leq k \leq j-1,$ we have the identity
\begin{equation}
  \label{eq:iteration}
  {\bm v}^{i,j}_{k+1}
  = d_k
  {\bm v}^{i,j}_{k}.
\end{equation}

\corAB{Using the vectors ${\bm v}^{i,j}$ we now construct approximate small singular vectors. Fix an integer $L:=L_N$, the choice of which will be determined later. Consider} integers
\[
  0 = i_1 < i_2 < i_3 < \dots < \corAB{i_L < i_{L+1}} = N.
\]
The vectors $\left\{ {\bm w}^j = {\bm v}^{i_j+1, i_{j+1}}\|{\bm v}^{i_j+1, i_{j+1}}\|_2^{-1} , j=1,2,\ldots,\corAB{L}\right\}$ have disjoint supports and are therefore orthogonal.  We use these vectors as approximations for the small singular vectors and quantify the approximation. To this end, define
\begin{equation}
  \label{eq:ddef}
    \begin{aligned}
      \mathcal{D}_{+}^{i_j,i_{j+1}}
      &:=
      \max_{i_j < s \leq r \leq i_{j+1}}
      \sum_{p=s}^{r} \bigl\{|\mathcal{D}^{p,r}|+ |\mathcal{D}^{s,p}|\bigr\}, \text{ and} \\
      \mathcal{D}_{-}^{i_j,i_{j+1}}
      &:=
      \max_{i_{j} < s\leq r \leq i_{j+1}}
      \sum_{p=s}^r \biggl\{\frac{1}{|\mathcal{D}^{p,r}|}+\frac{1}{|\mathcal{D}^{s,p}|}\biggr\}.
    \end{aligned}
  \end{equation}
  Provided the entries $\left\{ |d_k| : i_j \leq k \leq i_{j+1} \right\}$ are consistently larger than $1$ or consistenly smaller than $1$, at least one of these quantities will be close to $1.$
  Even in the case of independent $d_i$s, in which there may exist a relatively long string of diagonal entries with possibly atypical magnitude, it is unlikely that both of these will be large.
  We then let
\begin{equation}
  \mathfrak{D} := \max_{1 \leq j \leq \corOZ{L}}
  \left[
  \min\left\{
      \mathcal{D}_{+}^{i_j,i_{j+1}},
      \mathcal{D}_{-}^{i_j,i_{j+1}}
    \right\}\right]\geq 1.
  \label{eq:uniform}
\end{equation}
When $\mathfrak{D}$ is small, the approximation will be good.
Let $\pi_j$ denote the coordinate projection from $\C^N$ to the coordinates that support ${\bm w}^j.$  Our main result in this section on singular values of $M_N$ is the following.
\begin{theorem}
The $(\corAB{L} +1)$-st smallest singular value satisfies
\[
\sigma_{N-\corAB{L}}(M_N) \geq \mathfrak{D}^{-1}.
\]
There is an absolute constant $C> 1$ so that the product of the $\corAB{L}$ smallest singular values of $M_N$ satisfies
\begin{equation}\label{eq:thm31bd}
    (C\|M_N\| \mathfrak{D} \sqrt{\corAB{L}})^{-\corAB{L}}
    \prod_{k=1}^{\corAB{L}} \|\pi_k M_N {\bm w}^k\|_2
    \leq
    \prod_{k=0}^{\corAB{L}-1} \sigma_{N-k}(M_N)
    \leq
    \prod_{k=1}^{\corAB{L}} \|\pi_k M_N {\bm w}^k\|_2.
\end{equation}
\label{thm:DJssv}
\end{theorem}

\corAB{The proof of Theorem \ref{thm:DJssv} is deferred to Section \ref{sec:pfDJssv}.}

\subsection{Applications}
Theorem \ref{theo-2} regarding the behavior of the eigenvalues of $\model_N$
in the bidiagonal case
follows from direct applications of Theorems~\ref{thm:DJssv} and
\ref{thm:logdet}.  We show now how these applications follow.

We begin with a particularly simple case to consider that is not directly related to Theorem \ref{theo-2}.
\begin{corollary}
\label{cor:Jordan}
 Consider a Jordan block
  \[
    J_N(z)=
        \begin{bmatrix}
                z  & 1 &   &   &    \\
                   & z & 1 &   &    \\
                   &   &\ddots & \ddots  &   \\
                   &   &   & z & 1  \\
                   &   &   &   & z  \\
        \end{bmatrix}.
\]
Setting $\mathfrak{F}^{-1} = \min\{|1-|z||, N^{-1}, |1-|z|^{-1}|\}/2$, there is a constant \corAB{$C(z)>1$, depending only on $z$,} so that for all $N$, \corAB{we have}
\begin{align*}
  \sigma_{N-1}(J_N(z)) \geq \mathfrak{F}^{-1}, \qquad \sigma_{N}(J_N(z)) \geq  \mathfrak{F}^{-1}(|z| \wedge 1)^N/C(z), \qquad \text{ and } \qquad \sigma_{N}(J_N(z)) \leq C(z) (|z|\wedge 1)^N.
\end{align*}
\end{corollary}
\begin{proof}
  We apply Theorem~\ref{thm:DJssv} with a single block, i.e.~$L=1, i_1=0,$ and $i_2 = N.$  We bound $\mathfrak{D} \leq \mathfrak{F}$ upon observing that $\corABrev{|\mathcal{D}^{p,r}| = |z|^{r-p}}$ for $r\geq p$. \corAB{By}
  \corOZ{the} \corAB{Gershgorin circle theorem we also have that \corABa{$\|J_N(z)\| =\|J_N J_N^*\|^{1/2} = O( |z|\vee 1)$}. Therefore, the proof now finishes upon using \eqref{eq:vdefine}-\eqref{eq:MNv}}.
  \end{proof}

\begin{remark} Observe that the same conclusions as in
Corollary \ref{cor:Jordan}
  hold if the diagonal of $J_N$ were
  replaced by arbitrary complex numbers having the same modulus.
\end{remark}
\begin{remark} By combining
Corollary \ref{cor:Jordan}
  with Theorem~\ref{thm:logdet}, we get a
  new proof of \cite[Theorem 1.4]{FPZ}. \end{remark}

When \corAB{$D_N$ is a diagonal matrix} of i.i.d.\,random variables the outcome is similar.
\begin{corollary}
  \label{cor:iidssv}
  Suppose the entries of \corAB{$D_N$} are  i.i.d.\,complex random variables with
  \begin{equation}
    \label{eq-exceed}
    \Exp |d_1|^{\pm \corAB{\beta_0}} < \infty \quad \mbox{\rm for some \corAB{$\beta_0 >0$}}
    \quad
    \mbox{\rm
    and $\Exp \log |d_1| \neq 0$}.
  \end{equation}
  Then, for every $\varepsilon > 0,$ there is a $\delta > 0$
and an \corAB{$L \leq 5 N^{1-\delta}$} so that with probability approaching
$1$ as $N\to\infty$,
  \(
  \sigma_{N-\corAB{L}}(M_N) \geq N^{-\varepsilon}
  \)
  and
  \[
  \prod_{k=0}^{\corAB{L}-1} \sigma_{N-k}(M_N) = e^{-N(\Exp\log|d_1|)_{-} + o(N)}.
  \]
\end{corollary}
\begin{proof}
 %
%
%
  \corAB{The key to the proof is to construct a partition of $[N]$ so that $\mathfrak{D}$ is small and the upper and the lower bounds of \eqref{eq:thm31bd} are evaluated easily.} \corOZ{To this end, we first note that}
 \begin{equation}\label{eq:claim-log}
 \lim_{\beta \to 0} \left(\E |d_1|^{\pm \beta} \right)^{1/\beta} = e^{\pm \E \log |d_1|},
 \end{equation}
 \corOZ{which follows from Taylor's theorem and dominated convergence using
 \eqref{eq-exceed}.}
%

 \corAB{Now we focus on the case $\E \log |d_1| >0$. By \eqref{eq:claim-log} there exists a $\beta > 0$ so that $p_\beta:=\Exp |d_1|^{-\beta} < 1$. We fix this $\beta$ for the remainder of the proof. Next} we recall that if $\{Z_k\}$ is a non-negative
  martingale with $Z_0=1$ then
  Doob's maximal inequality implies that
 \begin{equation}
   \label{eq-sharktank1}
  \P\left[ \sup_{1 \leq k } Z_k \geq t\right] \leq \frac{1}{t}.
\end{equation}
Let $G_\delta$ be the set of $j \in [N]$ so that there exists a \corAB{$k \in [N]$ for which}
  \[
    \corAB{  \prod_{i=j\wedge k}^{j \vee k} ( |d_i|^{-\beta}p_\beta^{-1}) \geq N^{2\delta}.}
  \]
  Then
  $$\P(j\in G_\delta)\leq
  \P(\sup_{k\geq 0} \prod_{i=j}^{j+k} ( |d_i|^{-\beta}p_\beta^{-1}) \geq N^{2\delta})+
  \P(\sup_{\corAB{2}\leq k\leq j} \prod_{i=k-1}^j ( |d_i|^{-\beta}p_\beta^{-1}) \geq N^{2\delta})=:P_1+P_2.$$
  Setting $Z_k:= \prod_{i=j}^{j+k-1} ( |d_i|^{-\beta}p_\beta^{-1})$,
  one sees that $\{Z_k\}$ is a non-negative martingale with $Z_0:=1$.
  Applying \eqref{eq-sharktank1} one gets that $P_1\leq N^{-2\delta}$.
  Similarly, setting
  $Z_k= \prod_{i=j-k+1}^{j} ( |d_i|^{-\beta}p_\beta^{-1})$,
  one gets $P_2\leq N^{-2\delta}$.
 Thus,
  $\Exp |G_{\delta}| \leq 2N^{1-2\delta},$
  and \corAB{hence by Markov's inequality}, with probability at least $1-N^{-\delta},$
  $|G_{\delta}| \leq 2N^{1-\delta}.$
  Let $S \subset \mathbb{Z}$ be defined by
  \begin{equation}\label{eq:dfn-S}
    S := \corABrev{\left\{\lfloor N^\delta k\rfloor; k \in [\lfloor N^{1-\delta}\rfloor ] \right\}} \bigcup \{\cup_{x \in G_{\delta}}(x+ \{0,1\})\},
  \end{equation}
  \corAB{where for two sets $S_1$ and $S_2$ we denote $S_1+S_2:=\{s_1+s_2:\, s_1 \in S_1, s_2 \in S_2\}$}. Enumerate the elements of $S$ as $\{ i_2 < i_3 < i_4 < \dots < i_{|S|+1} \}.$  Extend \corAB{the collection $i$'s} by letting $i_1 =0,$ $L = |S|+1,$ and $i_{L+1}=N.$  Then $L \leq 5N^{1-\delta}$ with probability
  at least $1-N^{-\delta}$,
  and the separation $i_{k+1} - i_k \leq  N^{\delta}$ for all $k.$

  \corAB{From the definition of the set $S$, it follows that} when $i_j \not \in G_{\delta},$ we have for any $i_j \leq r < i_{j+1}$ and any $k \geq r,$
  \[
  \prod_{i=r}^k |d_i|^{-1} \leq N^{2\delta/\beta} p_\beta^{k/\beta} \leq N^{2\delta/\beta},
  \]
  \corAB{where the last inequality follows from the fact that $p_\beta < 1$}. Hence, $\mathcal{D}_{-}^{i_j,i_{j+1}} \leq 2N^{\delta(1+2/\beta)}.$ \corAB{Recalling the definition of $S$ once more we see that} when $i_j \in G_{\delta},$ we have $i_{j+1} = i_j + 1$. Thus, in that case $\mathcal{D}_{-}^{i_j,j_{j+1}} \leq 2.$  Therefore, we conclude
  that if $\E\log |d_1|>0$ then
  $\mathfrak{D} \leq 2N^{\delta(1+2/\beta)}$,
  with probability at least $1-N^{-\delta}$. \corAB{Using \eqref{eq:claim-log}, a similar argument yields that
  \begin{equation}\label{eq:D_+}
  \mathfrak{D}_+^{i_j,i_{j+1}} \le 2 N^{\delta(1+2/\beta)} \qquad \text{ for all } j =1,2,\ldots,L,
  \end{equation}
  with probability at least $1-N^{-\delta}$, when $\E \log |d_1| <0$. Therefore, the same bound as above holds for $\mathfrak{D}$ in this case.}

  Now taking $\delta=\delta(\varepsilon)$ small
  enough and invoking the first part of Theorem~\ref{thm:DJssv},
  one concludes that with the same probability,
  $\sigma_{N-L}(M_N) \geq N^{-\varepsilon}$, as claimed.

  To show the second part of the corollary,
  \corAB{we claim that for any $\eta > 0$ there is a $\delta':=\delta'(\eta) > 0$, with $\lim_{\eta \to 0} \delta'(\eta)=0$, such that for all $t>0$}
  \begin{equation}
    \label{eq:tb}
  \max\left\{ \P\left[ \sum_{i=1}^k \log |d_i| \leq -t + k(\Exp \log|d_1| -\eta)\right],
  \P\left[ \sum_{i=1}^k \log |d_i| \geq t + k(\Exp \log|d_1|+\eta)\right]
   \right\} \leq e^{-\delta' t}.
  \end{equation}
 \corAB{To see the above, applying Markov's inequality, we note that for any $\beta' \in (0,\beta_0]$,
  \[
  \P\left[ \sum_{i=1}^k \log |d_i| \geq t + k(\Exp \log|d_1|+\eta)\right] \le \frac{\E\left[ \prod_{i=1}^k |d_i|^{\beta'}\right]}{\exp\left(\beta' k \E\log |d_1|\right)} \cdot e^{-\beta' t - \beta' k \eta}.
  \]
 Now using \eqref{eq:claim-log} and choosing $\beta'$ sufficiently small (depending on $\eta$) we deduce
  \[
  \P\left[ \sum_{i=1}^k \log |d_i| \geq t + k(\Exp \log|d_1|+\eta)\right] \le e^{-\delta' t}.
  \]
Using a similar argument} \corOZ{for the second probability in
  the left hand side of \eqref{eq:tb},  we conclude that
  \eqref{eq:tb} holds.}

 \corOZ{We continue with the proof of the second
 part of the corollary.
 From \eqref{eq:tb}}
 we conclude that for any $\eta > 0$,
  there is a constant $\bar{C}(\eta)>0$, \corAB{with $\lim_{\eta \to 0} \bar{C}(\eta)^{-1}=0$}, so that,
  with probability $1-N^{-100}$,
  for all $1 \leq j \leq L$,
  \begin{equation}
    \label{eq-lawandorder3}
    e^{2(i_{j+1}-i_j)(\Exp \log |d_1|- \eta)-\bar{C}(\eta) \log N}
    \leq
    \prod_{i=i_j+1}^{i_{j+1}}
    |d_i|^2
    =
    \|\pi_j M_N {\bm v}^{i_j+1,i_{j+1}}\|^2_2\leq
    e^{2(i_{j+1}-i_j)(\Exp \log |d_1|+ \eta)+\bar{C}(\eta) \log N}.
  \end{equation}
In the case that $\Exp \log |d_1| > 0$, \corAB{using the fact that $i_{j+1} - i_j \le N^\delta$ we see that}
  \[
    \prod_{i=i_j+1}^{i_{{j+1}}} |d_i|^2
    \leq \|{\bm v}^{i_j+1,i_{j+1}}\|^2_2
    \leq N^{\delta} \max_{i_j < r \leq i_{j+1}} \prod_{i>i_j+1}^{r} |d_i|^2.
  \]
  \corAB{Proceeding similarly} \corOZ{to the proof of \eqref{eq-lawandorder3}
and applying an union bound} we conclude that with probability at least \corAB{$1-N^{-99}$},
  for all $1 \leq j \leq L$
  \begin{equation}
    \label{eq-lawandorder22}
    e^{2(i_{j+1}-i_j)(\Exp \log|d_1|-\eta)-\corAB{\bar{C}(\eta)}\log N}
    \leq
    \|{\bm v}^{i_j+1,i_{j+1}}\|^2_2
    \leq
    e^{2(i_{j+1}-i_j)(\Exp \log|d_1|+\eta)+\corAB{(\bar{C}(\eta)+\delta)}\log N}.
  \end{equation}
 Recall that
 $ \|\pi_k M_N {\bm w}^{k}\|_2=\|\pi_k M_n {\bm v}^{i_k+1,i_{k+1}}\|_2/
    \|{\bm v}^{i_k+1,i_{k+1}}\|_2$.
    \corAB{Hence combining \eqref{eq-lawandorder3}-\eqref{eq-lawandorder22}, using the fact $i_{j+1} - i_j \le N^\delta$ again, and taking $\eta \to 0$ sufficiently slowly with $N$, so that $\bar{C}(\eta) \log N = o(N^\delta)$, we}
  conclude that  if $\E\log |d_1|>0$ then
  \begin{equation}
    \label{eq-lawandorder1}
    e^{-o(N)}
    \leq
    \prod_{k=1}^L \|\pi_k M_N {\bm w}^{k}\|_2
    \leq
    e^{o(N)},
  \end{equation}
  \corAB{with probability} \corOZ{approaching} \corAB{one as $N \to \infty$.}


 \corAB{Turning to the case $\Exp \log |d_1| < 0$ we begin with the estimate}
  \[
    1
    \leq \|{\bm v}^{i_j+1,i_{j+1}}\|^2_2
    \leq N^{\delta} \corAB{\mathfrak{D}_+^{i_j,i_{j+1}} \le 2 N^{2\delta(1+1/\beta)}, \qquad \text{ for all } j =1,2,\ldots,L,}
  \]
  \corAB{with probability at least $1-N^{-\delta}$, where the last step follows from \eqref{eq:D_+}. Therefore}
  \begin{equation}
    \label{eq-lawandorder2}
    \corAB{e^{\sum_{i=1}^N \log|d_i| -o(N)}
    \leq
    \prod_{k=1}^L \|\pi_k M_N {\bm w}^{k}\|_2= \frac{\prod_{i=1}^N |d_i|}{\prod_{j=1}^L \|{\bm v}^{i_j+1,i_{j+1}}\|_2}
    \leq
    e^{\sum_{i=1}^N \log|d_i|},}
  \end{equation}
  \corAB{with probability approaching one, as $N \to \infty$.}

\corAB{To complete the proof using Markov's inequality and the union bound we note that $\P(\max_i |d_i| \ge N^{2/\beta_0}) =O(1/N)$.
\corEE{Therefore, using the Gershgorin circle theorem we derive that $\|M_N\|^2 = \|M_N^*M_N\| = O(N^{4/\beta_0})$ with probability tending to one.} Now applying Theorem~\ref{thm:DJssv}, the derived bound on $\mathfrak{D}$}, using \eqref{eq-lawandorder1} when $\E\log |d_1|>0$, and \eqref{eq-lawandorder2} \corAB{and Chebycheff's inequality} when $\E\log |d_1|<0$, the corollary follows.
  \end{proof}

  The next corollary and Remark \ref{rem-simon} following it,
  combine Corollary \ref{cor:iidssv}
with Theorem~\ref{thm:logdet} to obtain Theorem \ref{theo-2} b).
  \begin{corollary}
    \label{cor:iid}  Suppose that
    $D_N$ is a diagonal matrix of i.i.d.\,complex random variables of law $\nu$,
    and let $M_N=D_N+J$.  Suppose that
    $\Exp |d_1|^{\pm \beta_0} < \infty$ for some $\beta_0> 0$ and that
    $\Exp\log|d_1 - z|\neq 0$,  for Lebesgue-a.e.~$z \in \C$.  Then \corAB{$L_N$}
   converges weakly in probability to the probability measure with
   log potential $(\Exp \log| d_1 - z|)_+$.
  \end{corollary}

 \begin{remark}
   \label{rem-simon}
  Let $\corAB{\mathscr{S}:=} \, \supp(d_1)$.
   The condition that $\corAB{\LL_\nu(z) }= \Exp\log|d_1 - z|\neq 0$, for
   Lebesgue-a.e.\,$z$ is satisfied, for example, when  \corAB{$\mathscr{S}$}
   is contained in a compact simply connected set of \corAB{two-dimensional Lebesgue} measure $0$.  Indeed, note that $\corAB{\LL_\nu}$ is harmonic on \corAB{$\mathscr{S}^c$}.
\corAB{$\LL_\nu$} cannot vanish identically in \corAB{$\mathscr{S}^c$} because \corAB{$\mathscr{S}^c$} is connected and
$\limsup_{z\to\infty}\corAB{\LL_\nu(z)} = \infty$. By \cite[Theorem 3.1.18]{Simon}, it
follows that the zero set of \corAB{$\LL_\nu$} in \corAB{$\mathscr{S}^c$} has \corAB{zero} Lebesgue measure.
%
 \end{remark}

  \begin{proof}[Proof of Corollary \ref{cor:iid}.]
    We note first that by Weyl's inequalities for singular values,
    if $d_i^\downarrow$, $i=1,\ldots,N$ denote the monotone decreasing
    reordering of the
    variables $\corABrev{|d_1|,\ldots,|d_N|}$ then $|\sigma_i(\model_N)-d_i^\downarrow|\leq 2$
    for all $i$, with probability approaching $1$ as $N\to\infty$. (In the last statement, we used that $\|J\|\leq 1$ and
    that by \eqref{eq-gordon},
    $\|N^{-\gamma} G_N\|\to_{N\to\infty} 0$.) By Weyl's majorant theorem
    \cite[Theorem II.3.6]{bhatia} it follows that, with
    probability tending to $1$ as $N\to\infty$,
    $$\sum_{i=1}^N \log_+ |\lambda_i(\model_N)|\leq
    \sum_{i=1}^N \log_+ \sigma_i(\model_N)\leq
    \sum_{i=1}^N \log_+ (|d_i|+2).$$
   Since $\E |d_1|^{\beta_0} <\infty$ it follows that $\E\log_+(|d_1|+2)<\infty$. \corABrev{Therefore, denoting by $B_\C(0,R)$ the ball of radius $R$ in the complex plane centered at zero and using the law of large numbers, we find that
 \[
 L_N(B_\C(0,R)^c) = \frac1N \sum_{i=1}^N {\bf 1}_{\{|\lambda_i(\model_N)  | >R\}} \le (\log _+ R)^{-1} \cdot \frac1N \sum_{i=1}^N \log_+|\lambda_i(\model_N)| \le 2 \cdot \frac{\E \log_+(|d_1|+2)}{\log_+ R},
   \] 
 with probability approaching one. Since $R <\infty$ is arbitrary, this in turn implies that the sequence of random probability measures $\{L_N\}_{N \in \N}$ is tight.} Therefore, by
    \cite[Theorem 2.8.3]{tao2012topics}, the corollary follows
    once one proves that
    $\LL_{\model_N}(z)\to \E \log|d_1-z|_+$ in probability for Lebesgue
    almost every
    $z \in \C$.

    To prove the convergence of $\LL_{\model_N}$, we check that
    $D_N - z\Id + J$ satisfies the hypotheses of Corollary~\ref{cor:iidssv}
    \corAB{for Lebesgue-a.e.~every $z$}.  By assumption $\Exp |d_1|^{\beta_0} < \infty$
    and hence $\Exp |d_1-z|^{\beta_0} < \infty$ for all $z \in \C$.
    Observe that for any $M >0$
    \[
      \int_\C \int_{|z-y| < M} \frac{1}{|z-y|} |dz|^2\nu(dy)
      = 2\pi M
    \]
    with $|dz|^2$ the $2$-dimensional Lebesgue volume element.
    In particularly, as this is locally integrable, \corAB{by Jensen's inequality,} we have that
    \[
   \corAB{  \E |z- d_1|^{-\beta_0} \le \left( \Exp |z-d_1|^{-1} \right)^{\beta_0}}< \infty
    \]
    is finite for Lebesgue-a.e.\,$z$, \corAB{where without loss of generality we have assumed $\beta_0 \le 1$}. Hence $M_N-z\Id$ satisfies the hypotheses of Corollary \ref{cor:iidssv} for Lebesgue-a.e.\,$z \in \C.$

    Let $z \in \C$ be a point at which these hypotheses are satisfied.
    Choosing \corAB{$\varepsilon, \eta > 0$ sufficiently small so that} 
 $ \varepsilon < \gamma - \tfrac{1}{2}-\eta,$ we get from Corollary \ref{cor:iidssv} that there is a $\delta > 0$ and an $\corAB{L \leq 5 N^{1-\delta}}$ so that $\sigma_{N-\corAB{L}}(M_N-z\Id) \geq N^{-\varepsilon}$ and
    \begin{equation}\label{eq:prodsing-iid}
      \prod_{k=0}^{\corAB{L}-1} \sigma_{N-k}(M_N-z \Id) = e^{-N(\Exp\log|d_1-z|)_{-} + o(N)},
    \end{equation}
    \corAB{with probability approaching to one, as $N \to \infty$.} Applying Theorem~\ref{thm:logdet}, with $\varepsilon_N = N^{-\eta},$ we get that
    \begin{equation}\label{eq:conse-21}
      \frac{1}{N}\log|\det(\mathcal{M}_N-z\Id)|
      -
      \frac{1}{N}\log\frac{|\det(M_N-z\Id)|}{ \prod_{k=0}^{N^{*}-1} \sigma_{N-k}(M_N-z \Id)} \to 0
    \end{equation}
    in probability,
    where we recall that $N^*$ was defined in \eqref{eq:N*-dfn} as the largest $i$ so that
    \[
      \sigma_{N-i} < N^{\eta - \gamma}(N-i)^{1/2} < N^{-\varepsilon}.
    \]
    So $N^* \leq \corAB{L},$ and therefore
    \begin{equation}\label{eq:not-small-sing-bd}
      (N^{\eta-\gamma}(N-N^*)^{1/2})^{\corAB{L}}
      \leq
      \frac
      {\prod_{k=0}^{\corAB{L}-1} \sigma_{N-k}(M_N-z \Id)}
      {\prod_{k=0}^{N^{*}-1} \sigma_{N-k}(M_N-z \Id)}
      \leq
      \|M_N - z\Id\|^{\corAB{L}}.
    \end{equation}
    \corAB{As we have already seen during the proof of Corollary \ref{cor:iidssv} that $\|M_N\| = O(N^{2/\beta_0})$ with probability approaching one, by the triangle inequality the same bound continues to hold for $\|M_N -z \Id\|$. Therefore, using the fact that $L=O(N^{1-\delta})$ we deduce that both the upper and lower bounds in \eqref{eq:not-small-sing-bd} are sub-exponential in $N$. On the other hand, since $M_N$ is an upper triangular matrix, by Chebycheff's inequality it also follows that $|\det(M_N-z\Id)|=e^{N\Exp\log|d_1-z| + o(N)},$ with probability tending to one. Hence, combining \eqref{eq:prodsing-iid}-\eqref{eq:conse-21} we deduce}
%
    \[
      \frac{1}{N}\log|\det(\mathcal{M}_N-z\Id)|
      -
      \frac{1}{N}\log\frac{e^{N\Exp\log|d_1-z| + o(N)}}{e^{-N(\Exp\log|d_1-z|)_{-} + o(N)}} \to 0
    \]
    in probability.  As this holds for Lebesgue almost every $z \in \C,$ the claimed statement follows.
  \end{proof}
The next corollary deals with $M_N=D_N+J$ where the entries in $D_N$
vary slowly.
  \begin{corollary}
    \label{cor:holderssv} Suppose that
    $f : [0,1] \to \C$ is \corAB{$\alpha$}-H\"older continuous,
    and that $d_i = f(i/N)$ for $1 \leq i \leq N.$
    For every $\varepsilon > 0,$ there is a $\delta >0$ and an
    $\corAB{L = O(N^{1-\delta})}$ so that
  \(
  \sigma_{N-\corAB{L}}(M_N) \geq N^{-\varepsilon}
  \)
  and
  \[
    \prod_{k=0}^{\corAB{L}-1} \sigma_{N-k}(M_N) = e^{\sum_{i=1}^N -(\log |f(i/N)|)_{-} + o(N)}.
  \]
  \end{corollary}
  \begin{proof}
  \corAB{The key to the proof is again to construct a suitable partition of $[N]$. Fix any $\delta  \in (0,1)$.} Inductively define $a_j \in \N$ by setting $a_1:=1$ and letting
  \[
    a_{j+1} := \begin{cases}
      \inf\{ a_j < k \leq N : \log |f(k/N)| > N^{-\delta} \}, &\text{if } \log |f(a_j/N)| < 0, \\
      \inf\{ a_j < k \leq N : \log |f(k/N)| < -N^{-\delta} \}, &\text{if } \log |f(a_j/N)| > 0. \\
    \end{cases}
  \]
 \corAB{From the definition of $\{a_j\}$ and the H\"older continuity of $f$ we see that} there is \corAB{a constant $C_0$} so that
 \[
\corABrev{\frac12} N^{-\delta} \leq |f(a_j/N) - f(a_{j+1}/N)| \leq \corAB{C_0}\left(\frac{a_{j+1}-a_j}{N}\right)^{\alpha}
 \]
 \corAB{and note that $|\{a_j\}| = O(N^{\delta/\alpha})$}.
 Let
 \[
   S := \cup_{k=1}^{\lfloor N^{1-\delta}\rfloor } \lfloor N^\delta k\rfloor \bigcup \{a_j: j \in \N\}.
 \]
 Enumerate the elements of $S$ as $\{ i_2 < i_3 < i_4 < \dots < i_{|S|+1} \}.$  Extend this by letting $i_1 =0,$ $\corOZ{L}= |S|+1,$ and $i_{\corOZ{L}+1}=N.$  \corAB{Then, upon reducing $\delta$ if necessary, we obtain $L = O(N^{1-\delta})$}, and the separation $i_{k+1} - i_k < N^{\delta}$ for all $k$.

 Next, recalling the definition of $\mathfrak{D}^{s,r}$, by the construction of $\{a_j\}$, \corAB{and the triangle inequality} we also have that either
 \[
   |\mathcal{D}^{s,r}| \leq \corAB{2} e^{(r-s)N^{-\delta}}
   \;\text{for all $i_k < s \leq r \leq i_{k+1}$  or }
   |\mathcal{D}^{s,r}| \geq \corAB{\frac12} e^{-(r-s)N^{-\delta}}
   \;\text{for all $i_k < s \leq r \leq i_{k+1}$.}
 \]
 Hence $\mathfrak{D} \leq \corAB{4}eN^{\delta}$, \corAB{which upon an application of  Theorem~\ref{thm:DJssv} yields that given any $\vep >0$, there exists $\delta>0$, sufficiently small such that $\sigma_{N-L}(M_N) \ge N^{-\vep}$.} \corABa{As $\|M_N\| =\|M_N M_N^*\|^{1/2} = O(\sup_{x \in[0,1]} |f(x)|)$, by the Gershgorin circle theorem}, applying Theorem~\ref{thm:DJssv} \corOZ{once again, it} remains to show that
 \begin{equation}\label{eq:prodsing-d1}
    \prod_{k=1}^{\corAB{L}}\|\pi_k M_N {\bm w}^k\|_2
    =
    e^{\sum_{i=1}^N -(\log |f(i/N)|)_{-} + o(N)}
 \end{equation}
 where we recall that
 ${\bm w}^k ={\bm v}^{i_k+1, i_{k+1}}\|{\bm v}^{i_k+1, i_{k+1}}\|_2^{-1}$ for $1 \leq k \leq \corAB{L}.$ \corAB{To this end, we note that} for each
 $1 \leq k\leq \corOZ{L}$ we have that either
 \[
   \log |f(s/N)| \leq N^{-\delta}
  \; \text{for all $i_k < s \, \corAB{<} \, i_{k+1}$ or }
   \log |f(s/N)| \geq -N^{-\delta}
  \; \text{for all $i_k < s \, \corAB{<}\,  i_{k+1}$.}
 \]
 \corAB{Using the H\"older continuity of $f$ we further have that $|f(i_{k+1}/N)| \le 2 e^{N^{-\delta}}$ in the first case and $|f(i_{k+1}/N)| \ge \frac12 e^{-N^{-\delta}}$ in the second case.} Therefore, in the first case, we have that
 \[
   1
   \leq \|{\bm v}^{i_k+1, i_{k+1}}\|^2_2
   \leq 2 e^2 N^{\delta},
 \]
 and
 \begin{align*}
   |\mathcal{D}^{i_k+1,i_{k+1}+1}| \corAB{= \prod_{j=i_k+1}^{i_{k+1}} |f(j/N)|} & = \corAB{ \exp\left( \sum_{j=i_k+1}^{i_{k+1}} (\log |f(j/N)|)_{+}- \sum_{j=i_k+1}^{i_{k+1}} (\log |f(j/N)|)_{-}\right)}\\
      & =\exp\biggl(O(1)-\sum_{j=i_k+1}^{i_{k+1}} (\log |f(j/N)|)_{-}\biggr).
 \end{align*}
 Hence,
  \begin{equation}\label{eq:prodsing-d1-1}
   \|\pi_k \corAB{M_N}{\bm w}^k\|_2
   =\exp\biggl(O(\log N)-\sum_{j=i_k+1}^{i_{k+1}} (\log |f(j/N)|)_{-}\biggr).
 \end{equation}
Arguing similarly, in the second case, we have that
 \[
   |\mathcal{D}^{i_k+1,i_{k+1}}|^2
   \leq \|{\bm v}^{i_k+1, i_{k+1}}\|^2_2
   \leq 2e^2 N^{\delta} |\mathcal{D}^{i_k+1,i_{k+1}}|^2,
 \]
 and therefore
 \begin{equation}\label{eq:prodsing-d1-2}
   \|\pi_k \corAB{M_N} {\bm w}^k\|_2
   =\exp\biggl(O(\log N)\biggr) =\exp\biggl(O(\log N)-\sum_{j=i_k+1}^{i_{k+1}} (\log |f(j/N)|)_{-}\biggr).
 \end{equation}
  \corAB{Combining \eqref{eq:prodsing-d1-1}-\eqref{eq:prodsing-d1-2} we arrive at \eqref{eq:prodsing-d1}. This completes the proof.}
 \end{proof}
 Building on
    Corollary \ref{cor:holderssv}, the next
 corollary is Theorem \ref{theo-2} b).
 \begin{corollary}
   \label{cor:holder}
Suppose $f : [0,1] \to \C$ is H\"older continuous,
    and that $d_i = f(i/N)$ for $1 \leq i \leq N.$
    Set $M_N=D_N+J$. Then
   \corAB{$L_N$} converges weakly in probability to a probability
   measure with log potential
   \[
     \int_0^1 \bigl(\log |f(t)-z|\bigr)_+\,dt.
   \]
  \end{corollary}
  \corAB{The proof
    follows a very similar track as the
    derivation of Corollary~\ref{cor:iid} from Corollary~\ref{cor:iidssv}. Moreover, as we will see later in Section \ref{sec:pf-pc-constant}, the proof of Theorem \ref{thm:finite-off-diag-pc} \corEE{also follows a very similar line of arguments}. Hence, the proof of Corollary \ref{cor:holder} is omitted.}

\subsection{ Proof of Theorem~\ref{thm:DJssv}:
  estimates for the small singular values of $D_N + J$.}\label{sec:pfDJssv}
The proof is
divided
into three separate claims: \corAB{Corollary~\ref{cor:lb},}
\corOZ{which is a} bound on the $(L+1)$-st smallest singular value,
\corAB{Proposition~\ref{prop:ub},} \corOZ{which is an}
\corAB{upper bound on the product of small singular values,
  and Proposition~\ref{prop:lb},}
\corOZ{which is a lower bound on the product of small singular values.}

Recall
the notation in \eqref{eq:iteration} and \eqref{eq:uniform}, and that
that $\pi_j$ is the coordinate projection from $\C^N$ to the coordinates that support ${\bm w}^j.$  Let $\corAB{\rho_j}$ be the same coordinate projection that in addition kills the $i_{j+1}$ coordinate, i.e.\,the last entry of the support of ${\bm w}^j$.
\begin{lemma}
  For all $1\leq j \leq \corAB{L},$ and any vector ${x}\,  \corAB{\in \C^N}$,
  \[
    \frac{
      \inf_{a \in \C} \|\pi_j({x} - a{\bm w}^{j})\|_2
    }{
      \| \corAB{\rho_j}( M_N {x})\|_2
    }
    \leq
    \min\left\{
      \mathcal{D}_{+}^{i_j,i_{j+1}},
      \mathcal{D}_{-}^{i_j,i_{j+1}}
    \right\},
  \]
  where $\corAB{\mathcal{D}_{+}^{i_j,i_{j+1}}}$ and $\corAB{\mathcal{D}_{-}^{i_j, i_{j+1}}}$ were defined in \eqref{eq:ddef}.
\label{lem:localstructure}
\end{lemma}

\begin{proof}[Proof of Lemma~\ref{lem:localstructure}]
  By definition of $M_N,$ we have for any $1 \leq p < N$ and any vector ${x}\in \C^N,$
  \[
    {x}_{p+1} = d_p {x}_p \corABrev{-} (M_N {x})_p.
  \]
  By iterating this identity and using \eqref{eq:MNv}, we have for $i_{j+1} \geq p > k \geq i_j + 1$
  \[
    {x}_p - a{\bm w}^{j}_p
    = ({x}_{k} - a{\bm w}^{j}_{k})\mathcal{D}^{k,p}
    +\sum_{r=k}^{p-1} (M_N {x})_r \mathcal{D}^{r+1,p}.
  \]
  Reversing the roles of $k$ and $p$ and rearranging the formula,
  this also shows that for $i_{j+1} \geq k > p \geq i_j + 1$
  \[
    {x}_p - a{\bm w}^{j}_p
    =
    \frac{
      ({x}_{k} - a{\bm w}^{j}_{k})
    }{
    \mathcal{D}^{p,k}
    }
    \corABrev{+} \sum_{r=p}^{k-1} \frac{(M_N {x})_r}{\mathcal{D}^{p,r+1}}.
  \]

  Picking \corO{$a$ so that $({x}_{i_j+1} - a{\bm w}^{j}_{i_j+1})=0$ or $({x}_{i_{j+1}} - a{\bm w}^{j}_{i_{j+1}})=0$} we have that
  \begin{equation}
    |{x}_p - \corO{a}{\bm w}^{j}_p|
    \leq
    \begin{cases}
      \sum_{r=i_j+1}^{p-1} |(M_N {x})_r| \cdot |\mathcal{D}^{r+1,p}|, & p > i_j+1, \\
      \sum_{r=p}^{i_{j+1}-1} |(M_N {x})_r| \cdot |\mathcal{D}^{p,r+1}|^{-1}, & p < i_{j+1}-1. \\
    \end{cases}
    \label{eq:lociterate}
  \end{equation}
  Hence, \corAB{using the first inequality of \eqref{eq:lociterate},} upon applying Jensen's inequality,
  \[
    |{x}_p - \corO{a{\bm w}^{j}_p}|^2
    \leq
    \biggl\{
      \sum_{r=i_j+1}^{p-1} |(M_N {x})_r|^2 \cdot |\mathcal{D}^{r+1,p}|
    \biggr\}
    \cdot
  \biggl\{
    \sum_{r=i_j+1}^{p-1} |\mathcal{D}^{r+1,p}|
  \biggr\}
  \]
  Summing this bound from $p=i_j+1$ to $p = i_{j+1}$, \corAB{rearranging the terms}, and using the definition of $\mathcal{D}_+^{i_j,i_{j+1}},$
  \[
    \|\pi_j({x}-a{\bm w}^j)\|_2^2=
    \sum_{p=i_j+1}^{i_{j+1}}
    |{x}_p - \corO{a{\bm w}^{j}_p}|^2
    \leq
    \biggl\{
      \sum_{r=i_j+1}^{i_{j+1}-1} |(M_N {x})_r|^2
    \biggr\}
    \cdot
    \biggl\{
      \mathcal{D}_+^{i_j,i_{j+1}}
    \biggr\}^2=
    \|\corAB{\rho_j}(M_N {x})\|_2^2
    \cdot
    \biggl\{
      \mathcal{D}_+^{i_j,i_{j+1}}
    \biggr\}^2 .
  \]
  \corAB{Next using the second inequality of \eqref{eq:lociterate} and proceeding similarly as above we} complete the proof.
\end{proof}

We now proceed to using these estimates in order to
control the product of the small singular values of $M_N$. \corAB{We begin with obtaining an upper bound on the product of small singular values. To this end,
  we use Lemma \ref{lem:prod_sing}, which is a multivariate generalization of Courant-Fischer-Weyl min-max principle for singular values.}

\begin{proposition}
  \begin{equation}
    \label{eq-fresh1}
  \prod_{k=0}^{\corAB{L}-1} \sigma_{N-k}(M_N) \leq
  \prod_{k=\corO{1}}^{\corAB{L}}
    \|\pi_{\corO{k}} M_N {\bm w}^{\corO{k}}\|_2.
  \end{equation}
  \label{prop:ub}
\end{proposition}
\begin{proof}
  \corAB{Denote ${\bm W}:=\begin{bmatrix}{\bm w}^1 & {\bm w}^2 & \cdots & {\bm w}^L\end{bmatrix}$. Since the columns of ${\bm W}$ are orthonormal, from Lemma \ref{lem:prod_sing} it follows that
\begin{equation}\label{eq:prodsing-d1-ubd}
\prod_{k=0}^{{L}-1} \sigma_{N-k}(M_N) \le \det({\bm W}^* M_N^* M_N {\bm W}).
\end{equation}
To evaluate the \abbr{RHS} of \eqref{eq:prodsing-d1-ubd} we define the $L \times L$} matrix $\mathcal{M}$ by
  \[
    \mathcal{M}_{j,k} := (e_{i_{j+1}})^t M_N {\bm w}^k,
  \]
  \corABrev{where $\{e_\ell\}_{\ell=1}^N$ are canonical basis vectors in $\C^N$.}
\corAB{Recalling \eqref{eq:MNv} we note that ${\bm W}^* M_N^* M_N {\bm W}= \mathcal{M}^* \mathcal{M}$. On the other hand, the matrix $\mathcal{M}$ being} upper triangular we also have
  \begin{equation}\label{eq:detMW}
    \corAB{\prod_{j=1}^{\corAB{L}} \|\pi_j M_N {\bm w}^j\|_2^2
    =\prod_{j=1}^{\corAB{L}} |\mathcal{M}_{j,j}|^2
    =|\det \mathcal{M}|^2 = \det({\bm W}^* M_N^* M_N {\bm W}).}
  \end{equation}
  \corOZ{The proof concludes by invoking \eqref{eq:prodsing-d1-ubd}.}
 \end{proof}


 \corAB{We next show that vectors $v$
   which have a sizeable component orthogonal
to $\cS:= \Span\left\{ {\bm w}^j \right\}$ will necessarily
have $\|M_N v\|_2$ large.} To this end, let $\psi$ be the orthogonal projection map from $\C^N$ to \corAB{$\cS$}. Let $\rho$ be the projection
\begin{equation}
  \label{eq-rho}\rho := \sum_{j=1}^\ell \corAB{\rho_j}.
\end{equation}
The projections $\rho$ and $\psi$ interact in that
\begin{equation}
  \rho M_N \psi = 0,
  \label{eq:rhogood}
\end{equation}
which follows immediately from the definition of ${\bm w}^j.$
We can then combine this observation with the earlier
Lemma~\ref{lem:localstructure} to obtain the next lemma.
%
\begin{lemma}
  For any vector $v \in \C^N,$
  \[
    \|v\|^2_2 - \|\psi v\|^2_2
    =
    \|(1-\psi)v\|^2_2
    \leq \| \rho M_N v\|^2_2 \mathfrak{D}^2.
  \]
  \label{lem:orthW2}
\end{lemma}
\begin{proof}
  On the one hand, using the orthogonality of $\pi_j$,
  \begin{align*}
    \|v\|^2_2-\|\psi v\|^2_2
    = \inf_{\{a_j\} \subset \C} \| (v - \sum_j a_j {\bm w}^j)\|^2_2 = \inf_{\{a_j\} \subset \C} \sum_j \| \pi_j(v - a_j {\bm w}^j)\|^2_2 = \sum_j \inf_{a_j \in\C} \| \pi_j(v - a_j {\bm w}^j)\|^2_2.
  \end{align*}
  Hence using Lemma~\ref{lem:localstructure} and \eqref{eq:uniform}, we get
  \[
    \|v\|^2_2-\|\psi v\|^2_2
    \leq \sum_j \|\corAB{\rho_j} M_N v\|^2_2\mathfrak{D}^2.
  \]
 The stated conclusion of the lemma follows
  by the orthogonality of the $\corAB{\rho_j}$.
\end{proof}
The last lemma
immediately implies a lower bound for the $(\corAB{L}+1)$-st smallest singular value.
\begin{corollary}
  \label{cor:lb}
  \[
    \sigma_{N-\corAB{L}}(M_N) \geq \mathfrak{D}^{-1}
  \]
\end{corollary}
\begin{proof}
 We recall the standard variational characterization of this singular value in the maximin form:
\begin{equation*}
  \sigma_{N-\corAB{L}}(M_N) = \sup_{V_{\corAB{L}}} \inf_{x \perp V_{\corAB{L}}} \|M_Nx\|_2,
\end{equation*}
where $V_{\corAB{L}}$ is an $\corAB{L}$-dimensional space and $x$ is of unit $\ell_2$-norm. \corAB{Setting $V_L=\cS$, the stated corollary immediately follows from} Lemma \ref{lem:orthW2}.
\end{proof}
\corAB{Now it remains to find a lower bound on the product of the small singular values;} \corOZ{this is
slightly more involved.}
\begin{proposition}
  \corOZ{With notation as above,}
  \[
    \prod_{k=0}^{\corAB{L}-1} \sigma_{N-k}(M_N)
    \geq
    (\corOZ{8}
    \corAB{(\|M_N\| \vee 1)}\mathfrak{D} \sqrt{\corAB{L}})^{-\corAB{L}}
    \prod_{k=1}^{\corAB{L}} \|\pi_k M_N {\bm w}^k\|_2.
  \]
  \label{prop:lb}
\end{proposition}
\begin{proof}
Using Lemma \ref{lem:prod_sing} we see that it is enough to find a uniform lower bound on
  \(
   \prod_{k=1}^{L}
    \| {M}_{N} v_k\|_2
  \)
  over all collections of orthonormal vectors $\left\{ v_k \right\}_{k=1}^L$. We bound each
  $\| M_N v_k\|_2$ below in one of two ways.
  If $1-\|\psi v_k\|^2_2 \geq \frac{1}{2\corAB{L}}$ then $v_k$ has a large enough component in the \corABrev{$\cS^{\perp}$} direction that we apply Lemma~\ref{lem:orthW2} to conclude
  \begin{equation} \label{eq:slb0}
    \| M_N v_k\|_2
    \geq
    \|\rho M_N v_k\|_2
    \geq
    \frac{1}{\sqrt{2\corAB{L}} \mathfrak{D}}
    \corOZ{\geq \frac{1}{8\mathfrak{D}\sqrt{L} }\frac{\|\pi_k M_N {\bm w}^k\|_2}{\|M_N\|}},
  \end{equation}
  \corOZ{where $\rho$ is as in \eqref{eq-rho}.}
  Without loss of generality, we may permute the ordering of the vectors so that the first $v_1,v_2,\dots,v_p$ are those that satisfy $1-\|\psi v_k\|^2_2 < \frac{1}{2\corAB{L}}.$
  For these vectors, we have that
  \begin{equation}
    \|M_N v_k\|^2_2 =
    \| \rho M_N(1-\psi) v_k\|^2_2
    +\| (1-\rho) M_N v_k\|^2_2,
    \label{eq:srhosplit}
  \end{equation}
  where we have used that $\rho M_N \psi = 0.$  
  \corAB{We now consider two cases. First suppose}
  \[
    \| (1-{\rho}) M_N {\psi} v_k\|_2
    \leq
    \corAB{4} \| (1-{\rho}) M_N (1-{\psi}) v_k\|_2.
  \]
  \corAB{Since
  \[
  \| (1-{\rho}) M_N (1-{\psi}) v_k\|_2 \le \|M_N\| \cdot \|(1-{\psi})v_k\|_2,
  \]
  we obtain
   \[
    \| (1-{\rho}) M_N {\psi} v_k\|_2
        \leq
    4 \|M_N\| \cdot \|(1-{\psi})v_k\|_2.
  \]}
  Hence by \eqref{eq:srhosplit} and Lemma \ref{lem:orthW2} we deduce
  \begin{eqnarray}\label{eq:lb1d1}
    \|M_N v_k\|_2
  &  \geq&
    \| {\rho} M_N(1-{\psi}) v_k\|_2 \, =   \| {\rho} M_N v_k\|_2 \ge \mathfrak{D}^{-1} \|(1-{\psi})v_k\|_2 \nonumber\\
    &\ge& \frac{1}{\corAB{4}\|M_N\| \mathfrak{D}} \| (1-{\rho}) M_N {\psi} v_k\|_2.
  \end{eqnarray}
  On the other hand, if
    \[
    \| (1-{\rho}) M_N {\psi} v_k\|_2
    >
    \corAB{4} \| (1-{\rho}) M_N (1-{\psi}) v_k\|_2,
  \]
  \corAB{then by \eqref{eq:srhosplit} and the triangle inequality we see that
  \begin{align}\label{eq:lb2d1}
    \|M_N v_k\|_2
  & \,   \geq
    \| (1-{\rho}) M_N v_k\|_2\notag\\
    & \,     \geq
   \| (1-{\rho}) M_N \psi v_k\|_2 - \| (1-{\rho}) M_N (1-\psi) v_k\|_2 \ge
    \frac{3}{4}
   \| (1-{\rho}) M_N \psi v_k\|_2.
  \end{align}}
  Hence, combining \eqref{eq:lb1d1}-\eqref{eq:lb2d1}
  \corAB{and noting that $\mathfrak{D} \ge 1$}
  we conclude that \corOZ{in either case},
  \begin{equation}
    \|M_N v_k\|_2
    \geq
     \frac{1}{\corAB{4(\|M_N\|\vee 1)\mathfrak{D}}} \| M_N {\psi} v_k\|_2, \notag
  \end{equation}
 where we have again applied the fact that ${\rho} M_N {\psi} =0$. Thus \corAB{
 \begin{equation}\label{eq:lb-simplifyd1}
\prod_{k=1}^p   \|M_N v_k\|_2 \ge \left(\frac{1}{4(\|M_N\| \vee 1) \mathfrak{D}}\right)^p \prod_{k=1}^p\| M_N {\psi} v_k\|_2.
 \end{equation}
 }
 \corOZ{The rest of the proof boils down to finding}
 \corAB{a lower bound on the \abbr{RHS} of \eqref{eq:lb-simplifyd1}.}
 \corAB{To this end, let $Y_1$ be the matrix whose columns are $\{{\psi}v_k\}_{m=1}^p$. Since the columns of ${\bm W}:=\begin{bmatrix} {\bm w}^1 & {\bm w}^2 & \cdots & {\bm w}^L \end{bmatrix}$ span the subspace $\cS$, there must exist an $L \times p$ matrix $A_1$ such that $Y_1={\bm W} A_1$. We extend the matrix $A_1$ to an $L \times L$ matrix $A$ so that the last $L-p$ columns of $A$ are orthonormal and are also orthogonal to the first $p$ columns of $A_1$. Such an extension is always possible by first extending arbitrarily to a basis and then running Gram-Schmidt on the final $L - p$ columns. Set $Y:={\bm W} A$}
 \corOZ{and denote the columns of $Y$ by ${\bm y}_m$, for $m \in [L]$.}

 \corAB{Turning to bound the \abbr{RHS} of
 \eqref{eq:lb-simplifyd1}, by Hadamard's inequality we now find that
\begin{equation}\label{eq:lb-simplify-1d1}
\prod_{k=1}^p\| {M}_N {\psi} v_k\|_2^2 \ge \frac{\det(Y^* {M}_{N}^* {M}_{N} Y)}{\prod_{k=p+1}^L \|{M}_{N}{\bm y}_k\|_2^2}.
\end{equation}
We separately bound the numerator and the denominator of \eqref{eq:lb-simplify-1d1}.}

\corAB{Note that ${\bm y}_k= \sum_{m=1}^L a_{m,k} {\bm w}^{m}$ where $a_{m,k}$ is the $(m,k)$-th entry of $A$. Since $\{{\bm w}^{m}\}_{m=1}^{L}$ are orthonormal, for $k=p+1,p+2,\ldots,L$, we have
\begin{equation}\label{eq:lb-simplify-2d1}
\|{M}_{N} {\bm y}_k\|_2 \le \|{M}_{N}\| \cdot \|{\bm y}_k\|_2 =  \|{M}_{N}\| \cdot \sqrt{\sum_{m=1}^L |a_{m,k}|^2 \|{\bm w}^{m}\|_2^2}=  \|{M}_{N}\|,
\end{equation}
where in the last step we also use the fact that the last $L-p$ columns of $A$ have unit $\ell_2$-norm. The inequality \eqref{eq:lb-simplify-2d1} takes care of the denominator of \eqref{eq:lb-simplify-1d1}. Thus it remains to find a lower bound of the numerator of \eqref{eq:lb-simplify-1d1}. To obtain such a bound, we observe that
\begin{equation}\label{eq:lb-simplify-3d1}
\det(Y^* {M}_{N}^* {M}_{N} Y)= \det({\bm W}^* {M}_N^* {M}_N {\bm W}) \det(A A^*)= \left[\prod_{j=1}^{\corAB{L}} \|\pi_j M_N {\bm w}^j\|_2^2\right] \cdot \det(A A^*).
\end{equation}
where the last step follows from \eqref{eq:detMW}. It now remains to bound $\det(A A^*)$.}

\corAB{Let $\mathfrak{a}_{m, m'}$ be the $(m,m')$-th entry of $A^*A$. Using the orthonormality of $\{{\bm w}^k\}_{k=1}^L$ we have that for any $1\leq m,m' \leq p,$
\[
  \mathfrak{a}_{m, m'}=    \sum_{r=1}^L \bar{a}_{r,m} a_{r,m'}
  =
  \left[\sum_{r=1}^L a_{r,m} {\bm w}^r\right]^*
  \left[\sum_{r=1}^L a_{r,m'} {\bm w}^r\right]
  = ({\psi} v_m)^*({\psi} v_{m'}).
\]}
\corAB{Since $v_m \perp v_{m'}$, for $m \neq m'$, we see that
\[
  ({\psi} v_m)^*({\psi} v_{m'})
  = - ((\Id-{\psi})v_m)^*( (\Id-{\psi})v_{m'}).
\]
By our construction we have that
\(
  \|{\psi} v_m\|^2_2 > 1-\frac{1}{2L}
\),
for $m=1,2,\ldots,p$. Thus
\[
|\mathfrak{a}_{m,m'}| = |({\psi} v_m)^*({\psi} v_{m'})| \leq
\|(1 - {\psi})v_m\|_2
\|(1 - {\psi})v_{m'}\|_2 \leq \frac{1}{2L},
\]
for all $1 \le m \ne m' \le p$. By a similar reasoning we also obtain that
\[
|\mathfrak{a}_{m,m}| \ge 1 - \frac{1}{2L},
\]
for $m=1,2,\ldots,p$. Since the last $L-p$ columns of $A$ are orthonormal and orthogonal to the first $p$ columns we further obtain
\[
\mathfrak{a}_{m,m} =1, \, m=p+1,p+2,\ldots,L; \ \mathfrak{a}_{m,m'} =0, \, p+1 \le m \ne m' \le L,
\]
and
\[
 \mathfrak{a}_{m,m'}=0, \, m \in [p], m'=p+1,p+2,\ldots, L.
\]
So in the first $p$ rows of the matrix $A^*A$ we find that the diagonal entries are at least $1-\frac{1}{2L}$ and the sum of off-diagonal entries in a row is at most $(p-1)/(2L)$. The last $L-p$ rows of $A^*A$ are simply $\{e_{m'}\}_{m'=p+1}^L$. Hence by the Gershgorin circle theorem all eigenvalues of $A^*A$ are at least $\tfrac{1}{2}$ which implies
\begin{equation}\label{eq:detAA*}
\det(A A^*)=  \det (A^*A)
  \geq 2^{-L}.
\end{equation}}
\corAB{Therefore, from \eqref{eq:lb-simplifyd1}-\eqref{eq:lb-simplify-3d1}, we derive
\[
\prod_{k=1}^p   \|{M}_{N} v_k\|_2 \ge \left(\frac{1}{8 (\|M_N\| \vee 1) }\right)^L \cdot \mathfrak{D}^{-p} \cdot \prod_{k=1}^L \|\pi_k M_N {\bm w}^k\|_2^2.
\]}

\noindent
\corEE{Combining this bound with \eqref{eq:slb0} finishes the proof of the proposition.}
%
\end{proof}
\section{Limiting spectrum of noisy version of banded twisted Toeplitz matrices}
\label{sec:tt}
We consider in this section upper triangular
twisted Toeplitz matrices of finite symbols, namely upper triangular
matrices with a finite number of slowly varying diagonals; a particular
case is the case of upper triangular Toeplitz matrices of finite symbol.

Our main result is the following theorem.
\begin{theorem}\label{thm:finite-off-diag}
Fix $\gd \in \N$, \corOZ{$\alpha_0>1/2$ and $\alpha_\ell\in (0,1]$ for
$\ell \in [\gd]$.}
  For  each $\ell \in [\gd]\cup\{0\}$,
let $f_\ell: [0,1] \mapsto \C$ be an
$\alpha_\ell$-H\"{o}lder continuous function,
\corOZ{and let}
$D_N^{(\ell)}$ be the diagonal matrix with entries
$\{f_\ell(i/N)\}_{i \in [N]}$.  Set
$M_N:= \sum_{\ell=0}^{\gd} D_N^{(\ell)} J^\ell$ and set $\model_N$
as in \eqref{eq-1}. Then $L_N$ converges weakly in
probability to $\mu_{\gd, {\bm f}}$, the law
of  $\sum_{\ell=0}^\gd f_\ell(X) U^\ell$,
where $X \sim {\rm{Unif}}(0,1)$,
$U$ is uniformly distributed
on the unit circle in $\C$, and
$X$ and $U$ are independent of each other.
\end{theorem}
\begin{remark}
  Theorem \ref{theo-1} follows from Theorem \ref{thm:finite-off-diag}
  by taking $f_\ell(\cdot)=a_\ell$.
\end{remark}

\corOZ{Recall
  the notation $\LL_\mu$ for the log-potential of a measure $\mu$,
  see \eqref{eq-logpot}.}
Similar to Theorem \ref{theo-2},
we prove Theorem
\ref{thm:finite-off-diag} by showing that for
Lebesgue a.e.~$z \in \C$, $\LL_{L_N}(z)\to \LL_{\mu_{\gd, {\bm f}}}(z)$
in probability.
Toward this end we begin by identifying
$\mathcal{L}_{\mu_{\gd, {\bm f}}}(\cdot)$.
For $z \in \C $
and  $x \in [0,1]$, introduce the symbol
\begin{equation}
  \label{eq-symbol}
P_{z,x}(\lambda):=P_{z,x,\gd,{\bm f}}(\lambda):= f_\gd(x)\lambda^\gd+ f_{\gd-1}(x)\lambda^{\gd-1}+\cdots+f_1(x)\lambda +f_0(x)- z.
\end{equation}
Let
$\hat{\gd}:=\hat{\gd}(x)$ denote the degree of $P_{z,x}(\cdot)$.
If $\hat{\gd}>0$ then  $P_{z,x}(\cdot)$ has
$\hat{\gd}$ roots
$\lambda_1(z,x),\lambda_2(z,x),\ldots,\lambda_{\hat{\gd}}(z,x)$
(multiplicities allowed).
Partition $[0,1]$ as follows:
for $\ell \in \{1,\ldots,\corABrev{\gd}\}$ set
\begin{equation}
  \label{eq-Part}
  \mathcal{A}_\ell:=\{x \in [0,1]: f_\ell (x) \ne 0\},
\mathcal{B}_\ell:=\mathcal{A}_\ell
\setminus (\cup_{j=\ell+1}^\gd\mathcal{A}_j),
\end{equation}
and in
particular $\mathcal{B}_\gd:=\mathcal{A}_\gd$.
Set $\mathcal{B}_0:=[0,1]
\setminus (\cup_{\ell=1}^\gd \mathcal{B}_\ell)$.
\begin{lemma}\label{lem:limit-log-potential}
For \corABrev{Lebesgue a.e.}~$z \in \C$ we have
\begin{equation}\label{eq:limit-log-pot}
\mathcal{L}_{\mu_{\gd, {\bm f}}}(z)= \sum_{\ell=1}^\gd \left[\int_0^1 \left\{\sum_{j=1}^\ell \log_+ |\lambda_j(z,x)| + \log |f_\ell(x)|\right\} \cdot {\bf 1}_{\mathcal{B}_\ell}(x) \, dx\right] + \int_0^1\log |f_0(x) -z| {\bf 1}_{\mathcal{B}_0}(x) \, dx.
\end{equation}
\end{lemma}
\corABrev{Setting $Y:= \sum_{\ell=0}^\gd f_\ell(X)U^\ell$, we see that the law of $Y$ is compactly supported in $\C$. Therefore, for any $z \in \C$ and $\vep >0$,
\[
\E_{\mu_{\gd, {\bm f}}} \left|\log |Y - z| {\bf 1}_{\{|Y-z| \ge \vep\}} \right| < \infty.
\]
On the other hand from Lemma \ref{lem:regular} we will see that for Lebesgue a.e.~$z \in \C$,
\[
\lim_{\vep \downarrow 0}\E_{\mu_{\gd, {\bm f}}} \left|\log |Y - z| {\bf 1}_{\{|Y-z| \le \vep\}} \right| =0.
\]
Therefore, by Fubini's theorem one can use iterated integrals to evaluate $\mathcal{L}_{\mu_{\gd, {\bm f}}}(z)$, for Lebesgue almost every $z \in \C$. Note that this does not imply the integrability of individual terms under the integral sign in the \abbr{RHS} of \eqref{eq:limit-log-pot}.}
\begin{proof}[Proof of Lemma \ref{lem:limit-log-potential}]
  \corABrev{Following the discussion above we proceed to evaluate $\mathcal{L}_{\mu_{\gd, {\bm f}}}(z)$ using iterated integrals. To this end, we have}
\begin{align*}
\mathcal{L}_{\mu_{\gd, {\bm f}}}(z)
= & \E\left( \log \left|z - \sum_{\ell=0}^\gd f_\ell(X) U^\ell \right|\right)
= \frac{1}{2\pi}\int_{\mathbb{S}^1} \int_0^1 \log |P_{z,x}(\lambda)| \, dx \, d \lambda  \\
= & \sum_{\ell=1}^\gd \int_0^1\left[\frac{1}{2\pi} \int_{\mathbb{S}^1}  \left\{\sum_{j=1}^\ell \log| \lambda - \lambda_j(z,x)|  \right\}  \, d \lambda + \log |f_\ell(x)|\right] {\bf 1}_{\mathcal{B}_\ell}(x) \, dx +  \int_0^1\log |f_0(x) -z| {\bf 1}_{\mathcal{B}_0}(x) \, dx.
\end{align*}
Since
\[
\frac{1}{2\pi}\int_{\mathbb{S}^1} \log |z'- \lambda| d \lambda = \left\{\begin{array}{ll} \log |z'|, & \mbox{if } |z'| \ge 1\\
0, & \mbox{otherwise,}
\end{array}
\right.
\]
the claim follows.
\end{proof}





\subsection{Reduction to piecewise constant $\{f_\ell\}_{\ell=0}^\gd$}
To prove Theorem \ref{thm:finite-off-diag} we adopt a strategy similar
to the proof of Theorem \ref{theo-2}. Namely, we find approximate singular vectors corresponding to small singular values of $M_N - z \Id$ for {\em almost} all $z \in \C$. To this end, note that
\begin{equation}\label{eq:d-lin-eqn}
((M_N-z\Id) w)_i = (\corABrev{f_0(i/N)}-z) w_i + \sum_{\ell=1}^{\gd} \corABrev{f_\ell(i/N)} w_{\ell+i},
\end{equation}
for $i \in [N-\gd]$.
Therefore, given any arbitrary values of
$\{w_\ell\}_{\ell=j}^{j+\gd-1}$ one
can construct a $N$-dimensional vector
$w$ such that $((M_N-z\Id) w)_i = 0$ for $i \in[N-\gd]\setminus[j+\gd-1]$.
Such choices of $w$ will be candidates for approximate singular vectors.
To study these vectors we note from \eqref{eq:d-lin-eqn} that
$(w_{j+\gd}, \ldots,w_{j+1})^{\mathsf{T}}= T_j (w_{j+\gd -1},
\ldots,w_j)^{\mathsf{T}}$ for some transfer matrix $T_j$.
Iterating, we have that
$(w_{N}, \ldots,w_{N-\gd+1})^{\mathsf{T}} =
\left(\prod_k T_k\right) \cdot  (w_{j+\gd -1},
\ldots,w_j)^{\mathsf{T}}$. Unlike Theorem \ref{theo-2},
where the transfer matrices are actually scalars,
here the transfer matrices $\{T_k\}$ are in general non-commuting
if $\{f_\ell\}_{\ell=0}^\gd$ are varying. This complicates the study
of the
approximate singular value vector $w$.

To overcome this difficulty we employ the following two-fold argument.
We introduce a regularized model where the
$\{f_\ell\}_{\ell=0}^\gd$ are piecewise constant
and hence $\{D_N^{(\ell)}\}_{\ell=0}^\gd$ have constant diagonal blocks.
Then,
the transfer matrices $\{T_k\}$ are constant, and hence commute,
within each block.
This will be sufficient to derive the necessary properties of
the small approximate singular vectors,
which in turn allows us to deduce that
if the size\corEE{s} of the blocks are chosen carefully then
the \corAB{empirical spectral distribution (\abbr{ESD})} of the regularized model admits the limit as
described in Theorem \ref{thm:finite-off-diag}. To complete the proof of Theorem \ref{thm:finite-off-diag} we then show that the limits of the \abbr{ESD}s of the regularized model and the original model must be the same.

We now introduce the regularized model.
Let $\{f_\ell\}_{\ell=0}^\gd$ be as in Theorem \ref{thm:finite-off-diag}.
Fix some $\delta_1, \delta_2 \in(0,1)$. For $\ell \in [\gd]\cup\{0\}$,
let $\hat{D}_N^{(\ell)}$ be a diagonal matrix with
$$(\hat{D}_N^{(\ell)})_{i,i}:=f_\ell\left(\frac{\lfloor i N^{\delta_1-1}\rfloor}{N^{\delta_1}}\right)\cdot {\bf 1}\left(\left|f_\ell\left(\frac{\lfloor i N^{\delta_1-1}\rfloor}{N^{\delta_1}}\right)\right| \ge N^{-\delta_2}\right), \quad {i \in [N]},$$
and define the regularized version of $M_N$ as
\begin{equation}
  \label{eq-regularized}
\hat{M}_N:=\sum_{\ell=0}^\gd \hat{D}_N^{(\ell)} J^\ell,
\quad \hat\model_N=\hat M_N+N^{-\gamma} G_N.
\end{equation}
Note that in $\hat{M}_N$  we have an additional truncation ${\bf 1}( |f_\ell (\cdot)| \ge N^{-\delta_2})$. \corABrev{This means that if in a certain block $\{f_\ell\}_{\ell=\gd_\star+1}^\gd$ are smaller than $N^{-\delta_2}$ then in that block $\hat{M}_N$ can be treated as a matrix with $\gd_\star$ non-zero off-diagonal entries. This, in particular, implies that if $\gd_\star=0$ then in that block $\hat M_N$ becomes a diagonal matrix. Furthermore the truncation at $N^{-\delta_2}$ allows to derive bounds on the operator norm of the transfer matrices, which will be later used during the proofs.}

Now we can state our main result for $\hat \model_N$.
Its proof, which is the main technical part of the proof of
Theorem \ref{thm:finite-off-diag}, is deferred to
Section \ref{sec:pf-pc-constant}.
\begin{theorem}\label{thm:finite-off-diag-pc}
  Fix $\gd \in \N$, \corOZ{$\alpha_0>1/2$, $\alpha_\ell\in (0,1]$ for
    $\ell \in [\gd]$ and $\delta_1, \delta_2 \in (0,1/2)$ such that $\max\{\delta_1, \delta_2\} \le
  (\gamma - 1/2)/(20\gd^2)$.}
For  each $\ell \in [\gd]\cup\{0\}$, let $f_\ell: [0,1] \mapsto \C$ be an $\alpha_\ell$-H\"{o}lder continuous function.
Let $\hat \model_N$ be as in
\eqref{eq-regularized}.
Then $L_{\hat{\model}_N}$ converges weakly in probability to
$\mu_{\gd, {\bm f}}$.
\end{theorem}
\subsection{Proof of Theorem
\ref{thm:finite-off-diag} assuming
Theorem \ref{thm:finite-off-diag-pc}.} The proof
is motivated by the proof of \cite[Theorem 4, Theorem 5]{GWZ}
and the replacement principle, which was introduced in
\cite[Theorem 2.1]{TV-AM}. (We will use a version of the
replacement principle from \cite[Lemma 10.1]{BCZ}.)

We begin with some preparatory material.
To apply the replacement principle
we will need the following ``regularity'' property of the limit,
closely related to \cite[Definition 1]{GWZ}.
\begin{lemma}\label{lem:regular}
For Lebesgue almost every $z \in \C$,
\[
  \lim_{\varepsilon \to 0} \E_{\mu_{\gd,{\bm f}}} \left[ \log \left(\left|X - z\right|\right) {\bf 1}_{\{|X - z| \le \varepsilon\}}\right] = 0.
\]
\end{lemma}
\begin{proof}
\corABrev{Applying Tonelli's theorem for any probability measure $\mu$ on $\R$ and $0 \le
a_1 < a_2 <1$,
\[
\log(a_2) \mu((a_1,a_2)) - \int_{a_1}^{a_2} \log(x) d\mu(x) = \int_{a_1}^{a_2} \left[ \int_{x}^{a_2} \frac{1}{t} dt\right] d\mu(x) = \int_{a_1}^{a_2}\frac{\mu((a_1,t))}{t} dt.
\]
As $0 \le a_1 < a_2 < 1$, rearranging the above we obtain}
\[
\int_{a_1}^{a_2} |\log(t)| d\mu(t) \le |\log(a_2)| \mu((0,a_2)) + \int_{a_1}^{a_2} \frac{\mu((0,t))}{t}dt.
\]
Therefore, recalling the definitions of $\mu_{\gd,{\bm f}}$
\corOZ{from the statement of Theorem \ref{thm:finite-off-diag}
and $P_{z,x}(\cdot)$ from \eqref{eq-symbol}}, we have
\begin{eqnarray}
&&\left|\E_{\mu_{\gd,{\bm f}}} \left[ \log \left(\left|X - z\right|\right)
{\bf 1}_{\{|X - z| \le \vep\}}
\right] \right|\nonumber \\
&&\le \,
-\frac{1}{2 \pi}\log \vep \int_0^1 \int_0^{2 \pi}
{\bf 1}_{ \{|P_{z,x}(e^{\mathrm{i} \theta})| \le \vep\}}
\, d\theta \, d x
  + \frac{1}{2 \pi}\int_0^\vep \int_0^1 \int_0^{2 \pi} \frac{
  {\bf 1}_{\{ |P_{z,x}(e^{\mathrm{i} \theta})| \le t\}}}{t} \, d\theta \, d x \, dt\nonumber\\
  &&=:
-\frac{1}{2 \pi}\log \vep \int_0^1 \int_0^{2 \pi}
{\bf 1}_{ \{|P_{z,x}(e^{\mathrm{i} \theta})| \le \vep\}}
\, d\theta \, d x
+\mathrm{Term}_\gd(\vep).
   \label{eq:log-bound-1}
\end{eqnarray}
%
We will show that
\begin{equation}\label{eq:log-bound-2}
\lim_{\vep \to 0} \mathrm{Term}_\gd(\vep)=0.
\end{equation}
This will take care of the second term in
the \abbr{RHS} of \eqref{eq:log-bound-1},
and a similar argument (which we omit)
appl\corEE{ies} to the first term, completing the proof of the lemma.

Turning to prove \eqref{eq:log-bound-2}, fix
$ 1> \upalpha_\gd > \upalpha_{\gd -1} > \cdots >\upalpha_1>0$.
Recalling that $\{\lambda_i(z,x)\}_{i=1}^d$
are the roots of the equation $P_{z,x}(\lambda)=0$
on the set $\mathcal{B}_\gd$, \corOZ{see \eqref{eq-Part},} we write there
$|P_{z,x}(e^{\mathrm{i}\theta})| =
|f_\gd(x)| \cdot \prod_{j=1}^\gd |\lambda_j(z,x) - e^{\mathrm{i}\theta}|$.
\corAB{Splitting the integral into the
  two parts $|f_\gd(x)| \ge t^{\upalpha_\gd}$
  and $|f_\gd(x)| \le t^{\upalpha_\gd}$, we obtain}
\[\mathrm{Term}_\gd(\vep) \le \frac{1}{2 \pi}\int_0^\vep \int_{0}^1 \int_0^{2 \pi}  \frac{\sum_{j=1}^\gd {\bf 1}_{\{ |\lambda_j(z,x)-e^{\mathrm{i}
\theta}| \le t^{\frac{1-\upalpha_\gd}{\gd}}\}}}{t} \, d\theta \,
d x \, dt + \mathrm{Term}_{\gd-1}(\vep),
\]
where
\[
\mathrm{Term}_{\gd-1}(\vep):= \frac{1}{2 \pi}\int_0^\vep
\int_{0}^1 \int_0^{2 \pi}  \frac{{\bf 1}_{\{ |P_{z,x}(e^{\mathrm{i}\theta})|
\le t \}}\cdot {\bf 1}_{\{ |f_\gd(x)| \le t^{\upalpha_\gd}\}}}{t} \, d\theta \, d x \, dt
\]
Since for any $\lambda \in \C$, and $s >0$ sufficiently small,
$\int_0^{2\pi} {\bf 1}_{\{|\lambda - e^{\mathrm{i} \theta}| \le s\}}
 d\theta \leq \corAB{4}s$,
 we conclude that
 \begin{equation}
   \label{eq-dtod-1}
\limsup_{\vep \to 0} \mathrm{Term}_\gd(\vep) \le \limsup_{\vep \to 0} \mathrm{Term}_{\gd-1}(\vep).
\end{equation}
Denote $P^{(\gd)}_{z,x}(\lambda):= P_{z,x}(\lambda) - f_\gd(x) \lambda^\gd$
and, when $f_{\gd-1}(x)\neq 0$,
let $\{\lambda_i^{(\gd)}(z,x)\}_{i=1}^{d-1}$
denote its roots.
Using the triangle inequality, for any $t \in (0,1)$, we further have that
\begin{align*}
  &\{ |P_{z,x}(e^{\mathrm{i}\theta})| \le t \}\cap
  \{ |f_\gd(x)| \le t^{\upalpha_\gd}\}  \subseteq
  \{|P^{(\gd)}_{z,x}(e^{\mathrm{i}\theta})| \le
  2t^{\upalpha_\gd} \}\\
  & =
\left(\{|P^{(\gd)}_{z,x}(e^{\mathrm{i}\theta})| \le
2t^{\upalpha_\gd}\}\cap \{|f_{\gd-1}(x)| \ge t^{\upalpha_{\gd-1}}\}\right)
\bigcup
\left(\{|P^{(\gd)}_{z,x}(e^{\mathrm{i}\theta})| \le 2t^{\upalpha_\gd}\}
\cap \{|f_{\gd-1}(x)| \le t^{\upalpha_{\gd-1}}\}\right).
\end{align*}
Therefore arguing as in the lines leading to \eqref{eq-dtod-1} we obtain
\begin{eqnarray*}
&& \limsup_{\vep \to 0} \mathrm{Term}_{\gd-1}(\vep) \\
&&\le  \,  \limsup_{\vep \to 0} \frac{1}{2 \pi}\int_0^\vep
\int_{0}^1 \int_0^{2 \pi}  \frac{\sum_{j=1}^{\gd-1} {\bf 1}_{\{
 |\lambda_j^{(\gd)}(z,x)-e^{\mathrm{i} \theta}| \le
 2^{\frac{1}{\gd-1}}t^{\frac{\upalpha_\gd-\upalpha_{\gd-1}}{\gd-1}}\}}}{t} \, d\theta \, d x \, dt + \limsup_{\vep \to 0}\mathrm{Term}_{\gd-2}(\vep)\\
&&\le  \, \limsup_{\vep \to 0}\mathrm{Term}_{\gd-2}(\vep),
\end{eqnarray*}
where
\[
\mathrm{Term}_{\gd-2}(\vep) :=
\frac{1}{2 \pi}\int_0^\vep \int_{0}^1 \int_0^{2 \pi}
\frac{{\bf 1}_{\{ |P_{z,x}^{(\gd)}(e^{\mathrm{i}\theta})|
\le 2t^{\upalpha_\gd} \}}\cdot {\bf 1}_{\{ |f_{\gd-1}(x)| \le t^{\upalpha_{\gd-1}}\}}}{t} \, d\theta \, d x \, dt.
\]
Iterating the above argument and using induction we deduce that
\begin{equation}\label{eq:log-bound-3}
\limsup_{\vep \to 0} \mathrm{Term}_\gd(\vep) \le \limsup_{\vep \to 0}
\int_0^\vep \int_0^1 \frac{{\bf 1}_{\{ |f_0(x)-z| \le (\gd+1)t^{\upalpha_1}
\}}}{t} \, dx \, dt.
\end{equation}
Since $f_0$ is an $\alpha_0$-H\"{o}lder-continuous function
with $\alpha_0 >1/2$, we have that
$f_0([0,1]):= \{f_0(x): x \in [0,1]\} $ has zero Lebesgue measure (in $\C$),
by a volumetric argument.
Moreover the set $f_0([0,1])$ being a closed set, for every $z \notin f_0([0,1])$ we have $\mathrm{dist}(z, f_0([0,1])) := \inf_{x \in [0,1]} |z -f_0(x)| >0$. Therefore,
given a $z \notin f_0([0,1])$ there exists $\vep>0$ sufficiently small such that $\{x \in [0,1]: |f_0(x) -z| \le (\gd+1) \vep^{\upalpha_1}\}=\emptyset$. Thus, from \eqref{eq:log-bound-3} we deduce \eqref{eq:log-bound-2}. This completes the proof of the lemma.
\end{proof}
To prove Theorem \ref{thm:finite-off-diag} we need another ingredient.
\begin{lemma}\label{lem:e-sing-conv}
Fix $z \in \C$. Let $\tilde{\nu}_{\model_N}^z$ be the
symmetrized version of empirical measure of
the singular values of $\model_N - z \Id$.
Define $\tilde{\nu}_{\hat\model_N}^z$ similarly.
Then both $\tilde{\nu}_{\model_N}^z$ and $\tilde{\nu}_{\hat\model_N}^z$
converge weakly in probability, as $N\to\infty$,
to the symmetrized version $\tilde \nu^z$ of the law of
$|\sum_{\ell=0}^\gd f_\ell(X) U^\ell -z|$, where $X$ and $U$ are as in  Theorem \ref{thm:finite-off-diag}.
\end{lemma}
\begin{proof}[Proof of Lemma \ref{lem:e-sing-conv}]
  \corABrev{For any probability measure $\mu$ supported on $[0,\infty)$, we let $\tilde \mu$denote  its symmetrized version, given by $\tilde \mu((-y,-x)) = \tilde \mu((x,y)) = \frac{1}{2}\mu((x,y))$ for any $ 0 < x < y < \infty$. We let} $\tilde \nu^z_{M_N}$ and
  $\tilde \nu^z_{\hat M_N}$  be the symmetrized versions of 
  the empirical measures \corABrev{of the singular values} of $M_N-z\Id$ and of $\hat M_N-z\Id$, respectively.
  Note first that
$\|N^{-\gamma} G_N\|\to_{N\to\infty} 0$ in probability (and in fact,
a.s.) and that there exists a constant $C=C(\{f_\ell\}, z)$ so that
$\|M_N-z\Id\|\leq C$,
$\|\hat M_N-z\Id\|\leq C$. Therefore,
it follows from Weyl's inequalities that
for any metric on the space of probability measures on $\R$
compatible with the weak topology,
$$d(\tilde \nu^z_{\model_N},\tilde \nu^z_{M_N})\to_{N\to\infty} 0,
\quad
d(\tilde \nu^z_{\hat \model_N},\tilde \nu^z_{\hat M_N})\to_{N\to\infty} 0,$$
in probability.
On the other hand,
by definition,
$\|M_N -\hat{M}_N\| \le N^{-\min\{\alpha \delta_1, \delta_2\}}$,
where $\alpha:=\min_{\ell=0}^\gd\{\alpha_\ell\}$ and therefore, \corABrev{by
Weyl's inequalities,}
$$d(\tilde \nu^z_{M_N},\tilde \nu^z_{\hat M_N})\to_{N\to\infty} 0.$$
Combining the last two displays, we deduce
that it is enough to show that $\tilde{\nu}_{M_N}^z$ converges
weakly to $\tilde\nu^z_{\bm f}$, the symmetrized version of the law of
$|\sum_{\ell=0}^\gd f_\ell(X) U^\ell -z|$.
To this end, we will employ the method of moments.
We will show that for every $k \in \N$,
\begin{equation}\label{eq:mom}
\lim_{N \to \infty} \frac{1}{N} \tr \left[ (M_N - z\Id)(M_N - z\Id)^*\right]^k = \E \left[ \left|\sum_{\ell=0}^\gd f_\ell(X) U^\ell -z\right|^{2k}\right].
\end{equation}
This will complete the proof.
Since we can absorb $z$ in $f_0(\cdot)$ it is enough to prove \eqref{eq:mom} only for $z =0$.

We begin by evaluating
the \abbr{RHS} of \eqref{eq:mom} for $z=0$. As $U^{-1}=U^*$, we have
\begin{equation}\label{eq:mom-limit-1}
\left|\sum_{\ell=0}^\gd f_\ell(X) U^\ell \right|^{2k}= \sum_{\ell=-k\gd}^{k\gd} F_\ell^{(k)}(X) U^\ell,
\end{equation}
for some functions $\{F_\ell^{(k)}(\cdot)\}$. Since $\E U^\ell=\E U^{-\ell}=0$
for any $0\neq  \ell \in \N$, we get
\begin{equation}\label{eq:mom-limit-2}
\E\left[\left|\sum_{\ell=0}^\gd f_\ell(X) U^\ell \right|^{2k}\right] = \E[F_0^{(k)}(X)] = \int_0^1 F_0^{(k)}(x) dx.
\end{equation}
Expanding the sum in the \abbr{LHS} of \eqref{eq:mom-limit-1} and collecting the coefficient of $U^0$ it follows that
\begin{equation}\label{eq:mom-limit-3}
F_0^{(k)}(x)= \sum_{\{\ell_i\}_{i=1}^{2k}: \sum_{i=1}^k (\ell_{2i-1}- \ell_{2i})=0} \prod_{i=1}^k f_{\ell_{2i-1}}(x) \cdot \prod_{i=1}^k \overline{f_{\ell_{2i}}(x)}.
\end{equation}
Turning to identifying the \abbr{LHS} of  \eqref{eq:mom} we see that
\begin{equation}
  \label{eq-BCMD}
(M_N M_N^*)^k = \sum_{i=1}^k \sum_{\ell_i=0}^\gd D_N^{(\ell_1)}J^{\ell_1} (J^*)^{\ell_2} \overline{D_N^{(\ell_2)}} D_N^{(\ell_3)} \cdots  D_N^{(\ell_{2k-1})}J^{\ell_{2k-1}} (J^*)^{\ell_{2k}} \overline{D_N^{(\ell_{2k})}}.
\end{equation}
As $\{D_N^{(\ell)}\}$ are diagonal matrices, \corABrev{using the facts that 
\[
(J^\ell)_{i,j} = {\bf 1}_{\{1 \le j = i+\ell \le N\}}, \qquad [(J^*)^\ell]_{i,j} = {\bf 1}_{\{1 \le j = i-\ell \le N\}}, \quad i,j \in [N],
\]
we have that 
\begin{multline}\label{eq:expand}
\Delta ({\bm \ell})_{i,i} = (D_N^{(\ell_1)})_{i,i} \cdot \overline{(D_N^{(\ell_2)})_{i+s_1,i+s_1}} \cdot {(D_N^{(\ell_3)})_{i+s_1,i+s_1}} \cdots (D_N^{(\ell_{2k-1})})_{i+s_{k-1},i+s_{k-1}} \cdot \overline{(D_N^{(\ell_{2k})})_{i+s_k,i}} \\
\cdot \prod_{j=1}^k {\bf 1}_{\{i+s_{j-1}+\ell_j \in [N]\}} \cdot {\bf 1}_{\{i+s_j \in [N]\}},
\end{multline}
where
\[
\Delta({\bm \ell})_{i,i}:=\left(D_N^{(\ell_1)}J^{\ell_1} (J^*)^{\ell_2} \overline{D_N^{(\ell_2)}} D_N^{(\ell_3)} \cdots  D_N^{(\ell_{2k-1})}J^{\ell_{2k-1}} (J^*)^{\ell_{2k}} \overline{D_N^{(\ell_{2k})}}\right)_{i,i},
\]
and for $j \in [k]$ $s_j:=\sum_{i=1}^j (\ell_{2i-1} -\ell_{2i})$ and $s_0:=0$. Using the fact that $\{D_N^{(\ell)}\}$ are diagonal matrices again we deduce from \eqref{eq:expand} that if $s_k \ne 0$ then $\Delta({\bm \ell})_{i,i} =0$ for any $i \in [N]$.} 

\corABrev{Thus to establish \eqref{eq:mom} we only need to consider the sum over $\{\ell_i\}_{i=1}^{2k}$ such that $s_k =0$.} Fixing such a sequence of $\{\ell_i\}_{i=1}^{2k}$ we observe that \corAB{for any $i \in [N-2k\gd] \setminus [2k\gd]$}
\begin{equation}\label{eq:mom-limit-4}
\Delta ({\bm \ell})_{i,i} = (D_N^{(\ell_1)})_{i,i} \cdot (D_N^{(\ell_2)})_{i+s_1,i+s_1} \cdot \overline{(D_N^{(\ell_3)})_{i+s_1,i+s_1}} \cdots (D_N^{(\ell_{2k-1})})_{i+s_{k-1},i+s_{k-1}} \cdot (D_N^{(\ell_{2k})})_{i,i},
\end{equation}
\corABrev{and $\Delta({\bm \ell})_{i,i}$ is bounded for the remaining indices in $[N]$}. \corABrev{Since $(D_N^{(\ell)})_{i,i}=f_\ell(i/N)$, for any $\{\ell_i\}_{i=1}^{2k}$ we find that}
\[
\lim_{N \to \infty} \corAB{\frac{1}{N}}\sum_{i=1}^N\Delta({\bm \ell})_{i,i} = \int_0^1 \prod_{i=1}^k f_{\ell_{2i-1}}(x) \cdot \prod_{i=1}^k \overline{f_{\ell_{2i}}(x)} dx.
\]
\corOZ{Substituting} in \eqref{eq-BCMD} and
using \eqref{eq:mom-limit-2}-\eqref{eq:mom-limit-3},
we arrive at \eqref{eq:mom} (with $z=0$). This completes the proof.
\end{proof}
We next introduce further notation. The \textit{Stieltjes transform} of
a probability measure $\mu$ on $\R$
 is defined as
\[
G_\mu(\xi):=\int \frac{1}{\xi-y}\, \mu(dy), \, \xi \in \C\setminus \R.
\]
We use the following standard bounds on the Stieltjes transform, see
\cite[(6)-(8)]{GWZ}, in order to integrate the logarithmic function:
for any $\tau, \varrho >0$, and $a, b \in \R$ such that
$b - a > \varrho$ we have
\begin{equation}\label{eq:pr-stielt-ubd}
\mu([a,b]) \le \int_{a-\varrho}^{b+\varrho} |\Im G_\mu(x + \mathrm{i} \tau)| dx + \frac{\tau}{\varrho},
\end{equation}
and
\begin{equation}\label{eq:pr-stielt-lbd}
\mu([a,b]) \ge \int_{a+\varrho}^{b-\varrho} |\Im G_\mu(x + \mathrm{i} \tau)| dx  - \frac{\tau}{\varrho}.
\end{equation}
We need also the symmetrized form of the Stieltjes transform, as follows.
For a $N \times N$ matrix $C_N$, define
\[
\widetilde{C}_N:= \begin{bmatrix} 0 & C_N \\ C_N^* & 0 \end{bmatrix}
\]
and the Stieltjes transform
\[
  G_{C_N}(\xi):= \frac{1}{2N}\tr \left(\xi- \widetilde{C}_N \right)^{-1}, \, \xi \in \C \setminus \R.
\]
$G_{C_N}(\cdot)$ is the Stieltjes transform
of the symmetrized version of the empirical measure of the singular values of $C_N$. 
Using the resolvent identity, we have that
for two $N \times N$ matrices $C_N$ and $D_N$,
\begin{equation}\label{eq:stieltjes-diff}
  |G_{C_N}(\xi) - G_{D_N}(\xi)| \le \frac{\|C_N -D_N\|}{(\Im (\xi))^2}.
\end{equation}
We are finally ready to prove Theorem \ref{thm:finite-off-diag}.
\begin{proof}[Proof of Theorem \ref{thm:finite-off-diag}]
To show that $L_{\model_N}$ and $L_{\hat{\model}_N}$ admit the same limit \corABrev{we need to show that for every continuous bounded function $f : \C \mapsto \R$
\begin{equation}\label{eq:diff-f}
\int f(z) dL_{\model_N}(z) - \int f(z) dL_{\hat{\model}_N}(z) \to 0,
\end{equation}
as $N \to \infty$, in probability. By continuity we have that
$\max_\ell \sup_{x \in [0,1]} |f_\ell(x)| <\infty$, and then
$\|M_N\|, \|\hat{M}_N\| <\infty$. Therefore, using \eqref{eq-gordon}, we have that
$\|\model_N\|, \|\hat\model_N\| <\infty$ with probability approaching one. Thus, it suffices to show that \eqref{eq:diff-f} holds for compactly supported functions $f$. Furthermore, an application of the Stone-Weierstrass theorem yields that one can restrict attentions to smooth functions, namely it is enough to consider $f \in C_c^2(\C)$.} 

\corABrev{Turning to the proof of \eqref{eq:stieltjes-diff} for such $f$, we use \cite[Lemma 8.1]{BCZ} to note that we need to show that:}
\begin{enumerate}
\item[(i)] The expression
\[
\frac{1}{N}\|\model_N\|_2^2 + \frac{1}{N}\|\hat\model_N\|_2^2
\]
is bounded in probability.
\item[(ii)] For Lebesgue almost all $z \in B_\C(0,R)$,
\[
\frac{1}{N} \log |\det(\model_N - z \Id)| - \frac{1}{N} \log |\det(\hat\model_N - z \Id)|  \to 0,
\]
as $N \to \infty$, in probability.
\end{enumerate}
\vskip5pt
\corABrev{As already noted above, $\|\model_N\|, \|\hat \model_N\| < \infty$ with probability approaching one, condition (i) is immediate.} 


\corABrev{It remains to establish condition (ii).} By the definitions of
$\tilde\nu_{\model_N}^z$ and $\tilde\nu_{\hat\model_N}^z$, we have
\begin{equation}\label{eq:log-diff-1}
\frac{1}{N} \log |\det(\model_N - z \Id)| - \frac{1}{N} \log |\det(\hat\model_N - z \Id)|= \int \log |x| d\tilde\nu_{\model_N}^z(x) - \int \log |x| d\tilde\nu_{\hat\model_N}^z(x).
\end{equation}
\corABrev{We point out to the reader that both $\model_N - z \Id$ and $\hat \model_N - z \Id$ being Gaussian perturbations of some deterministic matrices are non-singular almost surely, and therefore both sides of \eqref{eq:log-diff-1} are well defined on a set with probability one. Now}
from Lemma \ref{lem:e-sing-conv} we have that for any $\vep >0$ and $R_0 \ge 1$,
\begin{equation}\label{eq:log-diff-2}
\int_{(\vep, R_0) \cup (-R_0, - \vep)} \log |x| d\tilde\nu_{\model_N}^z (x) - \int_{(\vep, R_0) \cup (-R_0, - \vep)} \log |x| d\tilde\nu_{\hat\model_N}^z (x) \to 0, \quad \text{ in probability}.
\end{equation}
We observe that
\begin{align*}
& \, \E \int x^2 d\tilde\nu_{\model_N}^z(x)\\
= & \, \E\left[\frac{1}{N}\tr \left\{ (\model_N - z \Id) (\model_N - z \Id)^*\right\}\right] \le \corEE{3} \left[ |z|^2 + \frac{1}{N}\tr (M_N M_N^*) + N^{-(1+2\gamma)} \tr(G_NG_N^*)\right] \le C,
\end{align*}
for some $C <\infty$. Thus using the fact that $|\log |x|| /x^2$ is decreasing for $|x| \ge  \sqrt{e}$, we have that, for any $R_0 >\sqrt{e}$ ,
\[
\E \left[\int_{(-R_0,R_0)^c} \log |x| d \tilde\nu_{\model_N}^z(x)\right] \le \frac{\log R_0}{R_0^2} \E \left[\int_{(-R_0,R_0)^c} x^2 d \tilde\nu_{\model_N}^z(x)\right] \le \frac{\log R_0}{R_0^2}  \cdot C.
\]
By the same argument
\begin{equation}\label{eq:log-near-infty-hat}
\E \left[\int_{(-R_0,R_0)^c} \log |x| d \tilde\nu_{\hat\model_N}^z(x)\right]  \le \frac{\log R_0}{R_0^2}  \cdot C.
\end{equation}
So, applying Markov's inequality we see that for any $\eta>0$, there exists $R_0(\eta)$ such that
\begin{equation}\label{eq:log-diff-3}
\lim_{\eta \to 0} \limsup_{N \to \infty} \P\left( \left| \int_{(-R_0(\eta),R_0(\eta))^c}  \log |x| d \tilde\nu_{\model_N}^z(x) - \int_{(-R_0(\eta),R_0(\eta))^c}  \log |x| d \tilde\nu_{\hat\model_N}^z(x)\right| > \eta\right) =0.
\end{equation}
Combining \eqref{eq:log-diff-1}-\eqref{eq:log-diff-3} we see that to establish condition (ii) it remains to show that for any $\eta >0$, there exists $\vep(\eta)$ such that
\begin{equation}\label{eq:log-diff-4.5}
\lim_{\eta \to 0} \limsup_{N \to \infty} \P\left( \left| \int_{-\vep(\eta)}^{\vep(\eta)}  \log |x| d \tilde\nu_{\model_N}^z(x) - \int_{-\vep(\eta)}^{\vep(\eta)}  \log |x| d \tilde\nu_{\hat\model_N}^z(x)\right| > 2\eta\right) =0. 
\end{equation}
To prove the above it is enough to show that
\begin{equation}\label{eq:log-diff-4}
\lim_{\eta \to 0} \limsup_{N \to \infty} \P\left( \left| \int_{-\vep(\eta)}^{\vep(\eta)}  \log |x| d \tilde\nu_{\model_N}^z(x)\right| > \eta\right) =0
\end{equation}
and
\begin{equation}\label{eq:log-diff-5}
\lim_{\eta \to 0} \limsup_{N \to \infty} \P\left( \left| \int_{-\vep(\eta)}^{\vep(\eta)}  \log |x| d \tilde\nu_{\hat\model_N}^z(x)\right| > \eta\right) =0,
\end{equation}
\corABrev{where $\vep(\eta)$ is as in \eqref{eq:log-diff-4.5}.}
It follows from
Theorem \ref{thm:finite-off-diag-pc} that for Lebesgue
almost every $z\in \C$,
$\LL_{L_{\hat\model_N}}(z)\to_{N\to\infty} \LL_{\mu_{d,{\bm f}}}\corABrev{(z)}$
in probability,
which is equivalent to the statement that
\begin{equation}\label{eq:log-pot-hat}
\int \log|x| d\tilde\nu^z_{\hat\model_N}(x) \to \int \log|x| d\corAB{\tilde\nu^z_{\bm f}}(x), \quad \text{ in probability}.
\end{equation}
\corABrev{Applying Lemma \ref{lem:e-sing-conv} we have that
\begin{equation}\label{eq:log-pot-hat-1}
\int_{(-\vep, \vep)^c \cap (-R_0,R_0)} \log |x| d\tilde\nu^z_{\hat\model_N}(x) \to \int_{(-\vep, \vep)^c \cap (-R_0,R_0)} \log |x| d\corAB{\tilde\nu^z_{\bm f}}(x), \quad \text{ in probability}, 
\end{equation}
for any $\vep >0$ and $R_0 < \infty$. As $\corAB{\tilde\nu^z_{\bm f}}$ is compactly supported using \eqref{eq:log-near-infty-hat} we obtain that
\begin{align*}
\, &\lim_{\eta \to 0} \limsup_{N \to \infty} \P\left( \left| \int_{(-R_0(\eta),R_0(\eta))^c}  \log |x| d \tilde\nu_{\hat \model_N}^z(x) - \int_{(-R_0(\eta),R_0(\eta))^c}  \log |x| \tilde\nu^z_{\bm f}(x)\right| > \eta\right) \\
=\, & \lim_{\eta \to 0} \limsup_{N \to \infty} \P\left( \left| \int_{(-R_0(\eta),R_0(\eta))^c}  \log |x| d \tilde\nu_{\hat \model_N}^z(x) \right| > \eta \right)=0,
\end{align*}
where $R_0(\eta)$ is as in \eqref{eq:log-diff-3}. Therefore, from \eqref{eq:log-pot-hat}-\eqref{eq:log-pot-hat-1} we now deduce that, for every $\vep >0$,
\[
\int_{-\vep}^{\vep} \log|x| d\tilde\nu^z_{\hat\model_N}(x) \to \int_{-\vep}^{\vep} \log|x| d\corAB{\tilde\nu^z_{\bm f}}(x), \quad \text{ in probability}.
\]
This together with Lemma \ref{lem:regular} yields \eqref{eq:log-diff-5}. It remains to establish \eqref{eq:log-diff-4}.}


To this end,
from \cite[Proposition 16]{GKZ} we have that, for any $t >0$,
\[
\P(\sigma_N(\model_N -z \Id) < t) \le C_0 N^{1+2\gamma} t^2,
\]
for some constant $C_0$. Therefore, \corABrev{for any $\eta>0$},
\begin{equation}\label{eq:log-diff-6}
\limsup_{N \to \infty} \P\left( \left|\int_{-N^{-(1+\gamma)}}^{N^{-(1+\gamma)}} \log |x| d\tilde\nu_{\model_N}^z(x) \right| \ge \eta/2\right) \le \limsup_{N \to \infty} \P(\sigma_N(\model_N - z\Id) < N^{-(1+\gamma)}) =0.
\end{equation}
Since $\|\model_N - \hat\model_N\| \le N^{-\delta'}$, where $\delta':=\min\{\alpha \delta_1,\delta_2\}$, from \eqref{eq:stieltjes-diff} and setting $\tau=N^{- \delta'/4}$ we obtain
\[
|G_{\model_N - z\Id}(x+\mathrm{i} \tau) - G_{\hat\model_N - z\Id}(x+\mathrm{i} \tau)| \le N^{-\delta'/2}.
\]
Setting $\varrho = N^{-\delta'/8}$, $\kappa=\delta'/16$, and using \eqref{eq:pr-stielt-ubd}-\eqref{eq:pr-stielt-lbd} in the second inequality,
we have that
\begin{align}
-\int_{N^{-(1+\gamma)}}^{N^{-\kappa}} \log (x) d\tilde\nu_{\model_N}^z(x) & \le (1+\gamma) \log N \cdot \nu_{\model_N}^z([N^{-(1+\gamma)}, N^{-\kappa}])\notag\\
 & \le (1+\gamma) \log N \left(2 N^{- \delta' /8} + \tilde \nu_{\hat
 \model_N}^z([N^{-(1+\gamma)}- 2\corAB{\varrho}, N^{-\kappa}+2\corAB{\varrho}]) \right)\notag\\
 & \le  (1+\gamma) \log N \left(2 N^{- \delta' /8} +
 2\tilde \nu_{\hat
 \model_N}^z([0, 2N^{-\kappa}]) \right) \notag\\
 & \le 2 (1+\gamma) \log N \cdot N^{-\delta' /8}-\frac{4(1+\gamma)}{\kappa} \int_0^{2N^{-\kappa}} \log (x) d\tilde
 \nu_{\hat\model_N}^z(x), \label{eq:log-near-0-1}
\end{align}
for all large $N$, where in the third inequality we used the symmetry of $\tilde \nu_{\hat \model_N}^z$
and $\varrho= o(N^{-\kappa})$.

It remains to bound the integral of $\log(\cdot)$ over $(N^{-\kappa}, \vep)$.
Toward this,
using integration by parts we note that, for $0 \le a_1 < a_2  <1$ and any probability measure $\mu$ on $\R$,
\begin{equation}\label{eq:log-by-parts}
-\int_{a_1}^{a_2} \log (x) d\mu(x) = -\log(a_2) \mu([a_1,a_2]) + \int_{a_1}^{a_2} \frac{\mu([a_1,t])}{t} dt.
\end{equation}
Arguing as in \eqref{eq:log-near-0-1} we obtain
\begin{align}
\int_{N^{-\kappa}}^\vep \frac{\tilde
  \nu_{\model_N}^z([N^{-\kappa}, t])}{t} dt& \le 2N^{- \delta' /8}
  \int_{N^{-\kappa}}^\vep \frac{1}{t} dt +
  \int_{N^{-\kappa}}^\vep \frac{\tilde \nu_{\hat\model_N}^z([N^{-\kappa}/2, t+N^{-\kappa}])}{t} dt \notag\\
& \le 2\kappa N^{-\delta' /8} \cdot \log N +2 \int_{N^{-\kappa}/2}^{2\vep} \frac{\tilde \nu_{\hat\model_N}^z([N^{-\kappa}/2, t])}{t} dt,  \label{eq:log-near-0-2}
\end{align}
where in the last step we have used the fact that $t + N^{-\kappa} \le 2t$ for any $t \ge N^{-\kappa}$, and a change of variable. A similar reasoning yields that
\begin{equation}
-\log(\vep) \tilde \nu_{\model_N}^z([N^{-\kappa},\vep]) \le - \log(\vep) \left( \corAB{2 N^{-\delta'/8}} +
\tilde \nu_{\model_N}^z([N^{-\kappa}/2,2\vep])\right). \label{eq:log-near-0-3}
\end{equation}
Thus combining \eqref{eq:log-near-0-1}, and \eqref{eq:log-near-0-2}-\eqref{eq:log-near-0-3}, and using \eqref{eq:log-by-parts} we deduce that for $\vep >0$ sufiiciently small and all large $N$,
\[
-\int_{N^{-(1+\gamma)}}^{\vep} \log (x) d\tilde
\nu_{\model_N}^z(x)  \le C_0' \left[ \log N \cdot N^{- \delta'/8}-\int_0^{2\vep}\log(x) d\tilde \nu_{\hat\model_N}^z(dx)\right],
\]
where $C_0'$ is some large constant. Now, combining this with
\eqref{eq:log-diff-5}-\eqref{eq:log-diff-6}, the claim in \eqref{eq:log-diff-4} follows. This completes the proof of the theorem.
\end{proof}

\section{The piecewise constant case - proof of Theorem \ref{thm:finite-off-diag-pc}}\label{sec:pf-pc-constant}
\corAB{Similar to the proof of Theorem \ref{theo-2}, the main step in the proof of Theorem \ref{thm:finite-off-diag-pc}
is the proof of convergence of the log-potentials $\LL_{L_{\hat\model_N}}(z)$,
which will use}
\corAB{Theorem \ref{thm:logdet}. To verify the assumptions of Theorem
\ref{thm:logdet} we need to find an analogue of Theorem \ref{thm:DJssv}. To this end, we need to}
\corAB{identify approximate singular vectors  of}
\corOZ{$\hat{M}_N(z):=\hat{M}_N - z\Id$}
\corAB{corresponding to small singular values, establish
that $\|\hat{M}_N(z) w\|_2$ cannot be small for any vector
$w$ orthogonal to these approximate singular vectors, and obtain matching upper and lower bounds,
up to sub-exponential factors,
on the product of the small singular values. Overall, we follow a scheme similar to the one in Section \ref{sec-d=1}. However, as we will see below, some significant changes are necessary}
\corOZ{when treating}
\corAB{the case $\gd >1$, even in the constant diagonal set-up.}

\corOZ{We begin by fixing additional notation.}
\corOZ{Set}
\begin{equation}
  \label{eq-L0}
\gb_k:=\{i \in [N]:
  \lfloor i N^{\delta_1-1}\rfloor =k\}, \quad
  \corOZ{k=0,1,\ldots,L_0:=\lfloor N^{\delta_1}\rfloor.}
\end{equation}
Note that
  if $i \in \gb_k$, for some $k$, then
  for any $\ell \in [\gd] \cup \{0\}$,
  $(\hat{D}^{(\ell)}_N)_{i,i}=
  f_\ell(k N^{-\delta_1}){\bf 1}_{\{|f_\ell(k N^{-\delta_1})| \ge
  N^{-\delta_2}\}}=:t_\ell(k)$, i.e. the diagonals of $\hat M_N$ are constant
  within each $\gb_k$.
  Therefore, for any $w \in \C^N$ and $i \in \gb_k$,
\begin{equation}\label{eq:d-lin-eqn-1}
(\hat{M}_N(z) w)_i = (t_0(k)-z) w_i + \sum_{\ell=1}^{\gd} t_\ell(k) w_{\ell+i}.
\end{equation}
Using \eqref{eq:d-lin-eqn-1} we will construct vectors $w$ for which $\hat{M}_N(z) w \approx 0$. It will be easier to reformulate \eqref{eq:d-lin-eqn-1} as a system of linear equations, for which we need to define the following notion of transfer matrix.  
\begin{dfn}\label{dfn:transfer-matrix}
Fix $k \in \N \cup \{0\}$ such that $\gb_k \ne \emptyset$. Denote
\[
\hat\gd_N:=\hat\gd_N(k):= \max\{\ell: |t_\ell(k)| \ne 0\}.
\]
For $\hat\gd_N(k) >0$ denote
\[
\hat{\bm t}(k):= \left(\frac{t_{\hat\gd_N-1}(k)}{t_{\hat\gd_N}(k)}, \, \frac{t_{\hat\gd_N-2}(k)}{t_{\hat\gd_N}(k)}, \, \ldots, \, \frac{t_{1}(k)}{t_{\hat\gd_N}(k)}\right)
\]
and define the following $\hat\gd_N \times \hat\gd_N$ matrix:
\begin{align*}
T_k(z):= \begin{bmatrix} -\hat{\bm t}(k) & \frac{z-t_0(k)}{t_{\hat\gd_N}(k)}\\ I_{\hat\gd_N-1} & {\bf 0}_{\hat\gd_N-1}  \end{bmatrix},
\end{align*}
where ${\bf 0}_{n}$ is the $n$-dimensional vector of all zeros.
\end{dfn}

\corOZ{Recall the symbol $P_{z,x}$, see \eqref{eq-symbol}.}
The next result shows that the eigenvalues of the transfer matrix $T_k(z)$ are the roots of the equations $P_{z,x}(\lambda)=0$ for some appropriate choices of $x$.
\begin{lemma}\label{lem:transfer-matrix-spectrum}
Fix $k \in \N \cup \{0\}$ such that $\gb_k \ne \emptyset$. Let $\hat\gd_N$ and $T_k(z)$ be as in Definition \ref{dfn:transfer-matrix}. \corABrev{Assume $\hat\gd_N >0$}. Fix any $\ell \in [\hat\gd_N]$. Denote
\[
{\bm v}_{\ell}(k)^{\sf T}:= \begin{pmatrix} \hat{\lambda}^{
  \hat\gd_N-1}_\ell(z,k) & \hat{\lambda}_\ell^{\hat\gd_N-2}(z,k) & \cdots  & 1\end{pmatrix},
\]
where $\{\hat{\lambda}_\ell(z,k)\}_{\ell=1}^{\hat\gd_N}$ are the roots of the equation
\[\hat{P}_{z,k}(\lambda):=t_\gd(k)\lambda^\gd+ t_{\gd-1}(k) \lambda^{\gd-1}+ \cdots + t_1(k)\lambda + t_0(k)-z=0.
\]
Then for any $\ell \in[\hat\gd_N]$, ${\bm v}_\ell(k)$ is an eigenvector of $T_k(z)$ corresponding to $\hat{\lambda}_\ell(z,k)$.
\end{lemma}
\begin{proof}
Since $\{\hat{\lambda}_\ell(z,k)\}_{\ell=1}^{\hat\gd_N}$ are the roots of the equation $\hat{P}_{z,k}(\lambda)=0$, from the definition of \corAB{$T_k(z)$} we note that
\[
T_k(z) {\bm v}_\ell(k) = \begin{pmatrix} \hat{\lambda}_\ell^{\hat\gd_N}(z,k) - \frac{1}{t_{\hat\gd_N}(k)}\hat{P}_{z,k}(\hat{\lambda}_\ell(z,k)) \\ \hat{\lambda}_\ell^{\hat\gd_N-1}(z,k) \\ \vdots \\ \hat{\lambda}_\ell(z,k)\end{pmatrix} = \hat{\lambda}_\ell(z,k) \cdot {\bm v}_\ell(k).
\]
This completes the proof.
\end{proof}
Lemma \ref{lem:transfer-matrix-spectrum} shows that the eigenvalues of $T_k(z)$ are the roots of the polynomial equation $\hat{P}_{z,k}(\lambda)=0$. If $\min_\ell \{|f_\ell(x)|\} \ge N^{-\delta_2}$, then it is easy to see that $P_{z,x}(\lambda)=\hat{P}_{z,k}(\lambda)$ where $x=k N^{-\delta_1}$. Therefore, in such cases, the roots of $P_{z,x}(\lambda)$ and $\hat{P}_{z,k}(\lambda)$ coincide. This property will be later used in the proof of Theorem \ref{thm:finite-off-diag-pc}.

Note that Lemma \ref{lem:transfer-matrix-spectrum} also provides the eigenvectors of $T_k(z)$. Using these eigenvectors we now construct approximate singular vectors of $\hat{M}_N(z)$, corresponding to small singular values.

\subsection*{Construction of approximate singular vectors}
Recall, \corOZ{see \eqref{eq-L0},} that
$\{\gb_k\}_{k=0}^{L_0}$ is a partition of $[N]$ with
$L_0=\lfloor N^{\delta_1} \rfloor$, so that
within each block $\gb_k$, the entries of the diagonals of
$\hat{M}_N$ remains constant. However, for certain values of $N$, the
last block $\gb_{L_0}$ may have a small length. To overcome this, we
slightly modify $\hat{M}_N$. Namely, we replace the \corAB{(off)-}diagonal entries of
$\hat{M}_N$ in the last two blocks by their  average.
By a slight abuse of notation, we continue to write $\{\gb_k\}$ to
indicate the blocks in which the \corAB{(off)-}diagonal entries of \corAB{the modified} $\hat{M}_N$ are constant.
Note that we now have that $N^{1-\delta_1}/2 \le |\gb_k| \le 2 N^{1-\delta_1}$, \corAB{for all $k$}. Since this extra modification results in a change of operator norm of
$O(N^{-\alpha \delta_1}+N^{-\delta_2})$, \corAB{where we recall $\alpha:=\min_{\ell=0}^\gd\{\alpha_\ell\}$}, it is enough to prove
Theorem \ref{thm:finite-off-diag-pc} for the modified $\hat{M}_N(z)$.

{Fix $\delta_3 \in (0,1/3)$}. Next we choose a refinement of the above partition $\{\{\gb_k^{(j')}\}_{j'=1}^{L_k'}\}_{k=0}^{L_0} \subset \{\gb_k\}_{k=0}^{L_0}$ such that $N^{\delta_3}/2 \le |\gb_k^{(j')}| \le 2N^{\delta_3}$ for all $k \in [L_0]\cup \{0\}$ and $j' \in [L_k']$. Since $\delta_1<1/2$ and
$\delta_3 <1/3$,
such a property can be ensured.
Fix
\begin{equation}
  \label{eq-L}
  \corOZ{  L:=\sum_{k=0}^{L_0} L_k'=O(N^{1-\delta_3}).}
\end{equation}
Further let
\[
  0 = i_1 < i_2 < i_3 < \dots  < i_{L+1} = N,
\]
\corAB{be} the end points of the partition $\{\{\gb_k^{(j)}\}_{j=1}^{L_k'}\}_{k=0}^{L_0}$. That is, for any $k \in [L_0] \cup\{0\}$ and $j' \in [L_k']$ we have $\gb_k^{(j')}=[i_{j+1}]\setminus [i_j]$ for some $j \in [L] \cup\{0\}$, which in turn implies that $N^{\delta_3}/2 \le i_{j+1} - i_j \le 2N^{\delta_3}$ for all $j \in [L] \cup \{0\}$.

Next fix $k \in [L_0]\cup\{0\}$ and $j' \in [L_k']$ and assume that $\gb_k^{(j')}=[i_{j+1}]\setminus [i_j]$ for some $j \in [L] \cup\{0\}$. For $\ell \in [\hat\gd_N(k)]$ define the $N$-dimensional vectors ${\bm w}_\ell^j(k)$ as follows
\[
({\bm w}_\ell^{j}(k))_m:= \left\{	\begin{array}{ll} \hat{\lambda}_\ell(z,\corAB{k})^{m-i_j-1} & \mbox{ for } m=i_j+1,i_j+2,\ldots,i_{j+1}\\
0 & \mbox{otherwise}
\end{array}
\right.
\]
where $\hat{\lambda}_1(z,\corAB{k}), \hat{\lambda}_2(z,\corAB{k}), \ldots, \hat{\lambda}_{\hat\gd_N}(z,\corAB{k})$ are the roots of the equation $\hat{P}_{z,k}(\lambda)=0$. 


\begin{remark}
Note that
when $\hat\gd_N(k)=0$ then a block in $\hat{M}_N(z)$ becomes diagonal.
Since the singular values of a diagonal matrix are the
absolute values of its diagonal entries,
we do not need to bother with constructing
approximate singular vectors.
Therefore, when computing
bounds on the small singular values we will assume that $\hat\gd_N(k) >0$, and
define the candidates for small singular vectors $\{{\bm w}_\ell^j(k)\}$ only
in that case.
\end{remark}

The next lemma
gives a simple but useful property of
the vectors $\{{\bm w}^j_\ell(k)\}$. Before stating the result let us introduce one more notation. Fix any $N$-dimensional vector $u$. For $k \in [L_0]\cup\{0\}$, denote
\[
u\dvec{k, m}:= \begin{pmatrix} u_{m+\hat\gd_N(k)-1} \\ u_{m+\hat\gd_N(k)-2} \\ \vdots \\ u_m \end{pmatrix}, \qquad  m =1,2, \ldots, N-\hat\gd_N(k)+1.
\]
\begin{lemma}\label{eq:small-sing-vector}
Fix $k \in [L_0] \cup\{0\}$ and $j' \in [L_k']$ such that $\gb_k^{(j')}=[i_{j+1}]\setminus [i_j]$ for some $j \in [L]\cup\{0\}$. Assume that $\hat\gd_N(k) >0$.
\begin{enumerate}[(i)]
\item Let $u$ be an $N$-dimensional vector such that
\begin{equation}\label{eq:tran-mat-eqn}
u\dvec{k, m+1}= T_k(z) u\dvec{k,m}, \qquad m \in [i_{j+1}-\hat\gd_N(k)]\setminus [i_j].
\end{equation}
Then
\[
(\hat{M}_N(z) u )_m=0, \qquad m\in [i_{j+1}-\hat\gd_N(k)]\setminus [i_j].
\]
\item For any $\ell \in[\hat\gd_N(k)]$ we have
\[
(\hat{M}_N(z) {\bm w}_\ell^j(k) )_m=0, \qquad m\in [i_{j+1}-\hat\gd_N(k)]\setminus [i_j].
\]
\end{enumerate}
\end{lemma}
\begin{proof}
The condition \eqref{eq:tran-mat-eqn} implies that
\[
t_{\corAB{\hat\gd_N(k)}}(k) u_{m+\hat\gd_N(k)}= (z-t_0(k)) u_m - \sum_{m'=1}^{\hat\gd_N(k)-1} t_{m'}(k) u_{m+m'},
\]
for $m \in [i_{j+1}-\hat\gd_N(k)]\setminus [i_j]$. The conclusion of part (i) is now immediate from \eqref{eq:d-lin-eqn-1} and the definition of $\corAB{\hat\gd_N(k)}$. To prove (ii) we note that ${\bm w}_\ell^j(k) \dvec{k,m} = \hat{\lambda}_\ell(z,x)^{m-i_j-1} {\bm v}_\ell(k)$, for $m \in [i_{j+1}-\hat\gd_N(k)]\setminus [i_j]$, where ${\bm v}_\ell$ is as in Lemma \ref{lem:transfer-matrix-spectrum}. Thus from Lemma \ref{lem:transfer-matrix-spectrum} it follows that
$$T_k(z){\bm w}_{\ell}^j(k) \dvec{k,m} = \hat{\lambda}_\ell(z,k)^{m-i_j-1} T_k(z) {\bm v}_\ell(k) = \hat{\lambda}_\ell(z,k)^{m-i_j}  {\bm v}_\ell(k) = {\bm w}_\ell^j(k) \dvec{k,m+1}.$$
Now the proof completes by applying part (i).
\end{proof}

\subsection*{Identification of the set of bad $z$'s}
Recall that to prove Theorem \ref{thm:finite-off-diag-pc} we only need to find the limit of the log-potential for Lebesgue almost every $z \in \C$.
\corEE{As we will see below, \corABa{our methods to control} the small singular values of $\hat{M}_N(z)$ break
down when $P_{z,x}(\lambda)=0$ has roots near the
unit circle or when the Vandermonde matrix
\begin{equation}
\label{eq:bmV}
{\bm V}(k):=\begin{bmatrix} {\bm v}_1(k) & {\bm v}_2(k) & \cdots & {\bm v}_{\hat\gd_N}(k) \end{bmatrix}
\end{equation}
loses invertibility.}
In the next lemma
we show that the collection of all such bad $z$'s has small Lebesgue measure.

\begin{lemma}\label{lem:bad-z}
Let $\cN_N$ denote the collection of $z \in \C$ such that either of the following properties hold:
\begin{enumerate}
\item[(i)] For some $k \in [L_0] \cup \{0\}$, such that $\hat\gd_N(k) >0$,
\[
|t_{\hat\gd_N(k)}(k)|^{\hat\gd_N(k)-1}
\corEE{|\det( {\bm V}(k))|^2}
\le N^{-2\delta_1\corOZ{\gd}}.
\]
\item[(ii)] For some $k \in [L_0]\cup\{0\}$, such that $\hat\gd_N(k) >0$, there exists a root $\lambda$ of the equation $\hat{P}_{z,k}(\lambda)=0$ such that $1-N^{-3\delta_1} \le |\lambda| \le 1+N^{-3\delta_1}$. 

\item[(iii)]\corABrev{
\[
\inf_{k=0}^{L_0} |z - t_0(k)| \le N^{-\frac{\alpha_0\delta_1}{2(2\alpha_0-1)}}.
\]}
\end{enumerate}
Then $\mathrm{Leb}(\cN_N) \le N^{-\delta_1}$ for all large $N$.
\end{lemma}
\begin{proof}
\corEE{We first estimate the area of $z$ satisfying (ii).}
Fix $k \in [L_0]\cup\{0\}$ such that $\hat \gd_N:=\hat\gd_N(k) >0$. For any $\vep>0$, denote $\mathcal{A}_\vep:= \overline{B_\C(0,1+\vep)}\setminus B_\C(0,1-\vep)$, \corABrev{where we recall that $B_\C(0,r)$ denote the open disc of radius $r$ in the complex plane centered at zero}. Let
\corEE{$Q_k(\lambda):= \hat{P}_{z,k}(\lambda) + z = \sum_{\ell=0}^{\hat\gd_N} t_\ell(k)\lambda^\ell$.  Then, the set of $z$ for which there exists a $\lambda \in \mathcal{A}_{N^{-3 \delta_1}}$ so that $\hat{P}_{z,k}(\lambda) = 0$ are contained in the image $Q_k(\mathcal{A}_{N^{-3 \delta_1}}).$}
Therefore the set of all $z$'s satisfying property (ii) is contained in $\cup_k Q_k(\mathcal{A}_{N^{-3 \delta_1}})$.
\corEE{The area of such an image $Q_k(\mathcal{A}_{N^{-3 \delta_1}})$ can be estimated by
$\sup_{|z| < 2} |Q_k'(z)|\mathrm{Leb} \left(\mathcal{A}_{N^{-3\delta_1}}\right).$
Since $L_0=O(N^{\delta_1})$, it follows that for all $N$ sufficiently large}
\[
\mathrm{Leb} \left(\cup_k Q_k(\mathcal{A}_{N^{-3\delta_1}})\right) \le \corABrev{\frac13} N^{-\delta_1}.
\]

\corEE{Turning to the set of $z$ satisfying (i),}
we recall that the \textit{discriminant} of
a polynomial
$\tilde{P}(\lambda):=\sum_{\ell=1}^m t_\ell \lambda^\ell$ is
\begin{equation}\label{eq:def-disc}
D(\tilde{P}):= \det[ \mathrm{Disc}(\tilde{P}) ]= t_{m}^{2m-1}\prod_{1 \le \ell <\ell' \le m} (\lambda_\ell - \lambda_{\ell'})^2,
\end{equation}
where
\[
\mathrm{Disc}(\tilde{P}):=\begin{bmatrix} t_m & t_{m-1} & \cdots &\cdots & t_0\\ & \ddots &  & & & \ddots & \\ & & t_m  & t_{m-1} & \cdots & \cdots & t_0  \\ mt_m & \cdots & 2t_2 & t_1 \\ & \ddots & & & \ddots\\ & & \ddots & & & \ddots \\ & & & mt_m & \cdots & 2t_2 & t_1\end{bmatrix}.
\]
Since $z$ appears in $\hat{P}_{z,k}(\lambda)$ only as coefficient of $\lambda^0$, expanding the determinant we see
$$\det[ \mathrm{Disc}(\hat{P}_{z,k})]= (\hat\gd_N t_{\hat\gd_N}(k))^{\hat\gd_N} \cdot z^{\hat\gd_N-1} + P_{\corAB{\hat\gd_N-2}}(x) z^{\hat\gd_N-2}+\cdots +P_0(x),$$
for some continuous functions $\{P_\ell(\cdot)\}_{\ell=0}^{\hat\gd_N-2}$. Let $\{z_\ell(k)\}_{\ell=1}^{\hat\gd_N-1}$ be the roots of the equation $\det[\mathrm{Disc}(\hat{P}_{z,k})]=0$. 
\corABrev{Hence,
\[
|D(\hat P_{z,k})|= (\hat\gd_N t_{\hat\gd_N}(k))^{\hat\gd_N} \prod_{\ell=1}^{\hat \gd _N -1} |z - z_\ell(k)|.
\]
Therefore, recalling \eqref{eq:def-disc} we} obtain that
\[
\corEE{\begin{aligned}
|t_{\hat\gd_N}(k)|^{2\hat\gd_N-1} |\det( {\bm V}(k))|^2
 &= |t_{\hat\gd_N}(k)|^{2\hat\gd_N-1} \prod_{1 \le \ell < \ell' \le \hat\gd_N(k)} |\hat{\lambda}_\ell(z,k) - \hat{\lambda}_{\ell'}(z,k)|^2 \\
 &\corABrev{= |D(\hat P_{z,k})|}= |\hat\gd_N t_{\hat\gd_N}(k)|^{\hat\gd_N} \prod_{\ell=1}^{\hat\gd_N-1} |z- z_\ell(k)|.
 \end{aligned}}
\]
Thus the set of all $z$'s satisfying property (i) is contained in
\[
\cup_{k=0}^{L_0} \cup_{\ell=1}^{\hat\gd_N-1} B_\C(z_\ell(k), \hat\gd_N^{-1} N^{-2\delta_1}),
\]
whose Lebesgue measure is bounded above by
$N^{-3\delta_1}$. Taking a union on $O(N^{\delta_1})$ possible $k$-s
\corABrev{we find that the Lebesgue measure of the set of $z$ satisfying property (i) is bounded by $\frac13 N^{-\delta_1}$. To complete the proof it remains to prove the same for the set of $z$ satisfying property (iii). This follows from a volumetric argument. Indeed, recalling the definition of $t_0(\cdot)$ we find that
\[
\inf_{k =0}^{L_0} |z - t_0(k)| \ge \min\{ {\rm dist}(z, f_0([0,1])), |z|\}.
\]
Using the fact that $f_0$ is an $\alpha_0$-H\"{o}lder continuous function from the triangle inequality we obtain that
\[
{\rm dist}(z, f_0([0,1])) \le N^{-\vep} \Rightarrow z \in \bigcup_{k =0}^{\lceil N^{\vep/\alpha_0}\rceil} B_\C(f_0(k/N^{\vep/\alpha_0}), 2 N^{-\vep})
\]
for $\vep >0$. As $\alpha_0 > \frac12$, setting $\vep = \frac{\alpha_0 \delta_1}{2(2\alpha_0 -1)}$ it yields that
\[
{\rm Leb}(\{z: {\rm dist}(z, f_0([0,1])) \le N^{-\vep}\} \cup B_\C(0,N^{-\vep})) \le \frac13 N^{-\delta_1}.
\]
This completes the proof.}
\end{proof}

Next, building on Lemma \ref{eq:small-sing-vector}(ii)
we show that for $z \notin \cN_N$ and vectors $w$ not
belonging to the span of $\{{\bm w}_\ell^j(k)\}$,
the $\ell_2$-norm of $\hat{M}_N(z) w$ cannot be too small.
This yields a bound on the number of small singular values of $\hat{M}_N(z)$.

Fix $k \in [L_0] \cup \{0\}$ and $j' \in [L_k']$. This fixes some $j \in [L] \cup\{0\}$ such that $\gb_k^{(j')}=[i_{j+1}] \setminus [i_j]$. Denote $\cS_{k,j}:=\Span({\bm w}_\ell^j(k): \ell \in [\hat\gd_N(k)])$ and let $\psi_{k,j}$ be the orthogonal projection onto $\cS_{k,j}$.
Further let
$\pi_{k,j}$ and $\rho_{k,j}$ be the projections onto the span of $\{e_m\}_{m=i_j+1}^{i_{j+1}}$ and $\{e_m\}_{k=i_j+1}^{i_{j+1}-\corAB{\hat\gd_N}(k)}$, respectively, where $e_m$ is the $m$-th canonical basic vector. When needed, we will view $\pi_{k,j}, \rho_{k,j}$, and $\psi_{k,j}$ as projection matrices of appropriate dimensions.

\begin{lemma}\label{lem:not-small-norm}
Fix $R<\infty$ and $z \in B_\C(0,R) \setminus \cN_N$. Let $k \in [L_0]$ and $j' \in [L_k']$ such that $\gb_k^{(j')}= [i_{j+1}] \setminus [i_j]$ for some $j \in [L] \cup\{0\}$. Then there exists a positive finite constant $C_1(R,\gd, {\bm f})$, depending only on $R, \gd$, and $\max_\ell \sup_{x \in [0,1]}|f_\ell(x)|$, such that for any $w \in \C^N$, we have
\begin{equation}\label{eq:notsmallubd}
\|\pi_{k,j}(w -\psi_{k,j} w)\|_2 \le  C_1(R,\gd, {\bm f}) N^{2\gd\delta_1+ \gd^2 \delta_2+2\delta_3 +\corABrev{\frac{\alpha_0\delta_1}{2(2\alpha_0-1)}}}\|\rho_{k,j} \hat{M}_N(z) w\|_2.
\end{equation}
\end{lemma}

\corAB{Note that Lemma \ref{lem:not-small-norm} is similar to Lemma \ref{lem:localstructure}. Analogous to the proof of Lemma \ref{lem:localstructure}, here also the proof proceeds by identifying a vector $y \in \cS_{k,j}$ and showing that $\|\pi_{k,j}(x-y)\|_2$ satisfies the bound \eqref{eq:notsmallubd}. For $\gd >1$, the choice of an appropriate vector $y$ is significantly more difficult and requires new ideas.}

\begin{proof}[Proof of Lemma \ref{lem:not-small-norm}]
  We write $\hat{\lambda}_\ell:= \hat{\lambda}_\ell(z,k)$ and $\hat\gd_N:=
 \hat\gd_N(k)$. \corABrev{First let us consider the case $\hat \gd_N=0$. This implies that $\cS_{k,j}=\emptyset$. Therefore, $\psi_{k,j}=0$ and $\rho_{k,j}=\pi_{k,j}$. So, it is enough to show that
 \begin{equation}\label{eq:gd-eq-0}
 \|\pi_{k,j} w\|_2 \le N^{\frac{\alpha_0 \delta_1}{2(2\alpha_0 -1)}} \|\pi_{k,j} \hat M_N(z) w\|_2. 
 \end{equation}
 Since $\hat \gd_N=0$, from \eqref{eq:d-lin-eqn-1} we further have that 
 \[
 \pi_{k,j} \hat M_N(z) w =(t_0(k)-z) \pi_{k,j} w.
 \]
 Using the fact that $z \notin \cN_N$ we have that $\inf_k|t_0(k) -z| \ge N^{-\frac{\alpha_0 \delta_1}{2(2\alpha_0 -1)}}$. This yields \eqref{eq:gd-eq-0} and hence \eqref{eq:notsmallubd} is established for $\hat \gd_N =0$. It remains to prove the same when $\hat \gd_N >0$. 
 }
 

\corEE{Without loss of generality, we
assume that $\{\hat{\lambda}_\ell\}_{\ell=1}^{\hat\gd_N}$ are arranged in decreasing order of moduli, and define $\gd_0$ so that
\(
|\hat{\lambda}_1| \ge |\hat{\lambda}_2| \ge \cdots \ge |\hat{\lambda}_{\gd_0}|
\ge 1 > |\hat{\lambda}_{\gd_0+1}| \ge \cdots \ge |\hat{\lambda}_{\hat\gd_N}|,
\)
with $\gd_0 = \hat\gd_N$ if all $\hat\lambda_\ell \geq 1$
and $\gd_0 = 0$ if all $\hat\lambda_\ell < 1$
.} Define a $(2\hat\gd_N) \times (2\hat\gd_N)$ matrix
\[
\corAB{
{\sf L}:=
\begin{bmatrix}
{\bm v}_1(k)  &{\bm v}_2(k)  & \cdots & {\bm v}_{\hat\gd_N}(k) &
{\bm v}_1(k)   & \cdots & {\bm v}_{\gd_0}(k)& 0 & \cdots &0\\
\hat{\lambda}_1^{b_j-\hat\gd_N}{\bm v}_1(k)  &\hat{\lambda}_2^{b_j-\hat\gd_N(k)}{\bm v}_2(k)  & \cdots & \hat{\lambda}_{\corEE{\hat\gd_N}}^{b_j-\hat\gd_N}{\bm v}_{\hat\gd_N}(k)  &
0   & \cdots & 0& {\bm v}_{\gd_0+1}(k)& \cdots & {\bm v}_{\hat\gd_N}(k)
\end{bmatrix},
}\]
where $\{{\bm v}_\ell(k)\}_{\ell=1}^{\hat\gd_N}$ are as in Lemma \ref{lem:transfer-matrix-spectrum} and $b_j:=i_{j+1}-i_j$.
Since $z \in B_\C(0,R)\setminus \cN_N$,
the eigenvalues $\{\hat{\lambda}_\ell\}_{\ell=1}^{\hat\gd_N}$ are all distinct, and hence the vectors $\{{\bm v}_\ell(k)\}_{\ell=1}^{\hat\gd_N}$ are linearly independent. Therefore,
$\text{rank}({\sf L})=2\hat\gd_N$, and the system of linear equations
\begin{equation}\label{eq:lin-system-eqn-be}
\begin{pmatrix} w\dvec{k, i_j+1}\\ w\dvec{k,i_{j+1}-\corEE{\hat\gd_N}+1}\end{pmatrix} ={\sf L} \begin{pmatrix} a_1 \\ \vdots \\ a_{\hat\gd_N} \\ a_1' \\ \vdots \\ a_{\hat\gd_N}' \end{pmatrix}
\end{equation}
admits a unique solution. Set
\[
y:= \sum_{\ell=1}^{\hat\gd_N} a_\ell {\bm w}_\ell^j(k) \qquad \text{ and } \qquad \zeta:=\zeta(w):=w - y.
\]
With this choice of $\zeta$ we will show that
\begin{equation}\label{eq:zeta-lbd-norm}
\|\pi_{k,j}\zeta\|_2 \le C_1(R,\gd, {\bm f}) N^{2\gd\delta_1+ \gd^2 \delta_2+2\delta_3} \|\rho_{k,j} \hat{M}_N(z) w\|_2,
\end{equation}
for some constant $C_1(R,\gd, {\bm f})$. This will complete the proof.
Indeed, from the definition of the projection operator $\psi_{k,j}$ it follows that
\begin{equation}\label{eq:proj-dfn}
\|w- \psi_{k,j} w\|_2 \le \|\zeta\|_2.
\end{equation}
On other hand we note that $\mathcal{S}_{k,j} \subset \Span(\{e_m\}_{m=i_j+1}^{i_j+1})$. Therefore $\pi_{k,j} \psi_{k,j} = \psi_{k,j}$. Recalling
that $y \in \mathcal{S}_{k,j}$, we have
\begin{align*}
\|w-\psi_{k,j} w\|_2^2=\|(\corAB{\Id}-\pi_{k,j})w\|_2^2+ \|\pi_{k,j}(w - \psi_{k,j} w)\|_2^2
\end{align*}
and
\begin{align*}
\|\zeta\|_2^2=  \|(\corAB{\Id}-\pi_{k,j}) \zeta\|_2^2 + \|\pi_{k,j} \zeta\|_2^2 = \|(\corAB{\Id}-\pi_{k,j})w\|_2^2 +  \|\pi_{k,j} \zeta\|_2^2.
\end{align*}
Thus from \eqref{eq:proj-dfn} we obtain
\[
\|\pi_{k,j}(w-\psi_{k,j} w)\|_2 \le \|\pi_{k,j} \zeta\|_2,
\]
and so it is enough to show that \eqref{eq:zeta-lbd-norm} holds.

We turn now to the proof of \eqref{eq:zeta-lbd-norm}
Recalling that $y \in {\mathcal S}_{k,j}$, the span of $\{{\bm w}_\ell^j(k)\}_{\ell \in [\hat\gd]}$, an application of Lemma \ref{eq:small-sing-vector}(ii) implies that $\rho_{k,j} \hat{M}_N(z)y=0$. So, using \eqref{eq:d-lin-eqn-1} and recalling the definition of $T_k(z)$ we see that
\begin{equation}\label{eq:recursion}
\zeta \dvec{k,m+1}= T_k(z) \zeta\dvec{k,m} + \begin{pmatrix} \frac{1}{t_{\hat\gd_N}(k)}(\hat{M}_N(z) w)_{m} \\ 0 \\ \vdots  \\ 0 \end{pmatrix}, \qquad m \in [i_{j+1}-\hat\gd_N]\setminus [i_j]. 
\end{equation}
\corEE{From the linear independence of $\left\{ {\bm v}_\ell(k) \right\}_{1}^{\hat\gd_N}$,
there are  $\{\tilde{a}_\ell\}_{\ell=1}^{\hat\gd_N}$ so that}
\begin{equation}\label{eq:tilde-alpha}
\begin{pmatrix} 1 & 0 & \cdots & 0 \end{pmatrix} = \sum_{\ell=1}^{\hat\gd_N} \tilde{a}_\ell {\bm v}_\ell(k)^{\sf T}.
\end{equation}
Hence denoting $\beta_m:= (t_{\hat\gd_N}(k))^{-1}(\hat{M}_N(z)w)_m$ 
we observe that \eqref{eq:recursion} simplifies to
\begin{equation}\label{eq:recursion-1}
\zeta\dvec{k,m+1}= T_k(z) \zeta\dvec{k,m} + \beta_m\sum_{\ell=1}^{\hat\gd_N} \tilde{a}_\ell {\bm v}_\ell(k), \qquad m \in[i_{j+1}-\hat\gd_N]\setminus [i_j+1]. 
\end{equation}
Iterating \eqref{eq:recursion-1} we obtain
\begin{align}\label{eq:recursion-iterated}
\zeta\dvec{k,m+1} & \, = T_k(z)^{m-i_j} \zeta\dvec{k,i_j+1} + \sum_{m'=i_j+1}^m \beta_{m'} T_k(z)^{m-m'} \left[\sum_{\ell=1}^{\hat\gd_N} \tilde{a}_\ell {\bm v}_\ell(k)\right] \notag\\
& \, = T_k(z)^{m-i_j} \zeta\dvec{k, i_j+1} + \sum_{\ell=1}^{\hat\gd_N} \tilde{a}_\ell \left[\sum_{m'=i_j+1}^m \beta_{m'} \hat{\lambda}_\ell^{m-m'} \right] {\bm v}_\ell(k),
\end{align}
for $m \in [i_{j+1}-\hat\gd_N]\setminus[i_j+1]$, where the last step follows from the fact that ${\bm v}_\ell(k)$ is an eigenvector of $T_{k}(z)$ corresponding to $\hat{\lambda}_\ell$ (see Lemma \ref{lem:transfer-matrix-spectrum}). Recalling that $b_j:=i_{j+1}-i_j$ we note that \eqref{eq:recursion-iterated} in particular implies
\begin{equation}\label{eq:iterated-end}
\zeta\dvec{k, i_{j+1}-\hat\gd_N+1}= T_k(z)^{b_j-\hat\gd_N} \zeta\dvec{k,i_j+1} + \sum_{\ell=1}^{\corAB{\hat\gd_N}} \tilde{a}_\ell \left[\sum_{m'=i_j+1}^{i_{j+1}-\hat\gd_N} \beta_{m'} \hat{\lambda}_\ell^{i_{j+1}-\hat\gd_N-m'} \right] {\bm v}_\ell(k).
\end{equation}
Now recalling the definitions of $\{\corAB{{\bm w}_\ell^j(k)}\}_{\ell \in [\corAB{\hat\gd_N}]}$ we see that
\[
{\bm w}_\ell^j(k)\dvec{k,i_j+1} = {\bm v}_\ell(k) \qquad \text{ and } \qquad {\bm w}_\ell^j(k)\dvec{k, i_{j+1}-\hat\gd_N+1} = \hat{\lambda}_\ell^{b_j-\hat\gd_N} {\bm v}_\ell(k).
\]
Since $\zeta= w - \sum_{\ell=1}^{\hat\gd_N} a_\ell {\bm w}_\ell^j(k)$, from \eqref{eq:lin-system-eqn-be} we obtain
\begin{equation}\label{eq:zeta-be}
\zeta \dvec{k, i_j+1}= \sum_{\ell=1}^{\gd_0} a_\ell' {\bm v}_\ell(k) \quad \text{ and } \quad \zeta \dvec{k, i_{j+1}-\hat\gd_N+1}= \sum_{\ell=\gd_0+1}^{\hat\gd_N} a_\ell' {\bm v}_\ell(k).
\end{equation}
Plugging these in \eqref{eq:iterated-end} we deduce
\[
\sum_{\ell=1}^{\gd_0} \left( \tilde{a}_\ell \left[\sum_{m'=i_j+1}^{i_{j+1}-\hat\gd_N} \beta_{m'} \hat{\lambda}_\ell^{i_{j+1}-\hat\gd_N-m'} \right]+ a_\ell' \hat{\lambda}_\ell^{b_j-\hat\gd_N} \right){\bm v}_\ell(k) + \sum_{\ell=\gd_0+1}^{\hat\gd_N} \left( \tilde{a}_\ell \left[\sum_{m'=i_j+1}^{i_{j+1}-\hat\gd_N} \beta_{m'} \hat{\lambda}_\ell^{i_{j+1}-\hat\gd_N-m'} \right] - a_\ell'\right){\bm v}_\ell(k) =0.
\]
Since $\{{\bm v}_\ell(k)\}_{\ell=1}^{\hat\gd_N}$ are linearly independent vectors it further implies that
\begin{equation}\label{eq:lin-constraint-large-lambda}
a_\ell' \hat{\lambda}_\ell^{b_j-\hat\gd_N} +\tilde{a}_\ell \left[\sum_{m'=i_j+1}^{i_{j+1}-\hat\gd_N} \beta_{m'} \hat{\lambda}_\ell^{i_{j+1}-\hat\gd_N-m'} \right]=0, \quad \ell \in [\gd_0]; \quad  a_\ell' = \tilde{a}_\ell \left[\sum_{m'=i_j+1}^{i_{j+1}-\hat\gd_N} \beta_{m'} \hat{\lambda}_\ell^{i_{j+1}-\hat\gd_N-m'} \right], \quad \ell \in [\hat\gd_N]\setminus[\gd_0].
\end{equation}
Thus from \eqref{eq:recursion-iterated} and \eqref{eq:zeta-be}, using \eqref{eq:lin-constraint-large-lambda}, we further obtain that for any $m \in [i_{j+1}-\hat\gd_N]\setminus[i_j]$,
\begin{align}\label{eq:zeta_j_bd}
\zeta \dvec{k,m+1} & \, = T_k(z)^{m-i_j} \zeta \dvec{k,i_j+1} + \sum_{\ell=1}^{\hat\gd_N} \tilde{a}_\ell \left[\sum_{m'=i_j+1}^m \beta_{m'} \hat{\lambda}_\ell^{m-m'} \right] {\bm v}_\ell(k) \notag\\
& \, = \sum_{\ell=1}^{\gd_0} \left(a_\ell' \hat{\lambda}_\ell^{m-i_j}  + \tilde{a}_\ell \left[\sum_{m'=i_j+1}^m \beta_{m'} \hat{\lambda}_\ell^{m-m'} \right] \right){\bm v}_\ell(k) \notag\\
& \qquad \qquad \qquad  \qquad \qquad \qquad   \qquad \qquad + \sum_{\ell=\gd_0+1}^{\hat\gd_N} \tilde{a}_\ell \left[\sum_{m'=i_j+1}^m \beta_{m'} \hat{\lambda}_\ell^{m-m'} \right] {\bm v}_\ell(k) \notag\\
& \, = -\sum_{\ell=1}^{\gd_0}  \tilde{a}_\ell \left[\sum_{m'=m+1}^{i_{j+1}-\hat\gd_N} \beta_{m'} \hat{\lambda}_\ell^{m-m'} \right] {\bm v}_\ell(k) + \sum_{\ell=\gd_0+1}^{\hat\gd_N} \tilde{a}_\ell \left[\sum_{m'=i_j+1}^m \beta_{m'} \hat{\lambda}_\ell^{m-m'} \right] {\bm v}_\ell(k).
\end{align}
Since the first coordinate of $\zeta\dvec{k,m+1}$ is $\zeta_{m+\hat\gd_N}$, $|\hat{\lambda}_\ell| \ge 1$ for $\ell \in [\gd_0]$, and $|\hat{\lambda}_\ell| \le 1$ for $\ell \in [\hat\gd_N]\setminus [\gd_0]$, using the triangle inequality from  \eqref{eq:zeta_j_bd} we see that
\begin{align}\label{eq:zeta-bd-1}
|\zeta_{m+\hat\gd_N}| & \, \le \sum_{\ell=1}^{\gd_0} |\tilde{a}_\ell| |\hat{\lambda}_\ell|^{\hat\gd_N-1} \sum_{m'=m+1}^{i_{j+1}-\hat\gd_N} |\beta_{m'}| + \sum_{\ell=\gd_0+1}^{\hat\gd_N} |\tilde{a}_\ell| |\hat{\lambda}_\ell|^{\hat\gd_N-1} \sum_{m'=i_j+1}^{m} |\beta_{m'}| \notag\\
& \, \le \|T_k(z)\|^{\hat\gd_N-1} \left( \sum_{m'=i_j+1}^{i_{j+1}-\hat\gd_N} |\beta_{m'}|\right) \cdot  \left( \sum_{\ell=1}^{\hat\gd_N} |\tilde{a}_\ell|\right), \quad \text{ for } m \in [i_{j+1}-\hat\gd_N]\setminus [i_j].
\end{align}
From \eqref{eq:zeta-be} and \eqref{eq:lin-constraint-large-lambda} it also follows that
\begin{align}
\zeta\dvec{k,i_j+1}= \sum_{\ell=1}^{\gd_0} a_\ell' {\bm v}_\ell(k) = - \sum_{\ell=1}^{\gd_0}\tilde{a}_\ell \left[\sum_{m'=i_j+1}^{i_{j+1}-\hat\gd_N} \beta_{m'} \hat{\lambda}_\ell^{i_j-m'} \right] {\bm v}_{\ell}(k). \notag
\end{align}
Thus
\begin{align}\label{eq:zeta-bd-2}
\max_{m=i_j+1}^{i_j+\hat\gd} |\zeta(m)| & \, \le \left\|\zeta\dvec{k,i_j+1}\right\|_2 \le \sum_{\ell=1}^{\gd_0} |\tilde{a}_\ell| \cdot \left( \sum_{m'=i_j+1}^{i_{j+1}-\hat\gd_N} |\beta_{m'}|\right) \cdot \|{\bm v}_\ell(k)\|_2\notag\\
& \,  \le \sqrt{\gd} \cdot \left(\|T_k(z)\|^{\gd-1}\vee 1\right) \cdot \left( \sum_{m'={i_j+1}}^{i_{j+1}-\hat\gd_N} |\beta_{m'}|\right) \cdot  \left( \sum_{\ell=1}^{\hat\gd_N} |\tilde{a}_\ell|\right),
\end{align}
where we have used the fact that
\[
\|{\bm v}_\ell(k)\|_2 \le \sqrt{\gd} \cdot (|\hat{\lambda}_\ell|^{\gd-1}\vee 1).
\]
Hence, combining \eqref{eq:zeta-bd-1}-\eqref{eq:zeta-bd-2} we obtain
\begin{equation}\label{eq:pi_kj_zeta}
\|\pi_{k,j}(\zeta)\|_2 \le b_j\max_{m=i_j+1}^{i_{j+1}} |\zeta_j| \le 2N^{\delta_3} \sqrt{\gd}  \cdot \left(\|T_k(z)\|^{\gd-1}\vee 1\right)\cdot \left( \sum_{m=i_j+1}^{i_{j+1}-\hat\gd_N} |\beta_m|\right) \cdot  \left( \sum_{\ell=1}^{\hat\gd_N} |\tilde{a}_\ell|\right).
\end{equation}
Now to complete the proof we need to find a bound on $\sum_{\ell=1}^{\hat\gd_N} |\tilde{a}_\ell|$. To this end, recall that $(\tilde{a}_1,\ldots,\tilde{a}_{\hat\gd_N})$ satisfies the system of linear equations \eqref{eq:tilde-alpha}. Using Cramer's rule it is easy to check that $\tilde{a}_\ell = \prod_{\ell' \ne \ell} \frac{\hat{\lambda}_{\ell'}}{\hat{\lambda}_{\ell'}- \hat{\lambda}_\ell}$. Therefore, \corEE{recalling that ${\bm V}(k):=\begin{bmatrix} {\bm v}_1(k) & {\bm v}_2(k) & \cdots & {\bm v}_{\hat\gd_N}(k) \end{bmatrix}$},
\begin{align*}
|\tilde{a}_\ell|= \prod_{\ell' \ne \ell} \left|\frac{\hat{\lambda}_{\ell'}}{\hat{\lambda}_{\ell'}- \hat{\lambda}_\ell}\right| \le \left(\|T_k(z)\|^{\gd -1} \vee 1\right) \cdot \prod_{\ell' \ne \ell} \left|\frac{1}{\hat{\lambda}_{\ell'}- \hat{\lambda}_\ell}\right|  & \le \left((2\| T_k(z)\|)^{\frac{\gd(\gd-1)}{2}} \vee 1\right) \cdot |\det({\bm V}(k))|^{-1} \\
& \le \left((2\| T_k(z)\|)^{\gd(\gd-1)} \vee 1\right) \cdot |\det({\bm V}(k))|^{-2}.
\end{align*}
Recalling that $\beta_m=(t_{\hat\gd_N}(k))^{-1} (\hat{M}_N(z)w)_m$, an application of Cauchy-Schwarz yields
\begin{equation}\label{eq:zeta-lbd-norm-1}
\|\pi_{k,j}(\zeta)\|_2 \le   2^{\gd^2} {\gd^{3/2}} N^{3\delta_3/2}   \cdot \left(\|T_k(z)\|^{\gd^2}\vee 1\right)\cdot \|\rho_{k,j} \hat{M}_N(z) w\|_2 \cdot | t_{\hat\gd_N}(k)|^{-1} |\det({\bm V}(k))|^{-2}.
\end{equation}
Since $z\in \cN_N^c$ (cf.\ Lemma \ref{lem:bad-z}),
\[
|t_{\hat\gd_N}(k)| \cdot  |\det({\bm V}(k))|^2 \ge \frac{|t_{\hat\gd_N}(k)|^{\hat\gd_N(k)-1} |\det({\bm V}(k))|^2}{\sup_{x \in [0,1]} |f_{\hat\gd_N}(x)|^{\hat\gd_N-2}}  \ge C_0^{-1} N^{-2\delta_1 \gd},
\]
for some $C_0 <\infty$. By the Gershgorin circle theorem we also have that
\begin{equation}
\label{eq:Tkz}
\|T_k(z)\|^2
=
\|T_k(z)^*T_k(z)\|
= O\left((t_{\hat\gd_N}(k))^{-2}\right)=O(N^{2\delta_2}),
\end{equation}
where the last step follows from the fact that the non-zero entries of $\hat{M}_N$ are bounded below by $N^{-\delta_2}$. Therefore, plugging the last two bounds in \eqref{eq:zeta-lbd-norm-1} we arrive at \eqref{eq:zeta-lbd-norm}.
\end{proof}

Denote
\begin{equation}
  \label{eq-frakL}
  \corOZ{\gL:=\sum_{k=0}^{L_0} L_k' \cdot \hat\gd_N(k)= O(\gd N^{1-\delta_3}).}
\end{equation}
Building on Lemma \ref{lem:not-small-norm} we now prove
a lower bound on the $(\gL+1)$-st smallest singular value
of $\hat{M}_N(z)$. First we prove an estimate
that will also be useful in obtaining
a lower bound on the product of the small singular values of $\hat{M}_{N}(z)$.
To state it, we let
$\psi:=\sum \psi_{k,j}$, i.e.~$\psi$ is the orthogonal
projection operator onto the space spanned by
$\cup \cS_{k,j}$, and $\rho:=\sum \rho_{k,j}$.
\begin{lemma} \label{lem:orth-compl-large-norm}
Fix $R<\infty$ and $z \in B_\C(0,R) \setminus \cN_N$.  Then for any vector $w \in \C^N$ we have
  \[
    \|w- \psi w\|_2    \leq C_1(R,\gd, {\bm f}) N^{2\gd\delta_1+ \gd^2 \delta_2+2\delta_3+ \corABrev{\frac{\alpha_0\delta_1}{2(2\alpha_0-1)}}}\|\rho \hat{M}_N(z) w\|_2,
  \]
where $C_1(R,\gd, {\bm f})$ is as in Lemma \ref{lem:not-small-norm}.
 \end{lemma}
\begin{proof}
The proof is a simple application of Lemma \ref{lem:not-small-norm}. Since $\sum \pi_{k,j}=1$, $\pi_{k,j} \psi_{k,j}=\psi_{k,j}$, and $\{\pi_{k,j}\}$ are orthogonal it follows that
\begin{align*}
 \|w-\psi w\|_2^2 =  \left\|\sum (\pi_{k,j} w - \psi_{k,j} w)\right\|_2^2 = \sum \|\pi_{k,j}(w-\psi_{k,j}w)\|_2^2.
\end{align*}
Now the result follows from Lemma \ref{lem:not-small-norm} upon noting the fact that $\{\rho_{k,j}\}$ are orthogonal.
\end{proof}

From Lemma \ref{lem:orth-compl-large-norm} we immediately obtain the following corollary.

\begin{corollary}\label{cor:lb-not-small-sing}
Fix $R<\infty$ and $z \in B_\C(0,R) \setminus \cN_N$. Then
  \[
    \sigma_{N-\gL}(\hat{M}_N(z)) \geq C_1(R,\gd, {\bm f})^{-1} N^{-(2\gd\delta_1+ \gd^2 \delta_2+2\delta_3+\corABrev{\frac{\alpha_0\delta_1}{2(2\alpha_0-1)}})},
  \]
where $C_1(R,\gd, {\bm f})$ is as in Lemma \ref{lem:not-small-norm}.
\end{corollary}
\corAB{The proof of Corollary \ref{cor:lb-not-small-sing} is similar to that of Corollary \ref{cor:lb}, and hence omitted.}
\begin{remark}\label{rmk:delta}
 Let $\{\delta_i\}_{i=1}^3$ be such that
 \begin{equation}\label{eq:choice-delta}
 \max\left\{\delta_1,\delta_2,\delta_3, \corABrev{\frac{\alpha_0\delta_1}{2(2\alpha_0-1)}}\right\} \le \frac{1}{\corABrev{40}\gd^2} \cdot \left(\gamma - \frac{1}{2}\right) \quad \text{ and } \quad \delta_1 < \frac{\delta_3}{4}.
 \end{equation}
 It follows from Corollary \ref{cor:lb-not-small-sing} that there are only at most $\gL=O(N^{1-\delta_3})$ singular values of $\hat{M}_N(z)$ that are $O(N^{-\frac{1}{3}(\gamma-1/2)})$, which upon choosing $\vep_N= N^{-\frac{1}{3}(\gamma-1/2)}$ in \eqref{eq:N*-dfn}, implies that $N^* \le \gL = o(N/\log N)$. This verifies that the number of small singular values of $\hat{M}_N(z)$, for $z \in B_\C(0,R)\setminus \cN_N$ is as desired. In the remainder of this paper we will work with $\{\delta_i\}_{i=1}^3$ satisfying \eqref{eq:choice-delta}.
\end{remark}

Equipped with Remark \ref{rmk:delta} we note that it remains
to find  matching upper and lower bounds, up to sub-exponential factors,
on the product of small singular values. \corAB{In the context of Theorem \ref{thm:DJssv} the upper bound on the product of the small singular values}
\corOZ{was} \corAB{achieved by finding a collection of orthonormal vectors}
\corOZ{which were}
\corAB{approximate singular vectors, and then appealing to Lemma \ref{lem:prod_sing}.
In the current set-up, one notes that the approximate singular vectors, in particular $\{{\bm w}_\ell^j(k)\}_{\ell=1}^{\hat\gd_N}$ for any $j$ and $k$, are not orthogonal. Therefore we need to work with an orthonormal basis of ${\cS}_{k,j}$. To this end,} a key
\corOZ{step}
will be to obtain bounds on the determinant of $\mathfrak{U}_{k,j}^*\mathfrak{U}_{k,j}$ where $\mathfrak{U}_{k,j}:=\pi_{k,j} \hat{M}_N(z) U_{k,j}$, and the columns of $U_{k,j}$ \corOZ{are} an orthonormal basis of $\cS_{k,j}$.
We start with bounds on $\det(\mathfrak{W}_{k,j})$ where $\mathfrak{W}_{k,j}:=W_{k,j}^*W_{k,j}$, and \corAB{$W_{k,j}$} \corOZ{is} the matrix whose columns are $\{{\bm w}^j_\ell(k)\}_{\ell=1}^{\hat\gd_N(k)}$.

\begin{lemma}\label{lem:det-W-bd}
Fix $R<\infty$ and $z \in B_\C(0,R) \setminus \cN_N$. Let $k \in [L_0]$ and $j' \in [L_k']$ such that $\gb_k^{(j')}= [i_{j+1}] \setminus [i_j]$ for some $j \in [L] \cup\{0\}$. Assume $\hat\gd_N(k) >0$. Then
\[
  C_2(R,\gd, {\bm f})^{-1} N^{-\corEE{2\gd \delta_1 - \gd^2 \delta_2  }}
  \prod_{\ell=1}^{\corAB{\hat\gd_N(k)}} \left( |\hat{\lambda}_\ell (z,k)| \vee 1\right)^{2b_j}
  \le \det(\mathfrak{W}_{k,j})
  \le C_2(R,\gd, {\bm f}) N^{{3\delta_1\gd}}
  \prod_{\ell=1}^{\corAB{\hat\gd_N(k)}} \left(|\hat{\lambda}_\ell(z,k) | \vee 1\right)^{2b_j},
\]
for all large $N$, uniformly over $k \in [L_0]$ and $j' \in [L_k']$, where $C_2(R,\gd, {\bm f})$ is some positive finite constant depending only on $\gd, R$, and $\{f_\ell(\cdot)\}_{\ell=0}^\gd$.
\end{lemma}
\begin{proof}
Throughout the proof, for ease of writing, we write $\hat\gd_N$ and $\{\hat{\lambda}_\ell\}_{\ell=1}^{\hat\gd_N}$ instead of $\hat\gd_N(k)$ and $\{\hat{\lambda}_\ell(z,k)\}_{\ell=1}^{\hat\gd_N(k)}$.

We first derive the upper bound. Using Hadamard's inequality we observe that
 \begin{equation}\label{eq:hadamard-ubd-1}
 \det(\mathfrak{W}_{k,j}) \le \prod_{\ell=1}^{\corAB{\hat\gd_N}}(\mathfrak{W}_{k,j})_{\ell,\ell} = \prod_{\ell=1}^{\corAB{\hat\gd_N}} \|{\bm w}_\ell^j(k)\|_2^2.
 \end{equation}
  We note that $\|{\bm w}_{\ell}^j(k)\|_2^2= \sum_{m=0}^{b_j-1} |\hat{\lambda}_\ell|^{2m}$, where we recall that $b_j = i_{j+1}-i_j \le 2 N^{\delta_3}$. Therefore, using the fact that $z \notin \cN_N$, which in turn implies that $|\hat{\lambda}_\ell| \ne 1$, we obtain
 \begin{equation}\label{eq:hadamard-ubd-2}
 \|{\bm w}_{\ell}^j(k)\|_2^2 \le  |1-|\hat{\lambda}_\ell|^2|^{-1} \cdot \left[ 2 |\hat{\lambda}_\ell|^{2b_j} {\bf 1}(|\hat{\lambda}_\ell| > 1) + 2 {\bf 1}(|\hat{\lambda}_\ell| < 1)\right].
 \end{equation}
 Since $z \notin \cN_N$, the desired upper bound follows from \eqref{eq:hadamard-ubd-1}-\eqref{eq:hadamard-ubd-2}.


\corEE{For the lower bound, we apply \corOZ{the} \corEE{Cauchy--Binet} \corOZ{formula}, which gives
  \begin{equation}
    \label{eq:cbin}
\det(\mathfrak{W}_{k,j})
= \sum_{\substack{S \subset [N] \\ |S| = \hat\gd_N}}
\corABa{|\det(W_{k,j}[S])|^2,}
\end{equation}
where $W_{k,j}[S]$ is the $\hat\gd_N \times \hat\gd_N$ square submatrix of $W_{k,j}$ with rows \corABa{indexed by} $S.$  Hence for a lower bound, we may pick any $S$ and bound
\(
\det(\mathfrak{W}_{k,j})
\geq
\corABa{|\det(W_{k,j}[S])|^2}.
\)
Take
\[
\corABa{ S = ([\hat\gd_N - \gd_0] + i_j) \cup (i_{j+1} -\gd_0 + [\gd_0]).}
\]
Then we can write
\corABa{
\[
  W_{k,j}[S] :=
  \begin{bmatrix}
    V_1 & V_2 \\
    V_3 & V_4
  \end{bmatrix},
  \text{where}
  \quad\quad
  \begin{aligned}
    (V_1)_{i,\ell} &= \hat{\lambda}_\ell^{i-1},
    &i \in [\hat\gd_N - \gd_0],\ell \in [\gd_0], \\
    (V_2)_{i,\ell} &= \hat{\lambda}_{\ell+\gd_0}^{i-1},
    &i \in [\hat\gd_N - \gd_0], \ell \in [\hat\gd_N - \gd_0],  \\
    (V_3)_{i,\ell} &= \hat{\lambda}_{\ell}^{i+b_j-1-\gd_0},
    &i \in [\gd_0], \ell \in [\gd_0],  \\
    (V_4)_{i,\ell} &= \hat{\lambda}_{\ell+\gd_0}^{i+b_j-1-\gd_0},
    &i \in [\gd_0], \ell \in [\hat\gd_N - \gd_0].  \\
  \end{aligned}
\]}
In the cases that either $\gd_0=0$ or $\gd_0 = \hat\gd_N,$ we need only compute the determinant of $V_2$ or $V_3$ respectively.
Otherwise, by the Schur--complement formula,
\begin{equation}
  \label{eq:wschur}
 \corABa{ |\det( W_{k,j}[S])| = \left|\det \left(\begin{bmatrix} V_3 & V_4 \\ V_1 & V_2\end{bmatrix}\right)\right|
  =|\det(V_3)\det(V_2)\det(\Id - V_2^{-1}V_1 V_3^{-1} V_4)|.}
\end{equation}
Observe $\corABa{V_2}$ is a Vandermonde matrix.
\corABa{Writing $V_3 = \tilde V_3 \cdot D$}, where
\corABa{\(
D := \diag(
\hat\lambda_{1}^{b_j-\gd_0},
\cdots,
\hat\lambda_{\gd_0}^{b_j-\gd_0}
),
\)
\corABa{we see} that $\tilde V_3$ is \corOZ{also}  a Vandermonde matrix}.}

\corEE{
As $z \not\in \cN_N,$ we can bound the discriminant of $\left\{ \hat{\lambda}_\ell \right\}$ from below, and as we can bound $|\hat{\lambda}_\ell| \leq \|T_k(z)\|,$ we have that there is some \corABa{$C(R,\gd, {\bm f})$} so that
\begin{equation}
  \label{eq:v1v4}
  \corABa{|\det(V_2)^2\det(\tilde V_3)^2|}
  \geq
  (2\|T_k(z)\| \vee 1)^{-\gd(\gd-1)}
  \prod_{1 \le \ell < \ell' \le \hat\gd_N}
  \hspace{-1em}
  |\hat{\lambda}_\ell(z) - \hat{\lambda}_{\ell'}(z)|^2
  \geq C(R,\gd,{\bm f})^{\corABa{-1}}N^{-2\delta_1 \gd - \delta_2 \gd^2}
\end{equation}
Hence,
\begin{equation}
  \label{eq:v1v4a}
  \corABa{|\det(V_2)^2\det(V_3)^2|} \geq C(R,\gd,{\bm f})^{\corABa{-1}}N^{-2\delta_1 \gd- \delta_2 \gd^2}
  \cdot \prod_{\ell=1}^{\corAB{\hat\gd_N}} \left( |\hat{\lambda}_\ell (z,k)| \vee 1\right)^{2b_j}.
\end{equation}
\corABa{Note that the desired lower bound follows from \eqref{eq:cbin}, \eqref{eq:wschur} and \eqref{eq:v1v4a} once we show that
\begin{equation}\label{eq:V-1234-1}
\|V_2^{-1} V_1 V_3^{-1} V_4 \| \le 1/2.
\end{equation}}
}

\corABa{To this end, first let us recall the following standard inequality:
\[
\| \mathsf{M}\| \le \| \mathsf{M}\|_2 = \sqrt{\sum_{i,j} \mathsf{M}_{i,j}^2} \le \max\{n_1,n_2\}\cdot \max_{i,j} |\mathsf{M}_{i,j}|
\]
for any matrix $ \mathsf{M}$ of dimension $n_1 \times n_2$. As $|\lambda_\ell| < 1$ for $\ell \ge \gd_0+1$ and $|\lambda_\ell| \le \|T_k(z)\|$ for $\ell \le \gd_0$, we note that
\[
\|\tilde V_3\|,\|V_1\|, \le \gd \| T_k(z)\|^\gd \le \gd N^{\delta_2 \gd} \quad \text{ and } \quad \| V_2 \| , \|V_4\| \le \gd.
\]
}
\corABa{We can then trivially bound
\[
\|V_2^{-1}\|^{-1} =  \sigma_{\text{min}}(V_2) \geq	|\det(V_2)|/\|V_2\|^{\gd_0 - 1}
  \quad\text{and}\quad
  \|\tilde V_3^{-1}\|^{-1}=\sigma_{\text{min}}(\tilde V_3) \geq	|\det(\tilde V_3)|/\|\tilde V_3\|^{\hat\gd_N - \gd_0 - 1}.
\]
Hence, \corABa{using \eqref{eq:v1v4a}}
\begin{equation}\label{eq:V-1234}
  \|V_2^{-1}V_1 V_3^{-1} V_4\| \leq
  \frac{\gd^{4 \gd} N^{2 \gd^2\delta_2}}
  {|\det(V_2)|\cdot|\det(\tilde V_3)|} \|D^{-1}\| \le C(R,\gd,{\bm f}) \gd^{4 \gd} N^{3\gd^2\delta_2+ \gd \delta_1} \cdot \|D^{-1}\|.
\end{equation}
Since $z \notin \mathcal{B}_N$, and $4 \delta_1 <\delta_3$, the entries of $D$ are bounded below by
\[(1+N^{-3\delta_1})^{b_j - \hat\gd_N} \ge \exp\left(\frac12 N^{-3\delta_1} (b_j -\hat\gd_N)\right) \ge \exp \left(\frac18 N^{-3\delta_1 +\delta_3}\right) \ge \exp (N^{\delta_1}/8),
\]
for all large $N$. Therefore, $\|D^{-1}\| \le \exp(-N^{\delta_1}/8)$ and hence, in particular, it is smaller than any power of $N$. Thus, from \eqref{eq:V-1234} we establish  \eqref{eq:V-1234-1}. This completes the proof of the lemma.
}
 \end{proof}

 Building on Lemma \ref{lem:det-W-bd} we now derive bounds on $\det(\mathfrak{U}_{k,j}^* \mathfrak{U}_{k,j})$ where we recall that $\mathfrak{U}_{k,j}:=\pi_{k,j} \hat{M}_N(z) U_{k,j}$, and the columns of $U_{k,j}$ \corEE{are} an orthonormal basis of $\cS_{k,j}$.

\begin{lemma}\label{lem:sing-prod-ubd}
Fix $R<\infty$ and $z \in B_\C(0,R) \setminus \cN_N$. Let $k \in [L_0]$ and $j' \in [L_k']$ such that $\gb_k^{(j')}= [i_{j+1}] \setminus [i_j]$ for some $j \in [L] \cup\{0\}$. Assume $\hat\gd_N(k) >0$. Then there exists a constant $C_3(R,\gd, {\bm f})>1$, depending only on $R, \gd$, and $\{f_\ell(\cdot)\}_{\ell=0}^\gd$, such that 
\begin{align}\label{eq:det-lbd}
C_3(R,\gd, {\bm f})^{-1}
N^{\corEE{-5\gd \delta_1 -2\gd \delta_2}}
\prod_{\ell=1}^{\hat\gd_N(k)} \left(|\hat{\lambda}_\ell(z,k)| \wedge 1\right)^{2b_j}
& \,
\le \det(\corAB{\mathfrak{U}_{k,j}^* \mathfrak{U}_{k,j}}) \notag\\
& \, \le C_3(R,\gd, {\bm f})
N^{\corEE{2\gd^2 \delta_2+2\gd \delta_1}}
\prod_{\ell=1}^{\hat\gd_N(k)} \left(|\hat{\lambda}_\ell(z,k)|\wedge 1 \right)^{2b_j}.
\end{align}
\end{lemma}

\begin{proof}
Since $\{{\bm w}_\ell^j(k)\}_{\ell=1}^{\hat\gd_N}$ span the subspace ${\mathcal S}_{k,j}$, there exists a $\hat\gd_N \times \hat\gd_N$ matrix $\Gamma$ such that $U_{k,j}=W_{k,j} \Gamma$. The orthonormality of the columns of $U_{k,j}$ implies $\Gamma^* \mathfrak{W}_{k,j} \Gamma = \corAB{\Id}$. This in particular implies that $\Gamma \Gamma^* = (\mathfrak{W}_{k,j})^{-1}$. Thus
\begin{equation}\label{eq:det-U-to-W}
\det(\mathfrak{U}^*_{k,j}  \mathfrak{U}_{k,j}) = \det(\Gamma^* W_{k,j}^* \hat{M}_N(z)^* \pi_{k,j}^* \pi_{k,j}\hat{M}_N(z) W_{k,j} \Gamma ) = \frac{\det(W_{k,j}^* \hat{M}_N(z)^* \pi_{k,j}^* \pi_{k,j} \hat{M}_N(z) W_{k,j})}{\det(\mathfrak{W}_{k,j})}.
\end{equation}
The bound on the denominator of the \abbr{RHS} of  \eqref{eq:det-U-to-W} follows from Lemma \ref{lem:det-W-bd}.
To evaluate the numerator  we recall from Lemma \ref{eq:small-sing-vector}(ii) that $\rho_{k,j} \hat{M}_N(z)W_{k,j}={\bm 0}_{(i_{j+1} -\hat\gd_N) \times \hat\gd_N}$, where ${\bm 0}_{n_1\times n_2}$ is the matrix of zeros of dimension $n_1 \times n_2$. So, we only need to evaluate the next $\hat\gd_N$ rows of $\hat{M}_N(z)W_{k,j}$.

To this end, we note that for any $m=i_{j+1}-\hat\gd_N+1,\ldots,i_{j+1}$, and $\ell \in [\hat\gd_N]$, we have
\begin{align*}
(\hat{M}_{N}(z) {\bm w}_\ell^j(k))_m &\, = (t_0(k)-z) ({\bm w}_\ell^j(k))_m + \sum_{m'=m+1}^{i_{j+1}} t_{m'-m}(k)({\bm w}_{\ell}^j(k))_{m'} \\
& \, = (t_0(k)-z) \hat{\lambda}_\ell^{m-i_j-1} + \sum_{m'=m+1}^{i_{j+1}} t_{m'-m}(k)\hat{\lambda}_{\ell}^{m'-i_j-1} \\
& \, = \hat{\lambda}_\ell^{m-i_j-1}\left( \hat{P}_{z,k}(\hat{\lambda}_\ell) - \sum_{m'=i_{j+1}+1}^{m+\hat\gd_N} t_{m'-m}(k)\hat{\lambda}_{\ell}^{m'-m}\right)\\
& \, = - \sum_{m'=i_{j+1}+1}^{m+\hat\gd_N} t_{m'-m}(k)\hat{\lambda}_{\ell}^{m'-i_j-1}= -\hat{\lambda}_\ell^{b_j} \sum_{m'=1}^{m+\hat\gd_N-i_{j+1}} t_{m'+i_{j+1}-m}(k)\hat{\lambda}_\ell^{m'-1},
\end{align*}
where the second last step follows from the fact that $\hat{P}_{z,k}(\hat{\lambda}_\ell)=0$. This implies that
\[
 \pi_{k,j}\hat{M}_{N}(z) {\bm w}_\ell^j(k)= \begin{pmatrix} {\bm 0}_{(i_{j+1}-\hat\gd_N) \times 1} \\ -\hat{\lambda}_\ell^{b_j} \Delta {\bm v}_\ell(k)\\ {\bm 0}_{(N-i_{j+1}) \times 1} \end{pmatrix}, \quad \text{ where } \quad \Delta:= \begin{bmatrix} 0 & 0 & \cdots & 0 & t_{\hat\gd_N}(k) \\ 0 & 0 & \cdots& t_{\hat\gd_N(k)} & t_{\hat\gd_N-1}(k)\\
\vdots & \iddots & \iddots & \iddots & \vdots \\ 0& t_{\hat\gd_N}(k)& \cdots&  t_3(k)& t_2(k)\\ t_{\hat\gd_N}(k)& t_{\hat\gd_N-1}(k)& \cdots&  t_2(k)& t_1(k)\end{bmatrix}.
 \]
 It further yields that
 \[
 \pi_{k,j}\hat{M}_N(z)W_{k,j}= \begin{bmatrix} {\bm 0}_{(i_{j+1}-\hat\gd_N) \times \hat\gd_N} \\ - \Delta {\bm V}(k) \Lambda^{b_j}\\ {\bm 0}_{(N-i_{j+1}) \times \hat\gd_N} \end{bmatrix},
 \]
 where $\Lambda$ is a diagonal matrix with entries $\{\hat{\lambda}_\ell\}_{\ell=1}^{\hat\gd_N}$, and recall ${\bm V}(k)$ is the $\hat\gd_N \times \hat\gd_N$ matrix whose columns are $\{{\bm v}_\ell(k) \}_{\ell=1}^{\hat\gd_N}$. Thus
 \begin{align}\label{eq:det-U-to-W-num}
 \det(W_{k,j}^* \hat{M}_N(z)^* \pi_{k,j}^*\pi_{k,j} \hat{M}_N(z) W_{k,j}) & \, = \det\left((\Lambda^*)^{b_j} {\bm V}(k)^* \Delta^* \Delta {\bm V}(k) \Lambda^{b_j}\right) \notag\\
 & \, = \prod_{\ell=1}^{\hat\gd_N} |\hat{\lambda}_\ell|^{2b_j} \cdot \det({\bm V}(k) {\bm V}(k)^*) \cdot \det(\Delta^* \Delta).
 \end{align}
 \corEE{Using \eqref{eq:Tkz} and that $z\notin \cN_N$ (\corABa{cf.}\ Lemma \ref{lem:bad-z}) respectively,
 \[
\begin{aligned}
 \det({\bm V}(k) {\bm V}(k)^*) = &|\det({\bm V}(k))|^2  \le \| T_k(z)\|^{\hat\gd_N(\hat\gd_N-1)} = O\left(N^{\gd^2 \delta_2}\right),  \\
 &|\det({\bm V}(k))|^2 \ge N^{-2\delta_1 \gd} \cdot \left(\sup_{x \in [0,1]} |f_{\hat\gd_N}(x)|^{\hat\gd_N-1}\right)^{-1}.
 \end{aligned}
 \]
 As for $\Delta,$ $\det(\Delta) = t_{\hat\gd_N}(k)^{\hat\gd_N},$ and so
 \[
N^{-2\gd\delta_2} \le  \det(\Delta^*\Delta)\le \sup_{x \in [0,1]}|f_{\hat\gd_N}(x)|^{2\hat\gd_N}.
 \]}
Now the desired bound on $\det(\mathfrak{U}_{k,j}^*\mathfrak{U}_{k,j})$
follows from \eqref{eq:det-U-to-W}-\eqref{eq:det-U-to-W-num}, upon an application of Lemma \ref{lem:det-W-bd}.

\end{proof}

Building on Lemma \ref{lem:sing-prod-ubd} we now derive the upper bound on the product of small singular values of $\hat{M}_N(z)$. Before proceeding to the statement of the relevant result let us remind the reader that we chose a partition of $\{\gb_k\}_{k=0}^{L_0}$ of $[N]$ such that for $k \in [L_0]\cup \{0\}$, $\gb_k:=\{i \in [N]: \lfloor i N^{\delta_1-1}\rfloor=k\}$. We also noted that $N^{1-\delta_1}/2 \le |\gb_k| \le 2 N^{1-\delta_1}$ for all $k \in [L_0] \cup\{0\}$. We then considered a refinement $\{\{\gb_k^{(j')}\}_{j'=1}^{L_k'}\}_{k=0}^{L_0}$ of $\{\gb_k\}_{k=0}^{L_0}$ where $N^{\delta_3}/2 \le |\gb_k^{(j')}| \le 2N^{\delta_3}$ for all $k \in [L_0]\cup \{0\}$ and $j' \in [L_k']$. Finally recall that $0 = i_1 < i_2 < i_3 < \dots  < i_{L+1} = N$, with $L :=\sum_{k=0}^{L_0} L_k'$, are the endpoints of the partition $\{\{\gb_k^{(j')}\}_{j'=1}^{L_k'}\}_{k=0}^{L_0}$, and $b_j:=i_{j+1}-i_j$. Therefore fixing $k \in [L_0] \cup\{0\}$, and $j' \in [L_k']$ fixes $j \in [L]\cup\{0\}$ such that $\gb_{k}^{(j')} = [i_{j+1}]\setminus [i_j]$.

\begin{corollary}\label{cor:sing-prod-ubd}
Fix $R<\infty$ and $z \in B_\C(0,R) \setminus \cN_N$. Recall $\mathfrak{L}:=\sum_{k=0}^{L_0} L_k' \cdot \hat\gd_N(k)$ and $L=\sum_{k=0}^{L_0}L_k'$. Then
\[
  \prod_{m=0}^{\gL-1} \sigma_{N-k}(\hat{M}_N(z)) \leq
  {C_3(R,\gd, {\bm f})}^{L} N^{(\corEE{2\gd^2 \delta_2+2\gd \delta_1})L} \prod_{k=0}^{L_0} \prod_{\ell=1}^{\hat\gd_N(k)} \left(|\hat{\lambda}_\ell(z,k)| \wedge 1\right)^{|\gb_k|},
\]
for all large $N$, where $C_3(R, \gd, {\bm f})$ is as in Lemma \ref{lem:sing-prod-ubd}. If, for some $k$, $\hat\gd_N(k)=0$, then the innermost product becomes empty \corAB{which, by convention, is set to} \corOZ{equal $1$.}
\end{corollary}

\begin{proof}
Fix $k \in [L_0] \cup\{0\}$, and $j' \in [L_k']$ such that $\gb_k^{(j')}:=[i_{j+1}]\setminus [i_j]$. Let $U_{k,j}$ be the $N \times \hat\gd_N(k)$ matrix whose columns form an orthonormal basis of $\cS_{k,j}$. Denote
\begin{equation}\label{eq:dfn-onb}
U:= \begin{bmatrix} U_{0,0} & \cdots & U_{0,L_0'-1} & U_{1, L_0'} &\cdots & U_{1,L_0'+L_1'-1}&  \cdots & U_{L_0, L} \end{bmatrix}.
\end{equation}
Note that if $\hat\gd_N(k) =0$ for some $k$, then $U_{k,j}$ is an empty matrix. Therefore, it is equivalent to ignore such $k$'s while constructing the matrix $U$. We will show that
\begin{equation}\label{eq:prod-sing-ubd-det}
\left(\det(U^*\hat{M}_N(z)^* \hat{M}_N(z) U)\right)^{1/2} \le   \left\{C_3(R,\gd, {\bm f}) N^{\corEE{2\gd^2 \delta_2+2\gd \delta_1}}\right\}^L \cdot \prod_{k=0}^{L_0} \prod_{\ell=1}^{\hat\gd_N(k)} \left(|\hat{\lambda}_\ell(z,k)| \wedge 1\right)^{|\gb_k|}.
\end{equation}
Since, the columns of $U$ are orthonormal, this, together with Lemma \ref{lem:prod_sing}, yields the desired upper bound on the product of small singular values.

Turning to prove \eqref{eq:prod-sing-ubd-det} we note the following: For any $k \in [L_0]\cup\{0\}$ and $j' \in [L_k']$ such that $\gb_k^{(j')}:=[i_{j+1}]\setminus [i_j]$ and $\hat\gd_N(k)>0$, the columns of $\hat{M}_N(z) U_{k,j}$ belong to the subspace
$$\mathcal{T}_{k,j}:=\Span\left(\{e_{m}\}_{m=i_{j+1}-\hat\gd_N(k)+1}^{i_{j+1}}, \, \{e_{m}\}_{m=i_{j}-\hat\gd_N(k^-)+1}^{i_{j}}\right),$$
where (a) $k^-:=k$ if $j'>1$, and (b) $k^-:=k-1$ if $j'=1$, \corAB{and we set} $\hat\gd_N(-1):=0$.

This, in particular, implies that
\begin{equation}\label{eq:detU-1}
\det(U^*\hat{M}_N(z)^* \hat{M}_N(z) U)= \det(\mathfrak{U}^*_{\mathcal{T}} \mathfrak{U}_{\mathcal{T}}),
\end{equation}
where $\mathfrak{U}_{\mathcal{T}}$ is matrix obtained from $\pi_{\mathcal{T}} \hat{M}_N(z) U=\hat{M}_N(z) U$ by deleting its zero rows, and $\pi_{\mathcal{T}}$ is the orthogonal projection onto $\Span(\cup_{\hat\gd_N(k)>0} \mathcal{T}_{k,j})$. For any $v \in \C^N$, let us denote $\hat{\pi}_{k,j} v$ to be the $b_j$-dimensional vector obtained from ${\pi}_{k,j} v$ by deleting its zero rows. Equipped with this notation, we also note that $\mathfrak{U}_{\mathcal{T}}$ is a $\gL \times \gL$ block upper triangular matrix with $\{\hat\pi_{k,j} \hat{M}_N(z) U_{k,j}\}$ as its diagonal blocks. This yields that
\begin{equation}\label{eq:detU-2}
\det(\mathfrak{U}_{\mathcal{T}}) = \prod \det(\hat{\pi}_{k,j} \hat{M}_N(z) U_{k,j}).
\end{equation}
Since
\begin{equation}
\det(U_{k,j}^*\hat{M}_N(z)^* \hat{\pi}_{k,j}^*\hat{\pi}_{k,j} \hat{M}_N(z) U_{k,j})= \det(U_{k,j}^*\hat{M}_N(z)^* {\pi}_{k,j}^*{\pi}_{k,j} \hat{M}_N(z) U_{k,j}), \notag
\end{equation}
combining \eqref{eq:detU-1}-\eqref{eq:detU-2}, and applying Lemma \ref{lem:sing-prod-ubd} we arrive at \eqref{eq:prod-sing-ubd-det}. This completes the proof of the lemma.
\end{proof}
It remains to find a matching lower bound on the product of the small singular values.
\corOZ{Recall the notation $L_0,L,\gL$, see \eqref{eq-L0},\eqref{eq-L},\eqref{eq-frakL}.}
\begin{lemma}  \label{lem:lb-prod-sing}
Fix $R<\infty$ and $z \in B_\C(0,R) \setminus \cN_N$. Then there exists a constant $C_4(R,\gd, {\bm f})$, depending only on $R, \gd$, and $\{f_{\ell}\}_{\ell=0}^\gd$, such that
  \[
    \prod_{m=0}^{\gL-1} \sigma_{N-m}(\hat{M}_N(z))
    \geq
    \left(C_4(R,\gd, {\bm f}) N^{7\gd^2 \delta_1+3\corEE{\gd}^3\delta_2+4\gd \delta_3+\corABrev{\frac{\alpha_0\gd \delta_1}{2(2\alpha_0-1)}}}\right)^{-L} \cdot \gL^{-\gL/2}
    \prod_{k=0}^{L_0} \prod_{\ell=1}^{\hat\gd_N(k)} \left(|\hat{\lambda}_\ell(z,k)| \wedge 1\right)^{|\gb_k|},
  \]
for all large $N$.
\end{lemma}

\corAB{The proof of Lemma \ref{lem:lb-prod-sing} is similar to that of Proposition \ref{prop:lb}. Hence, we provide only a brief outline below.}


\begin{proof}[Proof of Lemma \ref{lem:lb-prod-sing}]
Using Lemma \ref{lem:prod_sing} \corAB{again} we see that it is enough to find a uniform lower bound on
  \[
   \prod_{m=1}^{\gL}
    \| \hat{M}_{N}(z) w_m\|_2
  \]
  over all collections of orthonormal vectors $\left\{ w_m \right\}_{m=1}^\gL$. \corAB{Analogous to the proof of Proposition \ref{prop:lb} we bound each $\| \hat{M}_{N}(z) w_m\|_2$ in one of two ways}. If $1-\|{\psi} w_m\|^2_2 \geq \frac{1}{2\gL}$ then applying Lemma \ref{lem:orth-compl-large-norm} we \corAB{deduce}
  \begin{equation} \label{eq:lb0}
    \| \hat{M}_{N}(z) w_m\|_2
    \geq
    \|{\rho} \hat{M}_{N}(z) w_m\|_2
    \geq
    \frac{1}{\sqrt{2\gL}} C_1(R,\gd, {\bm f})^{-1}N^{-(2\gd \delta_1+\corEE{\gd}^2\delta_2+2\delta_3+\corABrev{\frac{\alpha_0\delta_1}{2(2\alpha_0-1)}})},
  \end{equation}
  \corAB{where we recall that ${\psi}$ is the orthogonal projection onto the subspace ${\cS}:=\Span(\cup \cS_{k,j})$}.

  Without loss of generality, assume  $w_m,\, m \in [p]$ satisfies $1-\|{\psi} w_m\|^2_2 < \frac{1}{2\gL}$. \corAB{\corEE{Proceeding} similarly as in the steps leading to \eqref{eq:lb-simplifyd1} we find}
 \begin{equation}\label{eq:lb-simplify}
\prod_{m=1}^p   \|\hat{M}_{N}(z) w_m\|_2 \ge \left(\frac{C_1(R,\gd, {\bm f})^{-1}N^{-(2\gd \delta_1+\corEE{\gd}^2\delta_2+2\delta_3+\corABrev{\frac{\alpha_0\delta_1}{2(2\alpha_0-1)}})}}{4(\|\hat{M}_{N}(z)\|\vee 1)}\right)^p \prod_{m=1}^p\| \hat{M}_{N}(z) {\psi} w_m\|_2.
 \end{equation}
 Let $Y_1$ be the matrix whose columns are $\{{\psi}w_m\}_{m=1}^p$. Since the columns of $U$ span the subspace $\cS$, there must exist \corEE{an} $\gL \times p$ matrix $A_1$ such that $Y_1=U A_1$. We extend the matrix $A_1$ to \corEE{an} $\gL \times \gL$ matrix $A$ so that the last $\gL-p$ columns of $A$ are orthonormal and are also orthogonal to the first $p$ columns of $A_1$.
 Set $Y:=U A$ and let us denote the columns of $Y$ to be ${\bm y}_m$, for $m \in [\gL]$.

 Turning to bound the \abbr{RHS} \corOZ{of}
 \eqref{eq:lb-simplify}, by Hadamard's inequality we now find that
\begin{equation}\label{eq:lb-simplify-1}
\prod_{m=1}^p\| \hat{M}_N(z) {\psi} w_m\|_2^2 \ge \frac{\det(Y^* \hat{M}_{N}(z)^* \hat{M}_{N}(z) Y)}{\prod_{m=p+1}^\gL \|\hat{M}_{N}(z){\bm y}_m\|_2^2}.
\end{equation}
We separately bound the numerator and the denominator of \eqref{eq:lb-simplify-1}. \corAB{An argument similar to the proof of \eqref{eq:lb-simplify-3d1} yields}
\begin{equation}\label{eq:lb-simplify-2}
\corAB{\|\hat{M}_{N}(z) {\bm y}_m\|_2 \le  \|\hat{M}_{N}(z)\|.}
\end{equation}
It remains to find a lower bound of the numerator of \eqref{eq:lb-simplify-1}. To obtain such a bound, we observe that
\begin{equation}\label{eq:lb-simplify-3}
\det(Y^* \hat{M}_{N}(z)^* \hat{M}_{N}(z) Y)= \det(U^* \hat{M}_N(z)^* \hat{M}_N(z) U) \det(A A^*).
\end{equation}
Proceeding similarly as in the proof of \eqref{eq:prod-sing-ubd-det}, and applying the lower bound derived in Lemma \ref{lem:sing-prod-ubd} we deduce
\begin{equation}\label{eq:prod-sing-lbd-det}
\det(U^*\hat{M}_N(z)^* \hat{M}_N(z) U) \ge   \left\{C_3(R,\gd, {\bm f}) N^{5\gd \delta_1+2\gd \delta_2}\right\}^{-L} \cdot \prod_{k=0}^{L_0} \prod_{\ell=1}^{\hat\gd_N(k)} \left(|\hat{\lambda}_\ell(z,k)| \wedge 1\right)^{2|\gb_k|}.
\end{equation}
Arguments analogous to \eqref{eq:detAA*} further show that \(
\det(A A^*)  \geq 2^{-\gL} \ge 2^{-\gd L}.
\)
Plugging this bound in \eqref{eq:lb-simplify-3}, and using \eqref{eq:prod-sing-lbd-det} we obtain
\[
\det(Y^* \hat{M}_{N}(z)^* \hat{M}_{N}(z) Y) \ge  \left\{C_3(R,\gd, {\bm f}) 2^\gd N^{5\gd \delta_1+2\gd \delta_2}\right\}^{-L} \cdot \prod_{k=0}^{L_0} \prod_{\ell=1}^{\hat\gd_N(k)} \left(|\hat{\lambda}_\ell(z,k)| \wedge 1\right)^{2|\gb_k|}.
\]
Therefore, from \eqref{eq:lb-simplify}-\eqref{eq:lb-simplify-2}, and using the fact that $p \le \gL \le \gd L$, we derive
\[
\prod_{m=1}^p   \|\hat{M}_{N}(z) w_m\|_2 \ge \left(\frac{C_1(R,\gd, {\bm f})C_3(R,\gd, {\bm f}) N^{5\gd \delta_1+2\corEE{\gd}^2\delta_2+2\delta_3+\corABrev{\frac{\alpha_0\delta_1}{2(2\alpha_0-1)}}}}{8(\|\hat{M}_{N}(z)\|\vee 1)}\right)^{-\corEE{\gd} L} \cdot \prod_{k=0}^{L_0} \prod_{\ell=1}^{\hat\gd_N(k)} \left(|\hat{\lambda}_\ell(z,k)| \wedge 1\right)^{|\gb_k|}.
\]
Since by \corOZ{the}
Gershgorin circle theorem, $\|\hat{M}_N(z)\| \le |z| + \sum_{\ell=0}^\gd \sup_{x \in [0,1]}|f_\ell(x)|$, using \eqref{eq:lb0} \corOZ{we complete}  the proof of the lower bound.
 \end{proof}


We are now ready to finish the proof of Theorem \ref{thm:finite-off-diag-pc}.

\begin{proof}[Proof of Theorem \ref{thm:finite-off-diag-pc}]
The tightness of
the sequence of random probability measures
$\{L_{\hat\model_N}\}_{N \in \N}$ in $\mathcal{P}(\R)$,
the set of all probability measures on $\R$ is immediate from the domination
by singular values, see the proof of Corollary \ref{cor:iid}.
Therefore, by Prokhorov's theorem $\{L_{\hat\model_N}\}_{N \in \N}$
admits subsequential limits. We need to show that all subsequential limits coincide
and are given by the deterministic probability measure $\mu_{\gd,{\bm f}}$.

Suppose on the contrary that there exists a subsequence $\{N_m\}$ such that the above does not hold, i.e.~the
limit along the subsequence is not $\mu_{\gd,{\bm f}}$.
We \corOZ{fix a further arbitrary subsequence $\{N_{m_n}\} \subset \{N_m\}$ with
$N_{m_n} \ge 2^n$ for all $n \in \N$ and prove for that subsequence that}
\begin{equation}\label{eq:subseq-limit}
L_{\hat\model_{N_{m_n}}} \Rightarrow \mu_{\gd, {\bm f}} \quad \text{ as } \quad n \to \infty, \quad \text{ in probability}.
\end{equation}
This will prove the theorem.
Turning to \corOZ{the proof of
\eqref{eq:subseq-limit}, we first apply  \cite[Theorem 2.8.3]{tao2012topics}
and deduce that}
it is enough to show that for Lebesgue a.e.~$z \in \C$,
\begin{equation}\label{eq:log-pot-conv-1}
\mathcal{L}_{\hat\model_{N_{m_n}}}(z) \to \mathcal{L}_{\mu_{\gd,{\bm f}}}(z), \quad \text{ as } \quad n \to \infty, \quad \text{ in probability},
\end{equation}
where $\mathcal{L}_{\hat\model_N}(\cdot)$ is the log-potential of the \abbr{ESD} of $\hat\model_N$. Since $\mu_{\gd,{\bm f}}$ is compactly supported,
one can check the proof of \cite[Theorem 2.8.3]{tao2012topics} to deduce that
\corOZ{in fact} it suffices to establish \eqref{eq:log-pot-conv-1} for Lebesgue a.e.~$z \in B_\C(0,R)$ for some large $R$.

We will show that given any $\vep >0$, there exists a set $\corOZ{\hat \cN_\vep} \subset \C$, depending on the subsequence $\{N_{m_n}\}$, with
$\mathrm{Leb}(\corOZ{\hat \cN_\vep})
\le \vep$, such that for all $z \in B_\C(0,R)\setminus \corOZ{\hat \cN_\vep}$,
the convergence in \eqref{eq:log-pot-conv-1} holds. Since $\vep>0$ is arbitrary, this will complete the proof of  \eqref{eq:log-pot-conv-1}.

\corOZ{Toward this end,}
define
$\corOZ{\hat \cN_\vep}:=\cup_{n \ge \corOZ{n_0(\vep)}} \cN_{N_{m_n}} \cup f_0([0,1])$ for some $n_0(\vep) \ge 1$,
\corOZ{where $\cN_N$ is as in Lemma \ref{lem:bad-z}}.
\corOZ{Since} $f_0(\cdot)$ is a $\alpha_0$-H\"{o}lder continuous function
with $\alpha_0 \ge 1/2$, a simple volumetric estimate shows that $\mathrm{Leb}(f_0([0,1]))=0$. Hence, using Lemma \ref{lem:bad-z} and the union bound we see that, given any $\vep >0$, there exists $n_0:=n_0(\vep)$ such
that $\mathrm{Leb}(\corOZ{\hat \cN_\vep}) \le \vep$.
With this choice of
the set
$\corOZ{\hat \cN_\vep}$, we now prove \eqref{eq:log-pot-conv-1}.

To this end, our goal is to apply Theorem \ref{thm:logdet}. We need to show that all the assumptions of Theorem \ref{thm:logdet} are satisfied. From Remark \ref{rmk:delta} we have that for any $z \in B_\C(0,R)\setminus
\corOZ{\hat \cN_\vep}$, $N^* =o(N/\log N)$ along the subsequence $\{N_{m_n}\}$. Therefore applying Theorem \ref{thm:logdet} we conclude that for $z \in B_\C(0,R)\setminus
\corOZ{\hat \cN_\vep}$,
\begin{equation}\label{eq:conv-prob}
\left| \frac{1}{N_{m_n}} \log |\det (\hat\model_{N_{m_n}}(z))| - \frac{1}{N_{m_n}}\log|\det B(\hat{M}_{N_{m_n}}(z))|\right| \to 0, \quad \text{ as } \quad n \to \infty, \quad \text{ in probability}.
\end{equation}
Thus it remains to find
\[
\lim_{n \to \infty} \frac{1}{N_{m_n}}\log \det \left|B(\hat{M}_{N_{m_n}}(z))\right|.
\]
To this end, we note that
\begin{equation}\label{eq:B_N-limit0}
|\det B(\hat{M}_{N}(z)) |= \frac{|\det \hat{M}_{N}(z)|}{\prod_{p=0}^{N^*} \sigma_{N-p}(\hat{M}_{N}(z))}= \frac{\prod_{k=0}^{L_0}|t_0(k)-z|^{|\gb_k|}}{\prod_{p=0}^{N^*} \sigma_{N-p}(\hat{M}_{N}(z))}.
\end{equation}
Since $\|\hat{M}_N(z)\| =O(1)$ for any $z \in B_\C(0,R)$, and
$\corOZ{N^*\leq \gL=O(N^{1-\delta_3})}$ \corOZ{by \eqref{eq-frakL},
using} the definition of $N^*$ we have that
\[
\lim_{n \to \infty} \frac{1}{N_{m_n}} \log \left( \prod_{p=N^*+1}^{\gL-1} \sigma_{N_{m_n}-p}(\hat{M}_{N_{m_n}}(z)) \right) =0.
\]
Hence, it is enough to find $\lim_{n \to \infty} \Upsilon_n(z)$ and show that it equals $\mathcal{L}_{\mu_{\gd, {\bm f}}}(z)$, where
\[
\Upsilon_n(z):= \frac{1}{N_{m_n}}\log \left( \sum_{k=0}^{L_0} |\gb_k| \log |t_0(k)-z| - \sum_{p=0}^{\gL-1} \log \sigma_{N_{m_n}-p}(\hat{M}_{N_{m_n}}(z)) \right).
\]
\corOZ{Since} $\gL  =O(N^{\corOZ{1-\delta_3}})$,
\corOZ{we have,}
using Lemma \ref{lem:lb-prod-sing},
\corOZ{that} for any $z \in B_\C(0,R)\setminus \corOZ{\hat \cN_\vep}$,
\begin{align}\label{eq:limsup-ubd}
\limsup_{n \to \infty} \Upsilon_n(z)& \le \limsup_{n \to \infty} \frac{1}{N_{m_n}}\log \left( \sum_{k=0}^{L_0} |\gb_k|\left\{ \log |t_0(k)-z| -  \sum_{\ell=1}^{\hat\gd_{N_{m_n}}(k)} \log \left(|\hat{\lambda}_\ell(z,k)| \wedge 1\right)\right\} \right)\notag\\
& = \limsup_{n \to \infty} \frac{1}{N_{m_n}}\log \left( \sum_{k=0}^{L_0} |\gb_k|\left\{  \sum_{\ell=1}^{\hat\gd_{N_{m_n}}(k)} \log \left(|\hat{\lambda}_\ell(z,k)| \vee 1\right) + \log |t_{\hat\gd_{N_{m_n}}(k)}(k)|\right\} \right),
\end{align}
where the last step follows from the fact that $\{\hat{\lambda}_{\ell}(z,k)\}_{\ell=1}^{\hat{\gd}_N(k)}$ are the eigenvalues of the matrix $T_k(z)$, and hence
\begin{equation}\label{eq:small-big-connect}
\prod_{\ell=1}^{\hat\gd_{N_{m_n}}(k)} |\hat{\lambda}_\ell(z,k)| = |\det (T_k(z))| = \frac{|z-t_0(k)|}{|t_{\hat\gd_{N_{m_n}}(k)}(k)|}.
\end{equation}
We claim that for any $z \in B_\C(0,R)\setminus \corOZ{\hat \cN_\vep}$,
\begin{equation}\label{eq:claim-bct}
\sup_n \sup_{k\in[L_0]\cup \{0\}}  \left|  \sum_{\ell=1}^{\hat\gd_{N_{m_n}}(k)} \log \left(|\hat{\lambda}_\ell(z,k)| \vee 1\right) + \log |t_{\hat\gd_{N_{m_n}}(k)}(k)|\right| \le C({\bm f}, \gd, R) <\infty,
\end{equation}
for some constant $C({\bm f}, \gd, R)$ depending only on $\{f_{\ell}\}_{\ell=0}^\gd$, $\gd$, and $R$. Indeed, noting that the closed set $f_0([0,1]) \in \cN_\vep$, we have $\eta_0:=\mathrm{dist}(z, f_0([0,1]) >0$. Upon using the triangle inequality we therefore conclude that there exists $\eta >0$, such that for every $k \in [L_0]\cup\{0\}$, every root of the polynomial equation $\hat{P}_{z,k}(\lambda)=0$ is greater than $\eta$ in absolute value. Thus,
\[
\sup_{k \in [L_0]\cup\{0\}} \left| \sum_{\ell=1}^{\hat\gd_{N_{m_n}}(k)} \log \left(|\hat{\lambda}_\ell(z,k)|\,  \corAB{\wedge} \,1\right)\right| \le \gd |\log (\eta)|,
\]
for all $n$. Since $ \eta_0 \le |z -t_0(k)| \le R + \sup_{x \in [0.1]} |f_0(x)|$, for all $k \in [L_0]\cup\{0\}$, using \eqref{eq:small-big-connect} again the claim in \eqref{eq:claim-bct} follows.

Next fix $x \in [0,1]$. It is easy to check that for any $x \in (0,1)$, $\hat{P}_{z,k}(\lambda) \to P_{z,x}(\lambda)$ where $k = \lfloor x N^{\delta_1} \rfloor$. Since the roots of a polynomial are continuous function of its coefficients we have that
\[
\sum_{\ell=1}^{\hat\gd_{N_{m_n}}(k)} \log \left(|\hat{\lambda}_\ell(z,k)| \vee 1\right) + \log |t_{\hat\gd_{N_{m_n}}(k)}(k)| \to \sum_{\ell=1}^{\hat\gd(x)} \log \left(|{\lambda}_\ell(z,x)| \vee 1\right) + \log |f_{\hat\gd(x)}(x)|,
\]
as $n \to \infty$, where $k = \lfloor x N^{\delta_1} \rfloor$. Therefore, using \eqref{eq:claim-bct} and the bounded convergence theorem, from \eqref{eq:limsup-ubd} we deduce that
\[
\limsup_{n \to \infty} \Upsilon_n(z) \le \int_0^1 \left\{\sum_{\ell=1}^{\hat\gd(x)} \log \left(|{\lambda}_\ell(z,x)| \vee 1\right) + \log |f_{\hat\gd(x)}(x)|\right\} dx = \mathcal{L}_{\mu_{\gd, {\bm f}}}(z).
\]
Now applying Corollary \ref{cor:sing-prod-ubd} and using a similar reasoning as above, it also follows that for any $z \in B_\C(0,R)\setminus
\corOZ{\hat \cN_\vep}$, we have $\liminf_{n \to \infty} \Upsilon_n(z) \ge  \mathcal{L}_{\mu_{\gd, {\bm f}}}(z)$. This together with \eqref{eq:conv-prob} shows that, for any $\vep>0$, the convergence in \eqref{eq:log-pot-conv-1} holds for all $z$ outside a set of Lebesgue measure at most $\vep$. Hence, the proof of the theorem is now complete.
\end{proof}

\appendix
\section{Some algebraic facts}
\label{sec:prelim}
\corAB{In this section we collect a couple of standard matrix results which have been used in the proofs appearing in Sections \ref{sec:det-equiv}, \ref{sec-d=1}, and \ref{sec:pf-pc-constant}.}

\corAB{The first result shows that the determinant of the sum of the two matrices can be expressed as a linear combination of products of the determinants of appropriate sub-matrices. The proof trivially follows from the definition of the determinant. \corABrev{For a proof we refer the reader to \cite{marcus}}.} \corOZ{We adopt the convention that} the determinant of the matrix of size zero is one. The following result essentially follows from the definition of the determinant of a matrix.
\begin{lemma}\label{lem:cauchy-binet}
\corAB{For an $N\times N$ matrix $A$, and $\row,\col\subseteq [N]$ we write $A[X,Y]$ for the sub-matrix of $A$ which consists of the rows in $\row$ and the columns in $\col$. Then for any two $N\times N$ matrices $A$ and $B$ we have}
\begin{equation}\label{eq:det_decomposition}
        \det(A+B) = \sum_{\substack{\row,\col \subset [N] \\ |\row|=|\col|}} (-1)^{\sgn(\sigma_\row)\sgn(\sigma_\col)} \det(A[\COMP{\row}, \COMP{\col}])\det(B[\row, \col]),
\end{equation}
where $\COMP{\row}:=[N]\setminus \row$, $\COMP{\col}:=[N] \setminus \col$ and $\sigma_Z$ for $Z\in\{\row,\col\}$ is the permutation on $[N]$ which places all the elements of $Z$ before all the elements of $\COMP{Z}$, but preserves the order of elements within the two sets.
\end{lemma}

\corAB{The next lemma} \corOZ{deals with the} characterization of products of singular values.
\corA{
\begin{lemma}\label{lem:prod_sing}
Let $A$ be a $N \times N$ matrix. Then for any $k \le N-1$, we have
\begin{equation}
  \prod_{k'=0}^k \sigma_{N-k'}(A)
  =\inf_{\xi_0,\xi_1,\ldots,\xi_k} \left(\det(\Xi_k^* A^* A \Xi_k)\right)^{1/2}= \inf_{\xi_0,\xi_1,\dots,\xi_k}
  \prod_{k'=0}^k
  \| A \xi_{k'}\|_2,
  \label{eq:prod_sing}
\end{equation}
where the infimums are taken over set of orthonormal vectors $\{\xi_0,\xi_1,\ldots,\xi_k\}$ and $\Xi_k$ is the matrix whose columns are $\{\xi_{k'}\}_{k'=\corABrev{0}}^k$.
\end{lemma}
}
\corAB{The equality \eqref{eq:prod_sing} can be thought of as a generalization of Courant-Fischer-Weyl min-max principle. The equality of the leftmost and the rightmost terms in \eqref{eq:prod_sing} follows from}
\cite[Page 200, Ex. 12]{horn2012matrix}. For completeness, we provide a proof.
\begin{proof}
  \corAB{Using Hadamard's} \corOZ{determinantal inequality} \corAB{we first observe that
\[
\inf_{\xi_0,\xi_1,\ldots,\xi_k}\left(\det(\Xi_k^* A^* A \Xi_k)\right)^{1/2} \le   \inf_{\xi_0,\xi_1,\ldots,\xi_k}\prod_{k'=0}^k
  \| A \xi_{k'}\|_2.
\]
Next setting $\xi_0,\xi_1,\ldots,\xi_k$ to be the right singular vectors of $A$ corresponding to $\sigma_N(A), \sigma_{N-1}(A), \ldots, \sigma_{N-k}(A)$, respectively, we see that the product of the $\ell_2$-norms of $A\xi_{k'}$, for $k'=0,1,\ldots,k$, equals the product of $(k+1)$-st smallest singular values of $A$. Hence, we deduce
\[
\prod_{k'=0}^k \sigma_{N-k'}(A)  \ge \inf_{\xi_0,\xi_1,\ldots,\xi_k}\prod_{k'=0}^k  \| A \xi_{k'}\|_2.
\]}\corAB{
Therefore to prove \eqref{eq:prod_sing} it is enough to show that
\begin{equation}\label{eq:prod_sing0}
\prod_{k'=0}^k \sigma_{N-k'}(A) \le \inf_{\xi_0,\xi_1,\ldots,\xi_k}\left(\det(\Xi_k^* A^* A \Xi_k)\right)^{1/2}.
\end{equation}}

  Let $A = U\Sigma W$ be the singular value decomposition of $A$. 
  Thus
  \begin{equation}
    \label{eq:hmrd}
    \det( \Xi_k^* \corAB{A^* A} \Xi_k)
    =
    \det( \Xi_k^* W^* \Sigma^2 W \Xi_k).
  \end{equation}
  \corAB{Instead of} taking the infimum over all $\Xi_k$, \corAB{whose columns are othonormal}, we may change variables and take the infimum over all $W_k= W \Xi_k$, which is again a collection of $(k+1)$ orthonormal columns.
  Applying Cauchy-Binet \corAB{formula},
  \[
    \det(W_k^* \Sigma^2 W_k)
    =
    \sum_{\substack{S \subset [N] \\ |S| =k+1 }}
    |\det((\Sigma W_k)[S])|^2,
  \]
  where $(\Sigma W_k)[S]$ is the $(k+1)\times (k+1)$ matrix with rows in $S.$  \corAB{Since $\Sigma$ is a diagonal matrix we} observe that
  \[
    \det((\Sigma W_\ell)[S])
    =
    \det((W_\ell)[S])\cdot \prod_{i \in S} \sigma_i(A)
    \geq
    \det((W_\ell)[S]) \prod_{k'=0}^k \sigma_{N-k'}(A).
  \]
  Hence
  \[
    \det(W_k^* \Sigma^2 W_k)
    \geq
    \prod_{k'=0}^k \sigma_{N-k'}^2(A)
    \sum_{\substack{S \subset [N] \\ |S| =k+1 }}|\det((W_k)[S])|^2
    =
    \prod_{k'=0}^k \sigma_{N-k'}^2(A) \det(W_k^*W_k),
  \]
  where we have again applied \corAB{the} Cauchy-Binet \corAB{formula}. \corAB{Since the columns of $W_k$ are orthonormal, combining the above with \eqref{eq:hmrd}, the inequality \eqref{eq:prod_sing0} follows. This completes the proof.} 
\end{proof}

\end{document}